\documentclass[11pt,a4paper]{article}

\usepackage[left=2.45cm, top=2.45cm,bottom=2.45cm,right=2.45cm]{geometry}
\usepackage{mathtools,amssymb,amsthm,mathrsfs,calc,graphicx,xcolor,cleveref,dsfont,tikz,pgfplots,bm, extarrows}
\usepackage[british]{babel}
\usepackage{amsfonts}              
\usepackage[T1]{fontenc}
\numberwithin{equation}{section}

\newtheorem {theorem}{Theorem}[section]
\newtheorem {proposition}[theorem]{Proposition}
\newtheorem {lemma}[theorem]{Lemma}
\newtheorem {corollary}[theorem]{Corollary}

{\theoremstyle{definition}

}
{\theoremstyle{theorem}
\newtheorem {remark}[theorem]{Remark}
\newtheorem {example}[theorem]{Example}
}

\newcommand{\Cov}{\operatorname{Cov}}
\newcommand{\var}{\operatorname{var}}

\newcommand{\diam}{\operatorname{diam}}

\def\BB{\mathbb{B}}

\def\EE{\mathbb{E}}

\def\MM{\mathbb{M}}
\def\NN{\mathbb{N}}

\def\PP{\mathbb{P}}
\def\QQ{\mathbb{Q}}
\def\RR{\mathbb{R}}
\def\SS{\mathbb{S}}

\def\XX{\mathbb{X}}

\def\ZZ{\mathbb{Z}}

\def\sfD{{\sf D}}

\def\sfN{{\sf N}}

\def\bF{\mathbf{F}}

\def\bN{\mathbf{N}}
\def\bO{\mathbf{O}}

\def\fA{\mathfrak{A}}
\def\fB{\mathfrak{B}}

\def\fF{\mathfrak{F}}

\def\fN{\mathfrak{N}}

\def\cC{\mathcal{C}}

\def\cF{\mathcal{F}}

\def\cK{\mathcal{K}}

\def\cN{\mathcal{N}}

\def\cR{\mathcal{R}}

\def\sC{\mathscr{C}}

\def\sH{\mathscr{H}}

\def\sL{\mathscr{L}}

\def\dint{\textup{d}}

\def\SO{\textup{SO}}

\makeatletter
\let\@fnsymbol\@alph
\makeatother

\begin{document}

\title{\bfseries Variance asymptotics and central limit theory\\ for geometric functionals of Poisson cylinder processes}

\author{Carina Betken\footnotemark[1],\; Matthias Schulte\footnotemark[2]\;\; and Christoph Th\"ale\footnotemark[3]}

\date{}
\renewcommand{\thefootnote}{\fnsymbol{footnote}}
\footnotetext[1]{Ruhr University Bochum, Germany. Email: carina.betken@rub.de}

\footnotetext[2]{Hamburg University of Technology. Email: matthias.schulte@tuhh.de}

\footnotetext[3]{Ruhr University Bochum, Germany. Email: christoph.thaele@rub.de}

\maketitle

\begin{abstract}
\noindent  This paper deals with the union set of {a} stationary Poisson process of cylinders in $\mathbb{R}^n$ having an $(n-m)$-dimensional base and an $m$-dimensional direction space, where $m\in\{0,1,\ldots,n-1\}$ and $n\geq 2$. The concept {simultaneously} generalises those of a Boolean model and a Poisson hyperplane or $m$-flat process. Under very general conditions on the typical cylinder base a Berry-Esseen bound for the volume of the union set within a sequence of growing test sets is derived. Assuming convexity of the cylinder bases and {of} the window a similar result is shown for a broad class of geometric functionals, including the intrinsic volumes. In this context the asymptotic variance constant is analysed in detail, which in contrast to the Boolean model leads to a new degeneracy phenomenon. A quantitative central limit theory is developed in a multivariate set-up as well.
\end{abstract}

\tableofcontents

\section{Introduction}

The development of quantitative central limit theorems for spatial random structures has been one of the driving forces in stochastic geometry over the last years. Most of the recent breakthroughs were made possible due to the development of new technical devices that were perfectly adapted to geometry-driven applications. Most notable in this context is the Malliavin-Stein method for normal approximation of functionals of Poisson processes. {Originally introduced in \cite{PeccatiSoleTaqquUtzet} and further developed in many subsequent works} it has turned out to be a versatile device with a vast of potential applications. As concrete examples we mention the works {\cite{DecreusefondSchulteThaele,LachPecc,LachPeccYang,LachiezeSchulteYukich,LastPeccatiSchulte,ReitznerSchulte,ReitznerSchulteThaele,SchulteThaele}} on various models for geometric random graphs, the {paper \cite{DecreusefondFerrazEtAl}} dealing with geometric random simplicial complexes, the application to the classical Boolean model \cite{HugLastSchulteBM}, the works \cite{DecreusefondSchulteThaele,HeroldHugThaele,LPST,ReitznerSchulte} dealing with Poisson hyperplane tessellations in Euclidean and non-Euclidean spaces, the applications in \cite{LachiezeSchulteYukich,SchulteVoronoi} to Poisson-Voronoi tessellations, {the works on excursion sets of Poisson shot-noise processes \cite{LachiezeRey,LachPeccYang}} as well as the papers \cite{BesauThaele,BesauRosenThaele,LachiezeSchulteYukich,Thaele,ThaeleTurchiWespi,TurchiWespi} considering different models for random polytopes. For an illustrative overview on the Malliavin-Stein method for functionals of {Poisson processes} we refer to the collection of surveys in \cite{PeccatiReitzner}.

The present paper continues this line of research by developing a central limit theory for functionals of so-called Poisson cylinder processes. In this paper we understand by a cylinder in $\RR^n$ any set of the form {$X + E$,} where $E$ is an $m$-dimensional linear subspace of $\RR^n$, $X\subset E^\perp$ is a compact subset in the orthogonal complement of $E$ and $m\in\{0,1,\ldots,n-1\}$ is a fixed dimension parameter. We refer to $E$ as the direction space and to $X$ as the cylinder base. A Poisson cylinder process in $\RR^n$ (for fixed $m\in\{0,1,\ldots,n-1\}$) is a {Poisson process} on the space of cylinders in $\RR^n$ as just described, and our focus is on the induced random set that arises as the union of all cylinders. We shall focus on the case where this union set is a stationary random closed set in the usual sense of stochastic geometry (see \cite[Chapter 2]{SW}). In this situation a Poisson cylinder process is described by three parameters: {the} intensity parameter, the dimension parameter $m$ and the joint distribution of the direction and the base of the so-called typical cylinder in the sense of Palm distributions. One motivation to study Poisson cylinder processes arises from the observation that they interpolate between two classical models in stochastic geometry, the Boolean model (for $m=0$) and the Poisson hyperplane process (for $m=n-1$ and for degenerate cylinder bases) or, more generally, Poisson $m$-flat processes (again for degenerate cylinder bases). However, while the Boolean model is, in a sense, locally defined as all its grains are compact, Poisson hyperplane processes show long-range dependencies induced by the infinitely extended hyperplanes. A Poisson cylinder process inherits properties of both of these extreme cases and it is the purpose of this paper to contribute to a better understanding of the resulting geometric and probabilistic phenomena. Moreover, we would like to mention that Poisson cylinder processes have {found} concrete applications in material technology, for example, in mathematical models for gas diffusion layers, see \cite{SpiessSpodarev}. In this context, the results we develop in this paper can be useful in statistical inference, for example, by developing parametric hypothesis tests. {We remark that our results heavily depend on the fact that we have an underlying Poisson process. For some results on planar cylinder processes driven by other point processes we refer the reader to \cite{FlimmelHeinrich}.}

To be more concrete, we briefly describe the main contributions of this paper. We start by considering the volume of general stationary Poisson cylinder processes. In this case a central limit theorem together with a bound for the speed of convergence was already obtained in \cite{HeinrichSpiessCLTVolume}, but the Malliavin-Stein method allows us to significantly reduce  the necessary moment assumptions. While in \cite{HeinrichSpiessCLTVolume} it was assumed that the volume of the so-called typical cylinder base has finite exponential moments, we are able to deduce a rate of convergence of the same quality under a third or fourth moment condition (depending on the probability metric in which the speed is measured). This is {also in line} with the classical Berry-Esseen theorem for sums of independent random variables and answers a question raised in \cite{HeinrichSpiessCLTVolume,HeinrichSpiessCLTVundS}. In a next step, we {specialise} our set-up {by} requiring the cylinder bases to be convex. This implies that the intrinsic volumes and many more general geometric functionals related to a stationary Poisson cylinder process become well-defined random variables. For such functionals we develop a { comprehensive qualitative and quantitative central limit theory for the univariate and the multivariate case}, which parallels and extends the results for the classical Boolean model in \cite{HugLastSchulteBM}. However, in our situation we will uncover new effects which are not present for the Boolean model. We will see that due to the long-range correlations that are immanent to Poisson cylinder processes the speed of convergence in the central limit theorem gets slowed down with increasing dimension parameter $m$. Rather notable in this context is the asymptotic analysis of second-order quantities. While the asymptotic variance constant is known to be strictly positive for a very broad class of functionals of the Boolean model (including the celebrated intrinsic volumes, for example) and also for the volume of a Poisson cylinder process, we will see that this is not necessarily true for other geometric functionals such as the surface area or, more generally, the intrinsic volumes. More precisely, using the general Fock space representation of Poisson functionals we shall see that asymptotically for $m\in\{1,\ldots,n-1\}$, second-order quantities of Poisson cylinder processes behave like their (possibly vanishing) projections onto the first Wiener chaos whenever the distribution of the direction of the typical cylinder has no atoms. This unexpected and striking new effect is in sharp contrast to the results for the Boolean model, where the projections to all chaoses contribute to the asymptotic behaviour (see \cite{HugLastSchulteBM,LP}). Even more, we shall explain that this phenomenon breaks down for discrete direction distributions. In this situation again the projections to all chaoses play a non-trivial role. We shall also develop a criterion that ensures strict positivity of asymptotic variance constants of the intrinsic volumes of order $m$ to $n$. 
As a concrete example, we provide fully explicit formulas for the asymptotic covariance structure of the intrinsic volumes of order $n$ and $n-1$ by means of a reduction to the classical Boolean model in $\RR^{n-m}$ and some integral formulas from \cite{HugLastSchulteBM}.

\bigskip

The remaining parts of this paper are structured as follows. Before presenting our results in Section \ref{sec:MainResults} we gather some background material in Section \ref{sec:Background}. All proofs are contained in Sections \ref{proof:CLTVolume} -- \ref{sec:VarianceAsymptotics}.

\section{Background material}\label{sec:Background}

\subsection{Frequently used notation}

For $n\in\NN$ we denote by {$\RR^n$ the} $n$-dimensional Euclidean space which is supplied with the Euclidean norm $\|\,\cdot\,\|$. By $\diam(A):=\sup\{\|x-y\|:x,y\in A\}$ we indicate the diameter of a set $A\subset\RR^n$. {We} write $R(A)$ for the radius of the smallest ball containing $A$, $\partial A$ for the boundary and $A^\circ$ for the interior of $A$. Moreover, $ d(x,A) $ will denote the (Euclidean) distance of a point $ x \in \RR^n $ to the set $ A $. We let $\BB_r^k$ be the $k$-dimensional ball of radius $r>0$ centred at the origin and put $\BB^k:=\BB_1^k$. For $\varepsilon>0$ we let $A^{\varepsilon}:=A+\BB_{\varepsilon}^n=\{x\in\RR^n:d(x,A)\leq\varepsilon\}$ be the {$\varepsilon$-parallel set} of $A$, which consists of all points in $\RR^n$ that have distance at most $\varepsilon$ from $A$. We denote for $n\in\NN$ by $\sL^n$ the $n$-dimensional Lebesgue measure, {while $\sH^m$ stands for the $m$-dimensional Hausdorff measure for $m\in\mathbb{N}$.} Furthermore, we define for $k\in\NN$ the constant $\kappa_k:=\sL^k(\BB^k)=\pi^{k/2}/\Gamma(1+k/2)$.

{By $(\Omega,\fA,\PP)$ we denote our underlying probability space, which is implicitly assumed to be rich enough to carry all the random objects we consider.} 

\subsection{Random closed sets}

For a fixed space dimension $n\in\NN$ we let 
\begin{itemize}
	\item[-] $\cF(\RR^n)$ be the space of closed subsets of $\RR^n$,
	\item[-] $\cC(\RR^n)$ be the space of compact subsets of $\RR^n$, 
	\item[-] $\cK(\RR^n)$ be the space of compact convex subsets of $\RR^n$ and
	\item[-] $\cR(\RR^n)$ be the space of finite unions of compact convex subsets of $\RR^n$, the so-called convex ring.
\end{itemize}  
The \textbf{Fell topology} on $\cF(\RR^n)$ is generated by the families of sets $\{\{F\in\cF(\RR^n):F\cap G\neq\varnothing\} {:} G\subset\RR^n\text{ open}\}$ and $\{\{F\in\cF(\RR^n):F\cap C=\varnothing\} {:} C\subset\RR^n\text{ compact}\}$. The Borel $\sigma$-field generated by the Fell topology is denoted by $\fF(\RR^n)$ and $(\cF(\RR^n),\fF(\RR^n))$ is the measurable space of closed subsets of $\RR^n$. Accordingly, by a \textbf{random closed set} we understand an random element in $\cF(\RR^n)$, that is, a ($\fA$-$\fF(\RR^n)$)-measurable mapping $X:\Omega\to\cF(\RR^n)$, where $(\Omega,\fA,\PP)$ is our underlying probability space. We remark that $\cC(\RR^n)\in\fF(\RR^n)$, but also $\cR(\RR^n),\cK(\RR^n)\in\fF(\RR^n)$ according to \cite[Theorem 2.4.2]{SW}, which allows us to speak about \textbf{random compact} and \textbf{random convex sets} as well as of \textbf{random sets in the convex ring}. For further background material we refer the reader to the monograph \cite{SW}.

\subsection{Grassmannians, orthogonal groups and their invariant measures}

Let $n\geq 2$, $m\in\{0,1,\ldots,n\}$ and denote {by $G(n,m)$ the} \textbf{Grassmannian} of $m$-dimensional \textbf{linear} subspaces of $\RR^n$. Similarly, $A(n,m)$ stands for the \textbf{Grassmannian} of all $m$-dimensional \textbf{affine} subspaces of $\RR^n$. By $\nu_m$ we denote the unique rotation-invariant Haar probability measure on $G(n,m)$ and let
$$
\mu_m(\,\cdot\,) := \int_{G(n,m)}\int_{E^\perp}{\bf 1}\{E+x\in\,\cdot\,\}\, {\sL_{E^\perp}}(\dint x)\, \nu_m(\dint E),
$$
where {$\sL_{E^\perp}$ stands for the Lebesgue measure on $E^\perp$, the linear subspace orthogonal to $E$.} Following \cite{HeinrichSpiessCLTVolume,HeinrichSpiessCLTVundS} we identify a subspace $E\in G(n,m)$ with a unique element $O_E$ of the equivalence class $\bO_E=\{O\in\SO_n:E=O\EE^m\}$ of orthogonal matrices $O\in\SO_n$ satisfying $E=O\EE^m$, where $\EE^m={\rm span}(e_{n-m+1},\ldots,e_n)$ and $e_1,\ldots,e_n$ is the standard orthonormal basis in $\RR^n$. More precisely, one can choose for $O_E$ the lexicographically smallest element in $\bO_E$, which yields a one-to-one correspondence between $G(n,m)$ and the space
$$
\SO_{n,m}=\{O_E={\rm lex\,min}\,\bO_E:E\in G(n,m)\},
$$
up to orientation of the subspaces. We indicate by $\nu_{n,m}$ the unique $\SO_n$-invariant Haar probability measure on $\SO_{n,m}$, which may be derived from the invariant Haar probability measure on the quotient $\SO_n/{\rm S}({\rm O}_{n-m}\times{\rm O}_m)$ in which ${\rm O}_m$ and ${\rm O}_{n-m}$ stand for the orthogonal groups of $m\times m$ and $(n-m)\times(n-m)$ matrices, respectively. We refer to \cite{HeinrichSpiessCLTVolume,HeinrichSpiessCLTVundS} for further details.

\subsection{Univariate normal approximation of Poisson functionals}\label{subsec:SecondOrderPoincareUnivariate}

{Recall that $(\Omega,\fA,\PP)$ is our  underlying probability space.} Let $\XX$ be a Borel space with Borel $\sigma$-field $\fB(\XX)$ and $\mu$ be a $\sigma$-finite measure on $\XX$. We let $\sfN(\XX)$ be the space of $\sigma$-finite counting measures on $\XX$ supplied with the $\sigma$-field $\fN(\XX)$ generated by sets of the form $\{\zeta\in\sfN(\XX):\zeta(B)=k\},B\in\fB(\XX),k\in\{0,1,2,\ldots\}$. By a \textbf{point process} we understand a $(\fA-\fN(\XX))$-measurable mapping $\zeta:\Omega\to\sfN(\XX)$. A point process $\eta$ on $\XX$ is called a \textbf{Poisson process} with \textbf{intensity measure} $\mu$ provided that the following two conditions are satisfied:
\begin{itemize}
\item[(i)] for each $B\in\fB(\XX)$ the random variable $\eta(B)$ is Poisson distributed with parameter $\mu(B)$,
\item[(ii)] for each $n\in\NN$ and for each collection of disjoint subsets $B_1,\ldots,B_n\in\fB(\XX)$ the random variables $\eta(B_1),\ldots,\eta(B_n)$ are independent.
\end{itemize}
By a \textbf{Poisson functional} $F$ we understand any real-valued random variable $F$ satisfying $F=f(\eta)$ $\PP$-almost surely for some fixed $(\fN(\XX)-\fB(\RR))$-measurable function $f:\sfN(\XX)\to\RR$, see \cite{LP}. We call $f$ {a} representative of $F$. By $L_\eta^2$ we understand the space of Poisson functionals satisfying $\EE[F^2]<\infty$. 

For a Poisson functional $F$ with representative $f$ and $x\in\XX$ we define the \textbf{first-order difference operator} by
$$
\sfD_xF := f(\eta+\delta_x) - f(\eta);
$$
that is, $\sfD_xF$ measures the effect on $F$ when the point $x$ is added to the Poisson process.
By slight abuse of notation we will rewrite this and similar identities as
$$
\sfD_xF := F(\eta+\delta_x) - F(\eta),
$$
suppressing thereby the role of the representative of $F$.   Thus, {$\sfD F$ can be regarded as a bi-measurable mapping $\sfD F:\Omega\times\XX\to\RR$. We denote by $L^2(\PP\otimes\mu)$ the space of all bi-measurable mappings $g:\Omega\times\XX\to\RR$ that satisfy $\int_{\XX}\EE[g(x)^2]\,\mu(\dint x)<\infty$.} Similarly, we define for $x,y\in\XX$ the \textbf{second-order difference operator}
$$
\sfD_{x,y}^2F := \sfD_{x}\sfD_{y}F = \sfD_{y}\sfD_{x}F = F(\eta+\delta_x+\delta_y)-F(\eta+\delta_x)-F(\eta+\delta_y)+F(\eta).
$$
More generally, for $k\in\NN$ and $x_1,\ldots,x_k\in\XX$ we put $\sfD_{x_1,\ldots,x_k}^kF:=\sfD_{x_k}(\sfD_{x_1,\ldots,x_{k-1}}^{k-1}F)$ and observe that this definition is symmetric in $x_1,\ldots,x_k$. Using the notions of the first- and the second-order difference operator we can now define the following quantities associated with a Poisson functional $F$:
\begin{align*}
\alpha_{F,1}^2 &:= 4\int_\XX\int_\XX\int_\XX(\EE[(\sfD_{x_1}F)^2(\sfD_{x_2}F)^2])^{1/2}(\EE[(\sfD_{x_1,x_3}^2F)^2(\sfD_{x_2,x_3}^2F)^2])^{1/2}\,\mu(\dint x_1)\,\mu(\dint x_2)\,\mu(\dint x_3),\\
\alpha_{F,2}^2 &:=\int_\XX\int_\XX\int_\XX\EE[(\sfD_{x_1,x_3}^2F)^2(\sfD_{x_2,x_3}^2F)^2]\,\mu(\dint x_1)\,\mu(\dint x_2)\,\mu(\dint x_3),\\
\alpha_{F,3} &:= \int_{\XX}\EE[|\sfD_xF|^3]\,\mu(\dint x)\\
{\alpha_{F,3}'} & {:= \int_{\XX} \EE[|\sfD_xF|^3]^{1/3} \EE[\min(\sqrt{8}|\sfD_xF|^{3/2},|\sfD_xF|^3)]^{2/3} \,\mu(\dint x).}
\end{align*}
They can be used to bound the \textbf{Wasserstein distance} $d_W(F,N)$ between a Poisson functional $F$ and a standard Gaussian random variable $N\sim\cN(0,1)$, where we recall that
$$
d_W(F,N) := \sup\{|\EE[h(F)]-\EE[h(N)]|:h\in{\rm Lip_1}\},
$$
and where the supremum is taken over all Lipschitz functions $h:\RR\to\RR$ with Lipschitz constant less than or equal to $1$. The following result is taken from \cite[Theorem 1.1]{LastPeccatiSchulte} {and \cite[Theorem 3.1]{BPT},} see also \cite[Theorem 21.3]{LP}.

\begin{proposition}[Normal approximation of Poisson functionals, Wasserstein bound]\label{prop:CLTWasserstein}
Let $F\in L_\eta^2$ be such that $\sfD F\in L^2(\PP\otimes\mu)$, $\EE[F]=0$ and $\var(F)=1$. Then
$$
d_W(F,N) \leq \alpha_{F,1}+\alpha_{F,2}+\alpha_{F,3} \qquad {\text{and} \qquad d_W(F,N) \leq \alpha_{F,1}+\alpha_{F,2}+\alpha_{F,3}'},
$$
where $N\sim\cN(0,1)$ is a standard Gaussian random variable.
\end{proposition}

The \textbf{Kolmogorov distance} $d_K(F,N)$ between a Poisson functional $F$ and a standard Gaussian random variable $N\sim\cN(0,1)$, which is defined as
$$
d_k(F,N) := \sup\{|\PP(F\leq t)-\PP(N\leq t)|:t\in\RR\},
$$
can be treated in a similar way. To rephrase the corresponding {bound}, we need to define further terms $\alpha_{F,4}$, $\alpha_{F,5}$ and $\alpha_{F,6}$ as follows:
\begin{align*}
\alpha_{F,4} &:= {1\over 2}(\EE[F^4])^{1/4}\int_\XX (\EE[(\sfD_xF)^4])^{3/4}\,\mu(\dint x),\\
\alpha_{F,5}^2 &:= \int_\XX \EE[(\sfD_xF)^4]\,\mu(\dint x),\\
\alpha_{F,6}^2 &:= \int_{\XX}\int_{\XX}6(\EE[(\sfD_{x_1}F)^4])^{1/2}(\EE[(\sfD_{x_1,x_2}^2F)^4])^{1/2}+3\EE[(\sfD_{x_1,x_2}^2F)^4]\,\mu(\dint x_1)\,\mu(\dint x_2).
\end{align*}
The next result is Theorem 1.2 in \cite{LastPeccatiSchulte}.

\begin{proposition}[Normal approximation of Poisson functionals, Kolmogorov bound]\label{prop:CLTKolmogorov}
Let $F\in L_\eta^2$ be such that $\sfD F\in L^2(\PP\otimes\mu)$, $\EE[F]=0$ and $\var(F)=1$. Then
$$
d_W(F,N) \leq \alpha_{F,1}+\alpha_{F,2}+\alpha_{F,3}+\alpha_{F,4}+\alpha_{F,5}+\alpha_{F,6},
$$
where $N\sim\cN(0,1)$ is a standard Gaussian random variable.
\end{proposition}

We notice that the term $\alpha_{F,4}$ contains the fourth moment of the Poisson functional $F$. However, it was shown in \cite[Lemma 4.3]{LastPeccatiSchulte} that this quantity can be bounded in terms of the fourth moment of the first-order difference operator. In fact, for $F\in L_\eta^2$ satisfying $\EE[F]=0$ and $\var(F)=1$ one has that
\begin{align}\label{eq:4thMomentBound}
\EE[F^4] \leq \max\bigg\{256\bigg(\int_{\XX}(\EE[(\sfD_xF)^4])^{1/2}\,\mu(\dint x)\bigg)^2,4\int_{\XX}\EE[(\sfD_xF)^4]\,\mu(\dint x)+2\bigg\}.
\end{align}

\subsection{Multivariate normal approximation of Poisson functionals}\label{sec:PreparationsMultivariate}

As in the previous section we let $\XX$ be a Borel space with Borel $\sigma$-field $\fB(\XX)$ and $\eta$ be a Poisson process on $\XX$ with intensity measure $\mu$. In contrast to Proposition \ref{prop:CLTWasserstein} and Proposition \ref{prop:CLTKolmogorov} we are interested in this section in the multivariate normal approximation of a vector $\bF=(F_1,\ldots,F_d)$ of $d\in\NN$ Poisson functionals $F_1,\ldots,F_d$. In order to compare $\bF$ with a {centred} Gaussian random vector {$\bN_{\bF}$} having the same dimension and the same covariance matrix as $\bF$, we will work with the so-called $d_3$-metric. To define it, we denote by $\sC_d^3$ the space of thrice {continuously} differentiable functions $h:\RR^d\to\RR$ having the absolute values of their second and third partial derivatives bounded by one, i.e.
\begin{align*}
\sC_d^3 := \bigg\{h:\RR^d\to\RR:\max_{k,\ell=1,\ldots,d}\Big\|{\partial^2h\over\partial x_k\partial x_\ell}\Big\|_\infty\leq 1,\max_{k,\ell,p=1,\ldots,d}\Big\|{\partial^3h\over\partial x_k\partial x_\ell\partial x_p}\Big\|_\infty\leq 1\bigg\}.
\end{align*}
Then, we define
$$
d_3(\bF,{\bN_{\bF}}) := \sup\{|\EE[h(\bF)]-\EE[h({\bN_{\bF}})]|:h\in\sC_d^3\},
$$
whenever $\EE[\|\bF\|^2]<\infty$, where $\|\,\cdot\,\|$ refers to the Euclidean norm in $\RR^d$. {Similarly} to the quantities $\alpha_{F,1}$, $\alpha_{F,2}$ and $\alpha_{F,3}$, we introduce now 
\begin{equation}\label{Alpha1}
\begin{split}
\alpha_{\bF,1}^2 &:= \sum_{i,j=1}^d\int_\XX\int_\XX\int_\XX(\EE[(\sfD_{x_1}F_i)^2(\sfD_{x_2}F_i)^2])^{1/2}\\
&\hspace{2cm}\times(\EE[(\sfD_{x_1,x_3}^2F_j)^2(\sfD_{x_2,x_3}^2F_j)^2])^{1/2}\,\mu(\dint x_1)\,\mu(\dint x_2)\,\mu(\dint x_3),
\end{split}
\end{equation}
\begin{equation}\label{Alpha2}
\begin{split}
\alpha_{\bF,2}^2 &:=\sum_{i,j=1}^d\int_\XX\int_\XX\int_\XX(\EE[(\sfD_{x_1,x_3}^2F_i)^2(\sfD_{x_2,x_3}^2F_i)^2])^{1/2}\\
&\hspace{2cm}\times(\EE[(\sfD_{x_1,x_3}^2F_j)^2(\sfD_{x_2,x_3}^2F_j)^2])^{1/2}\,\mu(\dint x_1)\,\mu(\dint x_2)\,\mu(\dint x_3)
\end{split}
\end{equation}
and
\begin{equation}\label{Alpha3}
\hspace{-6.6cm}\alpha_{\bF,3} := \sum_{i=1}^d\int_{\XX}\EE[|\sfD_xF_i|^3]\,\mu(\dint x).
\end{equation}
These quantities can be used to bound the $d_3$-distance between $\bF$ and $\bN$, see \cite[Theorem 1.1]{SchulteYukichMulti}.

\begin{proposition}[Multivariate normal approximation of Poisson functionals]\label{prop:CLTmultiD3}
Fix $d\in\NN$ and let $\bF=(F_1,\ldots,F_d)$ be a vector of Poisson functionals such that $F_i\in L_\eta^2$, $\sfD F_i\in L^2(\PP\otimes\mu)$ and $\EE[F_i]=0$ for each $i\in\{1,\ldots,d\}$. Further, let $\bN_\bF$ be a centred Gaussian random vector in $\RR^d$ with the same covariance matrix as $\bF$. Then
$$
d_3(\bF,\bN_\bF) \leq d\,\alpha_{\bF,1}+{d\over 2}\,\alpha_{\bF,2}+{d^2\over 4}\,\alpha_{\bF,3}.
$$
\end{proposition}

\section{Main results}\label{sec:MainResults}

\subsection{Central limit theorem for the volume}\label{subsec:Volume}

\begin{figure}[t]
\centering
\includegraphics[width=0.45\columnwidth]{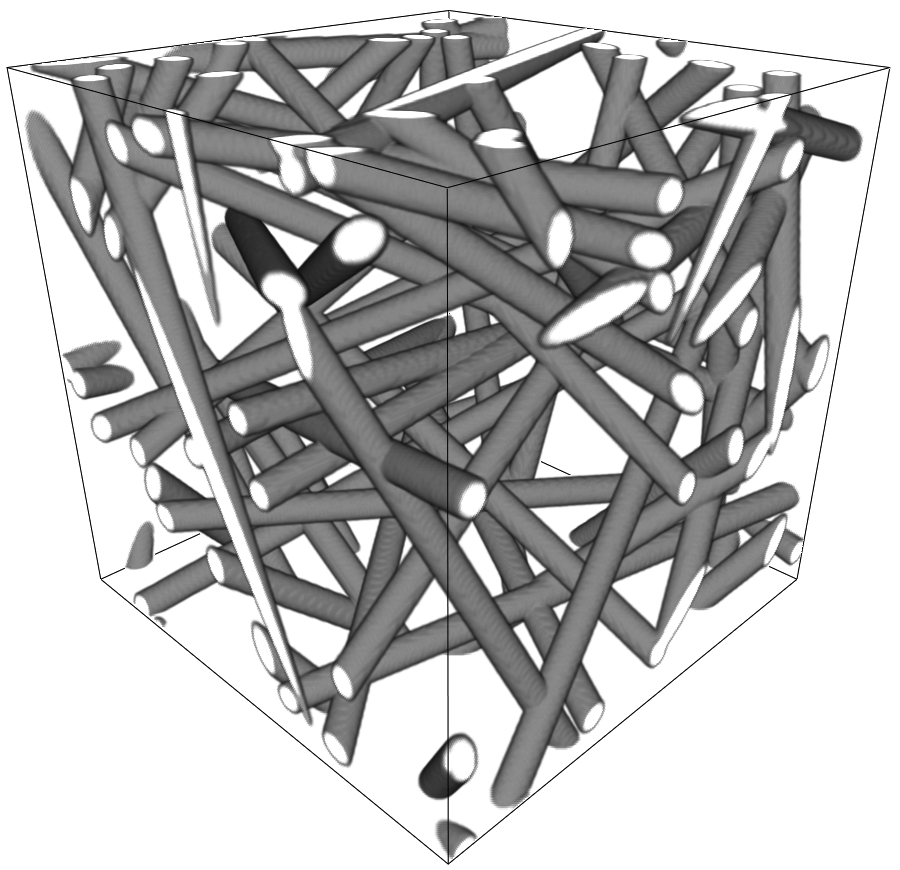}\qquad
\includegraphics[width=0.45\columnwidth]{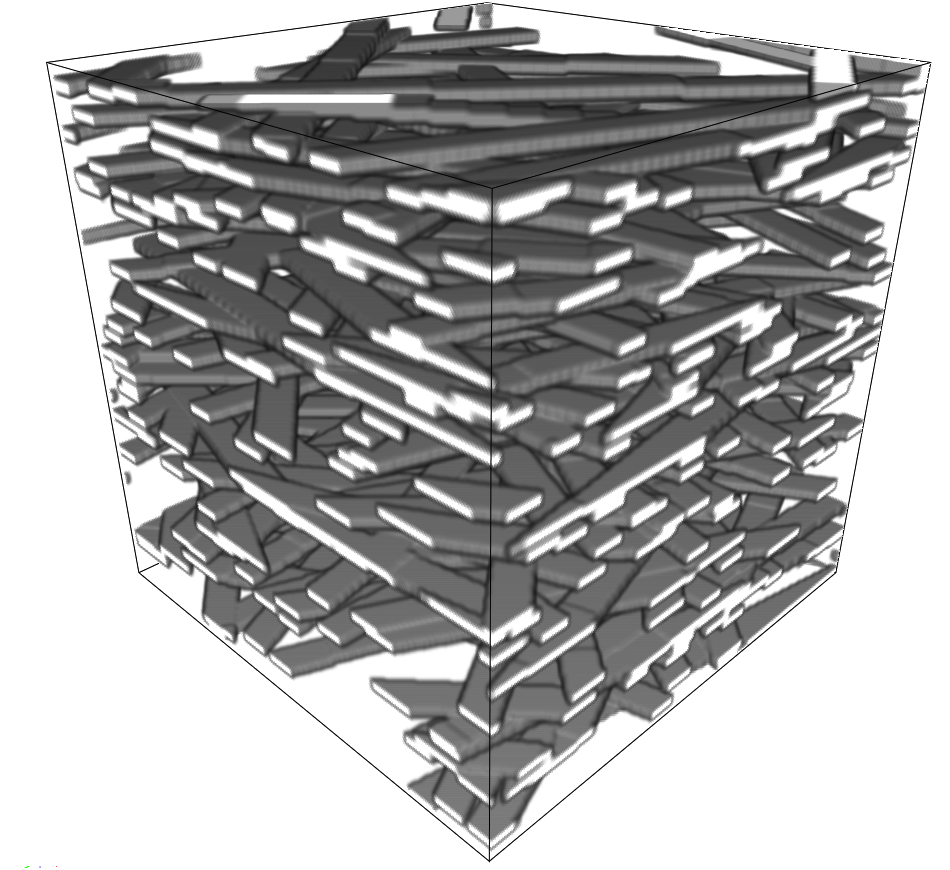}
\caption{{Realisations} of Poisson cylinder processes in $\RR^3$ with $m=1$ and circular (left) and rectangular (right) cylinder bases.}
\label{fig:cylinders}
\end{figure}

Fix a space dimension $n\in\NN$ and another dimension parameter $m\in\{0,1,\ldots,n-1\}$. As introduced before, we denote by $\cC(\RR^{n-m})$ the space of compact subsets of $\RR^{n-m}$ and for $X\in\cC(\RR^{n-m})$ we define the \textbf{cylinder}
$$
Z(X) := X\times\EE^m\subset\RR^n,
$$
where $\EE^m$ stands for the $m$-dimensional linear subspace of $\RR^n$ generated by the $m$ last unit vectors of the standard orthonormal basis of $\RR^n$. We define the product space
$$
\MM_{n,m}:=\SO_{n,m}\times\cC(\RR^{n-m})
$$
and let $\QQ$ be a probability measure on $\MM_{n,m}$. By $\zeta$ we denote a stationary Poisson  process on $\RR^{n-m}$ with intensity $\gamma\in(0,\infty)$ and by $\eta$ an independent $\QQ$-marking of $\zeta$. By the well-known marking property of Poisson processes (see, e.g., \cite[Theorem 5.6]{LP}), $\eta$ is itself a Poisson process on the product space $\RR^{n-m}\times\MM_{n,m}$ with intensity measure given by $\mu:=\gamma\sL^{n-m}\otimes\QQ$. We call $\eta$ a \textbf{Poisson cylinder process} and denote by
$$
Z=\bigcup_{(x,\theta, X)\in \eta} Z(x,\theta,X),
$$
where $Z(x,\theta,X)=\theta((X+x)\times\EE^m)$, the associated \textbf{union set}, see Figure \ref{fig:cylinders}. In what follows it will be convenient for us to denote by $(\Theta,\Xi)$ a random element in $\MM_{n,m}$ with distribution $\QQ$ and to put
$$
{m_a:=\EE[\sL^{n-m}(\Xi)^a]\qquad\text{for}\qquad a>0.}
$$
Under {a condition} on the typical cylinder base the union set $Z$ is a random closed set, see \cite[Lemma 4]{HeinrichSpiessCLTVolume}.

\begin{lemma}\label{lem:ZisRAC}
The union set $Z$ is a random closed set, provided that 
\begin{equation}\label{eqn:CylinderBase}
{\EE[\sL^{n-m}(\Xi+\BB_r^{n-m})]<\infty \quad \text{for some} \quad r>0,}
\end{equation}
where $\Xi+\BB_r^{n-m}$ stands for the Minkowski sum of $\Xi$ and $\BB_r^{n-m}$.
\end{lemma}

{From now on we shall assume that \eqref{eqn:CylinderBase} is satisfied. In particular, \eqref{eqn:CylinderBase} ensures that the union set $Z$ is $\PP$-almost surely locally finite.} We are interested in the volume of the union set $Z$ that can be observed in a test set $W\in\cC(\RR^n)$ with $\sL^n(W)>0$. The expectation and the variance of this random variable are given as follows, see \cite{HeinrichSpiessCLTVolume}. Also, we recall a lower variance bound from \cite[Lemma 1]{HeinrichSpiessCLTVolume}.

\begin{proposition}\label{prop:ExpectationVarianceVolume}
	Fix $n\in\NN$ and $m\in\{0,1,\ldots,n-1\}$, and let $W\in\cC(\RR^n)$. 
	 {Then the following statements hold:}
	\begin{itemize}
		\item[(i)] {$\EE[\sL^n(Z\cap W)] = \sL^n(W)\big(1-e^{-\gamma m_1}\big)$.}
		\item[(ii)] $\displaystyle\var(\sL^n(Z\cap W))=e^{-2\gamma m_1}\int_{\RR^n}\sL^n(W\cap(W-x))\big(e^{-\gamma\EE[\sL^{n-m}(\Xi\cap(\Xi-\Pi(\Theta^Tx))]}-1\big)\,\sL^n(\dint x)$,
		where $\Pi:\RR^n\to\RR^{n-m}$ stands for the orthogonal projection onto the first $n-m$ coordinates.
		\item[(iii)] Suppose that $m_2\in(0,\infty)$ and that $\BB_\delta^n\subseteq W$ for some $\delta>0$. Then, putting $W_r:=rW$ for $r>0$, and $c_v:=\liminf\limits_{r\to\infty}r^{-(n+m)}\var(\sL^n(Z\cap W_r))$, one has that $c_v\in(0,\infty)$ {and there exists a constant $\underline{c}_v\in(0,\infty)$ only depending on $ n,m,\gamma, \delta, m_1$ and $m_2$ such that $c_v\geq \underline{c}_v$.}
	\end{itemize}
\end{proposition}

The next {quantitative} central limit theorem is our main contribution for the volume of the union set of a Poisson cylinder process. 

\begin{theorem}[CLT for the volume]\label{thm:Volume}
	Fix $n\in\NN$ and $m\in\{0,1,\ldots,n-1\}$. Suppose that $m_2\in(0,\infty)$. Let $W\in\cC(\RR^n)$ be such that $\BB_\delta^n\subseteq W$ for some $\delta>0$, and put $W_r:=rW$ for $r>0$. Also, let $N\sim\cN(0,1)$ be a standard Gaussian random variable.
	\begin{itemize}
		\item[(i)] {If $m_3<\infty$, then there exist constants $C,r_0\in(0,\infty)$} such that, for all $ r \geq{r_0}$,
		$$
		d_W\bigg({\sL^n(Z\cap W_r)-\EE[\sL^n(Z\cap W_r)]\over\sqrt{\var(\sL^n(Z\cap W_r))}},N\bigg)\leq C\,r^{-{n-m\over 2}}.
		$$
		\item[(ii)] {If $m_4<\infty$, then there exist constants $C,r_0\in(0,\infty)$} such that, for all $ r \geq {r_0}$,
		$$
		d_K\bigg({\sL^n(Z\cap W_r)-\EE[\sL^n(Z\cap W_r)]\over\sqrt{\var(\sL^n(Z\cap W_r))}},N\bigg)\leq C\,r^{-{n-m\over 2}}.
		$$
	\end{itemize}
\end{theorem}

\begin{remark}\label{rm:ConstantVolume}\rm 
The constants $C$ in parts (i) and (ii) of Theorem \ref{thm:Volume} depend only on $n$, $m$, $\gamma$, $W$, $m_1$, $m_2$ and $m_3$ as well as - in case of (ii) - $m_4$.
The dependence on $W$ is only via $\sL^n(W)$, $\diam(W)$ and $\delta$. In both parts, the constant $ r_0 $ is chosen in such a way that $r^{-(n+m)}\var(\sL^n(Z\cap W_r))\geq \frac{1}{2} {\underline{c}_v} $ for all $ r\geq r_0 $, where {$\underline{c}_v$ is the lower bound for the asymptotic variance constant from Proposition \ref{prop:ExpectationVarianceVolume} (iii).}
\end{remark}

Let us discuss the relation of Theorem \ref{thm:Volume} to the existing literature. The bound for the Kolmogorov distance was previously found by Heinrich and Spiess \cite{HeinrichSpiessCLTVolume}, but under much more restrictive moment conditions. In fact, they assume that the $(n-m)$-volume of the typical cylinder base $\Xi$ has finite \textit{exponential} moments, while our approach works under a third- or fourth-order moment assumption, respectively, depending on whether one seeks for a bound for the Wasserstein or the Kolmogorov distance. Against this light, Theorem \ref{thm:Volume} answers a question raised after Theorem 2 in \cite{HeinrichSpiessCLTVolume}. Without a rate of convergence, a central limit theorem for the volume of the union set of a Poisson cylinder process was proved by Heinrich and Spiess \cite{HeinrichSpiessCLTVundS} under the second moment assumption $m_2<\infty$. This is in line with the classical central limit theorem for sums of independent and identically distributed random variables, which holds under a second moment assumption as well. When we take $m=0$, the union set $Z$ is the same as the classical Boolean model. {In this case, central limit theorems for the volume were first obtained by Baddeley \cite{Baddeley1980} and Mase \cite{Mase1982}. A quantitative central limit theorem for the Kolmogorov distance was shown by Heinrich \cite{Heinrich2005} under the assumption that the volume of the typical grain has some finite exponential moments.
A rate of convergence in terms of the Wasserstein distance is due to Hug, Last and Schulte \cite{HugLastSchulteBM} (see also Chapter 22 in \cite{LP}). A direct proof for the rate of convergence for the Kolmogorov distance under the assumption that the fourth moment exists} is new even for the Boolean model, but the result can also be concluded as a special case of the quantitative multivariate central limit theorem \cite[Theorem 4.2 {(d)}]{SchulteYukichMulti} for the convex distance.

\subsection{Central limit {theorems} for geometric functionals}\label{sec:CLTAdditiveFunctionals}

We adopt the notation from the previous section and let $\eta$ be a Poisson cylinder process with stationary union set $Z$. However, from now on we assume that all cylinders have a \textit{convex} base, that is, instead of $\SO_{n,m}\times\cC(\RR^{n-m})$ in the construction of $\eta$ we consider the mark space
$$
\MM_{n,m}=\SO_{n,m}\times\cK(\RR^{n-m}),
$$
where $\cK(\RR^{n-m})$ denotes the space of {compact} convex subsets of $\RR^{n-m}$. Assuming that {the condition \eqref{eqn:CylinderBase} on the typical cylinder base $\Xi$ is} satisfied, the set $Z\cap W$ is {$\PP$-almost} surely an element of the convex ring $\cR(\RR^n)$ for any test set $W\in\cK(\RR^n)$, where we recall that $\cR(\RR^n)$ is the space of finite unions of {compact} convex sets. In particular, this implies that the intrinsic volumes $V_i(Z\cap W)$ for $i\in\{0,1,\ldots,n\}$ are well-{-}defined random variables (here and in what follows, by the intrinsic volumes we understand the additive extension of the intrinsic volumes for convex sets to the convex ring $\cR(\RR^n)$). Our goal in this section is to prove a central limit theorem for $V_i(Z\cap W_r)$, as $r\to\infty$, where $W_r=rW$ for some fixed $W\in\cK(\RR^n)$. More generally, following \cite{HugLastSchulteBM,LP}, we shall consider a rather broad class of additive functionals satisfying some further natural assumptions.

A function $\varphi:\cR(\RR^n)\to\RR$ is called \textbf{additive}, provided that $\varphi(\varnothing)=0$ and
\begin{equation}\label{eq:Additivity}
\varphi(K\cup L)=\varphi(K)+\varphi(L)-\varphi(K\cap L)
\end{equation}
for all $K,L\in\cR(\RR^n)$. We call $\varphi$ \textbf{translation invariant} if 
\begin{equation}\label{eq:TranslationInvariance}
\varphi(K+x)=\varphi(K)
\end{equation}
for all $x\in\RR^n$ and $K\in\cR(\RR^n)$. Moreover, we say that $\varphi$ is \textbf{locally bounded} if
\begin{equation}\label{eq:LocalBoundedness}
M(\varphi):=\sup\{|\varphi(K)|:K\in\cK(\RR^n),K\subseteq[0,1]^n\}<\infty.
\end{equation}
Finally, by a \textbf{geometric} functional $\varphi$ we understand a measurable functional which is additive, translation invariant and locally bounded. The wide class of geometric functionals contains the following examples, which are of particular interest (see also the discussion in {\cite[page 79]{HugLastSchulteBM}}):
\begin{itemize}
	\item[-] the intrinsic volumes $\varphi(K)=V_i(K)$ with $i\in\{0,1,\ldots,n\}$ (see \cite[Chapter 4]{Schneider}),
	\item[-] mixed volumes of the form $\varphi(K)=V(K[\ell],K_1,\ldots,K_{n-\ell})$ for $\ell\in\{0,1,\ldots,n\}$, where the notation $K[\ell]$ means that $K$ is repeated $\ell$ times and $K_1,\ldots,K_{n-\ell}\in\cR(\RR^n)$ are fixed (see \cite[Chapter 5]{Schneider}),
	\item[-] integrals of surface area measures of the form $\varphi(K)=\int_{\SS^{n-1}}g(u)\,S_K^{(\ell)}(\dint u)$, where $S_K^{(\ell)}$, $\ell\in\{0,1,\ldots,n-1\}$ is the $\ell$th surface area measure of $K$ and $g:\SS^{n-1}\to\RR$ is a bounded measurable function (see \cite[Chapter 4]{Schneider}),
	\item[-] the centred support function $\varphi(K)=h(K-s(K),u)$ for fixed $u\in\SS^{n-1}$, where $s(K)$ stands for the Steiner point of $K$ (see \cite[page 262]{Schneider}),
	\item[-] the total measures from translative integral geometry (see \cite[{Chapter} 6.4]{SW}).
\end{itemize}
Our main result in this section is the following quantitative central limit theorem for geometric functionals. To state it, we introduce the following notation. For a geometric functional $\varphi$ and $W\in\cK(\RR^n)$ we define the asymptotic (lower) variance constant $v(\varphi,W)\in[0,\infty]$ by
\begin{equation}\label{eq:VarianceConstant}
v(\varphi,W) := \liminf_{r\to\infty}r^{-(n+m)}{\var(\varphi(Z\cap W_r))}.
\end{equation}
We {emphasise} that for $\varphi=V_n$ {(where we write $V_n$ instead of $\sL^n$ for sets from the convex ring)} we have  already encountered the variance constant  in Proposition \ref{prop:ExpectationVarianceVolume} (iii), where $ v(V_n,W)=c_v$.

\begin{theorem}[CLT {for geometric} functionals]\label{thm:AdditiveFunctionals}
	Fix $n\in\NN$ and $m\in\{0,1,\ldots,n-1\}$, and let $\varphi$ be a geometric functional. Let $W\in\cK(\RR^n)$ with $V_n(W)>0$, and put $W_r:=rW$ for $r>0$. Also, let $N\sim\cN(0,1)$ be a standard Gaussian random variable and assume that $v(\varphi,W)>0$.
	\begin{itemize}
	\item [(i)] {If $\EE[V_j(\Xi)^2]<\infty$ for all $j\in\{1,\ldots,n-m\}$, then
	$$
	{\varphi(Z\cap W_r)-\EE[\varphi(Z\cap W_r)]\over\sqrt{\var(\varphi(Z\cap W_r))}} \overset{d}{\longrightarrow} N \quad \text{as} \quad r\to\infty.
	$$}
		\item[(ii)] {If $\EE[V_j(\Xi)^3]<\infty$ for all $j\in\{1,\ldots,n-m\}$, then there exist constants $C, r_0\in(0,\infty)$} such that, for all $ r \geq \max\{1,r_0\}$, 
		$$
		d_W\bigg({\varphi(Z\cap W_r)-\EE[\varphi(Z\cap W_r)]\over\sqrt{\var(\varphi(Z\cap W_r))}},N\bigg)\leq C\,r^{-{n-m\over 2}}.
		$$
		\item[(iii)] {If $\EE[V_j(\Xi)^4]<\infty$ for all $j\in\{1,\ldots,n-m\}$, then there exist constants $C, r_0\in(0,\infty)$} such that, for all $ r \geq \max\{1,r_0\}$, 
		$$
		d_K\bigg({\varphi(Z\cap W_r)-\EE[\varphi(Z\cap W_r)]\over\sqrt{\var(\varphi(Z\cap W_r))}},N\bigg)\leq C\,r^{-{n-m\over 2}}.
		$$
	\end{itemize}
\end{theorem}

\begin{remark}\label{rm:ConstantsGeneral}\rm 
{The constants $C$ in parts {(ii) and (iii)} of Theorem \ref{thm:AdditiveFunctionals} only depend on $n$, $m$, $\gamma$, $W$, $M(\varphi)$, {$v(\varphi,W)$,} $\EE[V_j(\Xi)^2]$, $\EE[V_j(\Xi)^3]$ and - in case of {(iii)} - $\EE[V_j(\Xi)^4]$, $j\in\{1,\ldots,n-m\}$.	The dependence on the test set $W$} is only via {$v(\varphi,W)$,} $V_n(W^{\sqrt{n}})$ and $\diam(W)$, where we recall that $W^{\sqrt{n}}$ is the $\sqrt{n}$-parallel set of $W$. Note that according to Steiner's formula (see \cite[Equation (14.5)]{SW}), $V_n(W^{\sqrt{n}})$ can be expressed as a linear combination of the intrinsic volumes ${V_0(W)},\ldots,V_n(W)$ with dimension-dependent coefficients.	 The constant $ r_0 $ appearing in parts {(ii) and (iii)} is chosen in such a way that $ r^{-(n+m)}\var(\varphi(Z\cap W_r))\geq \frac{1}{2} v(\varphi,W) $ for all $ r\geq r_0 $, which is possible due to our assumption that $ v(\varphi,W) >0$.  
\end{remark}

Choosing $\varphi=V_n$ Theorem \ref{thm:AdditiveFunctionals} {(ii) and (iii) reduce} to Theorem \ref{thm:Volume} but under more restrictive geometric assumptions on the typical cylinder base $\Xi$ and the test set $W$. {We would like to remark that for $\varphi=V_{n-1}$ a central limit theorem for $\varphi(Z\cap W_r)$ under a second moment assumption was derived in \cite{HeinrichSpiessCLTVundS} without a rate of convergence.
Moreover, choosing $m=0$, in which case the random set $Z$ is just the Boolean model, the {results of Theorem \ref{thm:AdditiveFunctionals} (i) and (ii) are} known from \cite{HugLastSchulteBM,LP}, while the bound in Theorem \ref{thm:AdditiveFunctionals} {(iii)} is again a special case of \cite[Theorem 4.2 {(d)}]{SchulteYukichMulti}. Non-quantitative central limit theorems for the surface area and related quantities of Boolean models were derived in \cite{HeinrichMolchanov,Molchanov1995}.} 

\bigskip

After having developed a univariate central limit theory, we now turn  to the multivariate case. 
Fix $d\in\NN$, let $\varphi_1,\ldots,\varphi_d$ be geometric functionals and put
$$
F_r^{(i)} := \varphi_i(Z\cap W_r),\qquad i\in\{1,\ldots,d\},
$$
where $W_r=rW$ for $r>0$. Further define the $d$-dimensional random vector
$$
\bF_r:=\bigg({F_r^{(1)}-\EE[F_r^{(1)}]\over r^{n+m\over 2}},\ldots,{F_r^{(d)}-\EE[F_r^{(d)}]\over r^{n+m\over 2}}\bigg)
$$
as well as
$$
\sigma_{ij}(r) := \Cov\bigg({F_r^{(i)}-\EE[F_r^{(i)}]\over r^{n+m\over 2}},{F_r^{(j)}-\EE[F_r^{(j)}]\over r^{n+m\over 2}}\bigg),\qquad i,j\in\{1,\ldots,d\}.
$$

The following result is the multivariate version of Theorem \ref{thm:AdditiveFunctionals}, where in the quantitative part (ii) the speed of convergence is {measured} in the $d_3$-distance. 

\begin{theorem}\label{thm:Multivariate}
Let the set-up just described prevail and put $\Sigma_r:=(\sigma_{ij}(r))_{i,j=1}^d$.
	\begin{itemize}
	\item[(i)] Assume that $\EE[V_j(\Xi)^2]<\infty$ for $j\in\{1,\ldots,n-m\}$ and that the covariance matrix $\Sigma_r$ converges, as $r\to\infty$, to some covariance matrix $\Sigma$. Denote by $\bN_{\Sigma}$ a $d$-dimensional centred Gaussian random vector with covariance matrix $\Sigma$.	Then, $$\bF_r\overset{d}{\longrightarrow}\bN_\Sigma,\qquad\text{as}\qquad r\to\infty.$$
	\item[(ii)] Assume that $\EE[V_j(\Xi)^3]<\infty$ for $j\in\{1,\ldots,n-m\}$ and denote by $\bN_{\bF_r}$ a $d$-dimensional centred Gaussian random vector with covariance matrix $\Sigma_r$. Then there exists a constant $C\in(0,\infty)$ only depending on $n$, $m$, $d$, $\gamma$, $W$, $M(\varphi_i)$ for $i\in\{1,\ldots,d\}$, $\EE[V_j(\Xi)^2]$ and $\EE[V_j(\Xi)^3]$ for $j\in\{1,\ldots,n-m\}$  such that, for all $r\geq 1$,
$$
d_3(\bF_r,\bN_{\bF_r}) \leq C\,r^{-{n-m\over 2}}. 
$$
	\end{itemize}
\end{theorem}

As a special case the previous result contains a multivariate central limit theorem for the (centred and {normalised}) random vector {of} intrinsic volumes of $Z\cap W_r$, as $r\to\infty$, if we choose $d=n+1$ and {$\varphi_i=V_{i-1}$ for $i\in\{1,\ldots,n+1\}$.} {Theorem~\ref{thm:Multivariate}} {generalises} the multivariate central limit theorem from \cite{HugLastSchulteBM} for the Boolean model (corresponding to the choice $m=0$) to general Poisson cylinder processes. 

\begin{remark}\rm
As a multivariate counterpart to the one-dimensional Kolmogorov distance one usually considers the so called convex distance. For two $d$-dimensional random vectors $\mathbf{X}$ and $\mathbf{Y}$ it is given by
$$
d_{\rm convex}(\mathbf{X},\mathbf{Y}):=\sup \big\{ |\mathbb{P}(\mathbf{X}\in A) - \mathbb{P}(\mathbf{Y}\in A)| : A\in\mathcal{K}(\mathbb{R}^d) \big\}. 
$$
The convex distance dominates the multivariate Kolmogorov distance, which is the supremum norm of the difference of the distribution functions of $\mathbf{X}$ and $\mathbf{Y}$, and is in contrast to the multivariate Kolmogorov distance invariant under linear transformations. In \cite{SchulteYukichMulti} second-order Poincar\'e inequalities for the multivariate normal approximation of Poisson functionals in the $d_{\rm convex}$-distance were derived. As an example, multivariate central limit theorems with rates of convergence for the convex distance were shown for intrinsic volumes of Boolean models in \cite[Theorem 4.2]{SchulteYukichMulti}. By combining \cite[Theorem 1.2]{SchulteYukichMulti} with similar arguments as in the proofs of Theorem \ref{thm:AdditiveFunctionals}, Theorem \ref{thm:Multivariate} and \cite[Theorem 4.2]{SchulteYukichMulti}, one can extend Theorem \ref{thm:Multivariate} to the $d_{\rm convex}$-distance. Under the slightly stronger assumption that $\EE[V_j(\Xi)^4]<\infty$ for $j\in\{1,\ldots,n-m\}$ and that the covariance matrix $\Sigma_r$ is positive definite, one obtains
\begin{equation}\label{eqn:dconvex}
d_{\rm convex}(\bF_r,\bN_{\bF_r}) \leq C \max\{ \|\Sigma_r^{-1/2}\|_{\rm op}, \|\Sigma_r^{-1/2}\|_{\rm op}^3 \} \,r^{-{n-m\over 2}}
\end{equation}
for all $r\geq 1$. Here, $C\in(0,\infty)$ is some constant depending on the model parameters, $\Sigma_r^{-1/2}$ denotes the unique positive definite matrix such that $\Sigma_r^{-1/2} \Sigma_r^{-1/2}=\Sigma_r^{-1}$, and $\|\,\cdot\,\|_{\rm op}$ stands for the operator norm of a matrix. For the Boolean model, one can show for many choices of additive functionals that $\Sigma_r$ converges to a positive definite matrix as $r\to\infty$ (see \cite[Section 4]{HugLastSchulteBM}) so that the maximum in \eqref{eqn:dconvex} can be bounded by a constant. For general Poisson cylinder processes the covariance structure can behave differently than for the Boolean model and it can happen that the limiting covariance matrix is singular (see Subsection \ref{subsec:Variance} and, in particular, the discussion next to Remark \ref{remark:v}). In this situation it is not clear how the right-hand side of \eqref{eqn:dconvex} behaves for $r\to\infty$ as the maximum tends to infinity. For that reason we have restricted ourselves to the $d_3$-distance.
\end{remark}

\begin{remark}\rm
If $\Xi=\{0\}$ $\mathbb{P}$-almost surely, the Poisson cylinder process reduces to a Poisson process of $m$-dimensional flats in $\RR^n$. In this case one has for any additive functional $\varphi$ that $\mathbb{P}$-almost surely
$$
\varphi(Z\cap W_r)= \sum_{\ell=1}^d \frac{1}{\ell!} \sum_{((x_1,\theta_1,X_1),\hdots,(x_\ell,\theta_\ell,X_\ell))\in\eta^{\ell}_{\neq}} \varphi(\theta_1 (\{x_1\}\times \mathbb{E}^m) \cap \hdots \cap \theta_\ell (\{x_\ell\}\times \mathbb{E}^m) \cap W_r),
$$
where $\eta^{\ell}_{\neq}$ denotes the set of all $\ell$-tuples of distinct elements of $\eta$. The $\ell$-th summand on the right-hand side is a so-called Poisson $U$-statistic of order $\ell$. For such random variables multivariate central limit theorems were established in \cite{LPST}. If the geometric functionals are the intrinsic volumes, the assertion of Theorem \ref{thm:Multivariate} can also be deduced from \cite[Theorem 3]{LPST}.
\end{remark}

\subsection{Variance asymptotics for geometric functionals}\label{subsec:Variance}

In this section we discuss the assumption that $v(\varphi,W)>0$ in Theorem \ref{thm:AdditiveFunctionals}. In contrast to previous findings, we assume  that $m>0$  throughout and, thus, exclude the case of a Boolean model, for which the asymptotic covariance structure is studied in \cite{HugLastSchulteBM}. We start by presenting an explicit formula for the asymptotic variance constant $v(\varphi,W)$. To this end, we introduce some notation, which follows \cite[Chapter 9.2]{SW}. For a geometric functional $\varphi:\cR(\RR^n)\to\RR$, $X\in\cK(\RR^{n-m})$ and $\theta\in\SO_{n,m}$ we put
\begin{align}\label{eq:PhiHalfOpenCube}
\varphi(\theta([0,1)^m+X)) := \varphi(\theta([0,1]^m+X)) - \varphi(\theta(\partial^+[0,1)^m+X)),
\end{align}
where $\partial^+[0,1)^m:=[0,1]^m\setminus [0,1)^m$ is the upper right boundary of the unit cube in $\EE^m$. For $\theta\in\SO_{n,m}$ we define
\begin{equation}\label{eq:DefTwtheta}
T(W,\theta) := \int_{\RR^{n-m}}\sL^m(H(y,\theta)\cap W)^2\,\sL^{n-m}(\dint y),
\end{equation}
where $H(y,\theta)=\theta(\EE^m+y)$. Finally, if $\QQ$ denotes the base-direction distribution of a Poisson cylinder process, we call the marginal $\QQ_{n,m}$ of $\QQ$ onto the $\SO_{n,m}$-coordinate its \textbf{direction distribution}. This prepares us for the formulation of the main result of this section. 

\begin{theorem}\label{thm:asymptoticvariance}
	Fix $n\in\NN$ and $ m\in\{1,\ldots,n-1\}$, and consider a Poisson cylinder process with intensity $\gamma\in(0,\infty)$ and base-direction distribution $\QQ$ such that $\EE[V_j(\Xi)^2]<\infty$ for all $j\in\{1,\hdots,n-m\}$. Moreover, assume that $(\varrho_i)_{i\in I}$ are the at most countably many atoms of the direction distribution $\QQ_{n,m}$. Further, let $\varphi$ be a geometric functional and $W\in\cK(\RR^n)$. Then $v(\varphi,W)$ defined by \eqref{eq:VarianceConstant} satisfies
\begin{align*}
v(\varphi, W) &= {\lim_{r\to\infty}r^{-(n+m)}{\var(\varphi(Z\cap W_r))}} \\
&=\gamma\int_{\mathbb{M}_{n,m}} \bigg( \EE[ \varphi( Z\cap \theta ([0,1)^m+X))] - \varphi(\theta ([0,1)^m+X)) \bigg)^2 T(W,\theta) \,  \mathbb{Q}(\dint(\theta,X)) \allowdisplaybreaks \\
& \quad + \sum_{k=2}^\infty \frac{\gamma^k}{k!} \sum_{i\in I} \int_{\mathbb{M}_{n,m}^k}  \mathbf{1}\{\theta_1=\hdots=\theta_k=\varrho_i\} {T(W,\varrho_i)}\\
& \qquad \qquad \qquad \times \int_{(\mathbb{R}^{n-m})^{k-1}} \bigg( \EE\Big[ \varphi\Big( Z\cap \varrho_i \Big([0,1)^m+X_1\cap\bigcap_{j=2}^k (x_j+X_j)\Big)\Big)\Big] \\
& \qquad \qquad \qquad \qquad - \varphi\Big(\varrho_i \Big([0,1)^m+X_1\cap\bigcap_{j=2}^k (x_j+X_j)\Big)\Big) \bigg)^2 \, (\sL^{n-m})^{k-1}(\dint(x_2,\hdots,x_k)) \\
& \qquad \qquad \qquad \times \mathbb{Q}^k(\dint((\theta_1,X_1),\hdots, (\theta_k,X_k) )).
\end{align*}
In particular, if the direction distribution $\QQ_{n,m}$ has no atoms, i.e.\ if $I=\varnothing$, the series over $k$ vanishes.
\end{theorem}

\begin{remark}\label{rem:Covariance}\rm 
Let $\varphi_1$ and $\varphi_2$ be geometric functionals and $W\in\cK(\RR^n)$. Then the sum $\varphi_1+\varphi_2$ is a geometric functional as well and
\begin{align*}
&{\rm cov}(\varphi_1(Z\cap W_r),\varphi_2(Z\cap W_r)) \\
&\qquad= \frac{1}{2} \big( \var(\varphi_1(Z\cap W_r)+\varphi_2(Z \cap W_r)) - \var(\varphi_1(Z\cap W_r)) - \var(\varphi_2(Z\cap W_r)) \big)
\end{align*}
holds for any $r\geq 1$. Thus, dividing by $r^{n+m}$ and then taking the limit as $r\to\infty$ we can also conclude from Theorem \ref{thm:asymptoticvariance} a formula for the asymptotic covariance constant
$$
\lim_{r\to\infty}r^{-(n+m)}{{\rm cov}(\varphi_1(Z\cap W_r),\varphi_2(Z\cap W_r))}
$$
under the same assumptions on the typical cylinder base as in Theorem \ref{thm:asymptoticvariance}.
\end{remark}

{To exploit the formula in Theorem \ref{thm:asymptoticvariance}, we consider the asymptotic variances and covariances of $V_n$ and $V_{n-1}$. For corresponding formulas for Boolean models we refer the reader to \cite[Corollary 6.2]{HugLastSchulteBM}.} {We} start by discussing the case $\varphi=V_n$, where the expression for $v(V_n,W)$ is known from \cite[Corollary 1]{HeinrichSpiessCLTVundS}.

\begin{corollary}\label{cor:VarVolumeConstant}
Fix $n\in\NN$ and $m\in\{1,\ldots,n-1\}$, and consider a Poisson cylinder process with intensity $\gamma\in(0,\infty)$ and base-direction distribution $\QQ$ with $\EE[V_j(\Xi)^2]<\infty$ for all $j\in\{1,\hdots,n-m\}$. Let $W\in\cK(\RR^n)$ be such that $V_n(W)>0$. Then
\begin{align*}
v(V_n,W) &= {\lim_{r\to\infty}r^{-(n+m)}{\var(V_n(Z\cap W_r))}} \\
& = \gamma e^{-2\gamma m_1}\EE[V_{n-m}(\Xi)^2T(W,\Theta) \mathbf{1}\{\Theta\notin\{\varrho_i: i\in I\}\} ]\\
&\qquad+e^{-2\gamma m_1}\sum_{i\in I } T(W,\varrho_i)\int_{\RR^{n-m}}[e^{\gamma f(z,\varrho_i)} -1]\,\sL^{n-m}(\dint z),
\end{align*}
where $\{\varrho_i:i\in I\}$ is the (at most countable) set of atoms of the direction distribution $\QQ_{n,m}$ and
\begin{equation}\label{eq:defFunctionF}
f(z,\varrho_i):=\EE[V_{n-m}(\Xi\cap(\Xi+z)){\bf 1}\{\Theta=\varrho_i\}],\qquad z\in\RR^{n-m}, i\in I.
\end{equation}
\end{corollary}

\begin{remark}\rm 
Although the classical Boolean model, corresponding to the choice $m=0$, is excluded throughout this section, we can retrieve the asymptotic variance constant for the volume of the Boolean model from Corollary \ref{cor:VarVolumeConstant}, see e.g.\ \cite[Corollary 6.2]{HugLastSchulteBM}. For this, we fix an orthonormal basis $e_1,\ldots,e_{n+1}$ in $\RR^{n+1}$ and identify ${\rm span}(e_1,\ldots,e_n)$ with $\RR^n$. We consider a stationary Boolean model in $\RR^n$ with intensity $\gamma\in(0,\infty)$ and grains $K_i\in\cK(\RR^n)$, $i\in\NN$. Fix  $U\in\cK(\RR^n)$ and define
$$
\sigma_U^2:=\lim_{r\to\infty}r^{-n}\var(V_n(Y\cap U_r)),\qquad Y:=\bigcup_{i=1}^\infty K_i,
$$
where $U_r:=rU$, as usual. Next, we construct in $\RR^{n+1}$ the Poisson cylinder process  with cylinders $K_i\times{\rm span}(e_{n+1})$, $i\in\NN$, and union set $Z$. Further, we choose $W_r=U_r\times[0,r]$. The variance of $V_{n+1}(Z\cap W_r)$ is connected to the variance of the Boolean model by
$$
\var(V_{n+1}(Z\cap W_r) ) = \var(rV_n(Y\cap U_r)) = r^2\var(V_n(Y\cap U_r)) .
$$
Moreover, {we have} $T(W,{\operatorname{id}})=\int_{\RR^n}\sL(H(y,{\operatorname{id}})\cap U\times[0,1])^2\,\sL^n(\dint y)=V_n(U)$ {with $\operatorname{id}$ being the identity in $\SO_{n,m}$.}
Thus, it follows from Corollary \ref{cor:VarVolumeConstant} that
$$
\sigma_U^2 = e^{-2\gamma m_1}V_n(U)\int_{\RR^n}[e^{\gamma f(z,e_{n+1})}-1]\,\sL^n(\dint z),
$$
under the assumption that $m_2<\infty$, where $m_1$ and $m_2$ are the first and second moment of the volume of the typical grain of the Boolean model, respectively.
\end{remark}

It turns out to be convenient to assume in addition that the typical cylinder base $\Xi$ and its direction $\Theta$ are independent and that $\Theta$ has the uniform distribution $\nu_{n,m}$ on $\SO_{n,m}$. In this case we will speak of a \textbf{uniform Poisson cylinder process} with typical base $\Xi$. In particular, these assumptions imply that the random union set $Z$ is {not only stationary but also isotropic.} In fact, in the uniform case the expectation in $v(V_n,W)$ {in Corollary \ref{cor:VarVolumeConstant} factorises.} While ${\EE[V_{n-m}(\Xi)^2]}=m_2$, $\EE[T(W,\Theta)]$ can be expressed as
\begin{align*}
&\int_{G(n,m)}\int_{{E_0}^\perp}\sL^m(({E_0}+x)\cap W)^2\,\sL_{E_0^\perp}^{n-m}(\dint x)\,\widehat{\QQ}_{n,m}(\dint {E_0}) = {\kappa_m\over m+1}\int_{A(n,1)}\sL^1(L\cap W)^{m+1}\,\nu_1(\dint L),
\end{align*}
where $\widehat{\QQ}_{n,m}$ is the image measure on $G(n,m)$ of $\QQ_{n,m}$ under the mapping $\theta\mapsto\theta\EE^m$.
Indeed, it follows from the uniqueness of invariant measures that in this case
\begin{align*}
\int_{G(n,m)}\int_{{E_0}^\perp}\sL^m(({E_0}+x)\cap W)^2\,\sL_{E_0^\perp}^{n-m}(\dint x)\,\widehat{\QQ}_{n,m}(\dint {E_0}) &= \int_{A(n,m)}\sL^m(E\cap W)^2\,\mu_m(\dint E)
\end{align*}
and we can now use \cite[Equation (8.57)]{SW} to conclude that
$$
\int_{A(n,m)}\sL^m(E\cap W)^2\,\mu_m(\dint E)={\kappa_m\over m+1}\int_{A(n,1)}\sL^1(L\cap W)^{m+1}\,\nu_1(\dint L).
$$
In particular, in view of Corollary \ref{cor:VarVolumeConstant} and the fact that $I=\varnothing$ in the uniform case this leads to 
\begin{align}\label{eq:vVnIsotropic}
v(V_n, W) &= \gamma m_2e^{-2\gamma m_1}{\kappa_m\over m+1}\int_{A(n,1)}\sL^1(L\cap W)^{m+1}\,\nu_1(\dint L),
\end{align}
an expression which has previously been derived in \cite[Equation (11) and Section 5.3]{HeinrichSpiessCLTVundS}.\\

In what follows it is convenient to distinguish the cases where the cylinder process possesses a non-atomic or purely atomic direction distribution. We say that a Poisson cylinder process is \textbf{direction non-atomic}, provided that its direction distribution $\QQ_{n,m}$ has no atoms, that is, if $\QQ_{n,m}(\{\theta\})=0$ for all $\theta\in\SO_{n,m}$. On the other hand, a Poisson cylinder process is called \textbf{purely direction atomic} if there exists an at most countable set $I$ and elements $\varrho_i\in\SO_{n,m}$, $i\in I$, such that $\QQ_{n,m}(\{\varrho_i:i\in I\})=1$.

\begin{remark}\label{rem:VarianceVOLUME}\rm 
	\begin{itemize}
		\item[(i)] 
		The integral $I_{m+1}(W):=\int_{A(n,1)}\sL^1(L\cap W)^{m+1}\,\nu_1(\dint L)$ appearing in \eqref{eq:vVnIsotropic} is a well-known quantity in convex and integral geometry, the so-called \textbf{$(m+1)$st chord-power integral} of $W$. It also admits another representation. Namely, according to item 4.\ in the notes to Chapter 8.6 in \cite{SW} one has that
		$$
		I_{m+1}(W)={m(m+1)\over n\kappa_n}\int_W\int_W{\sL^n(\dint x)\sL^n(\dint y)\over\|x-y\|^{n-m}},
		$$
		an expression which is known as the \textbf{$(n-m)$-energy of $W$}.
		\item[(ii)] 
		{The term $T(W,\theta)$ (or $\int_{A(n,1)}\sL^1(L\cap W)^{m+1}\,\nu_1(\dint L)$ in the uniform case) is known to appear in the second-order analysis of Poisson processes of $m$-dimensional flats in $\RR^n$, see \cite[Section 5.4]{HeinrichSpiessCLTVundS} and \cite[{Section 6}]{LPST}.}
		 \item[(iii)] The term $T(W,\theta)$ appearing in Corollary \ref{cor:VarVolumeConstant} and also in Corollary \ref{cor:VarSurface} below simplifies considerably if we take for $W$ the $n$-dimensional unit ball $\BB^n$. In this case, $H(y,\theta)\cap\BB^n$ is a ball of radius $\sqrt{1-\|y\|^2}$ and $\sL^m(H(y,\theta)\cap\BB^n)=\kappa_m(1-\|y\|^2)^{m/2}$. Introducing spherical coordinates in $\RR^{n-m}$ we find that
		\begin{align*}
		T(\BB^n,\theta) &= \int_{\RR^{n-m}}\sL^m(H(y,\theta)\cap\BB^n)^2\,\sL^{n-m}(\dint y) \\
		 &= (n-m)\kappa_{n-m}\kappa_m^2\int_0^1(1-r^2)^mr^{n-m-1}\,\dint r\\
		  &= m!\pi^{-m}\kappa_m^2\kappa_{n+m},
		\end{align*}
		 independently of $\theta$. For example, this shows that
		 $
		 v(V_n,\BB^n) = \gamma m!\pi^{-m}\kappa_m^2\kappa_{n+m} m_2e^{-\gamma m_1}
		 $
		 if the Poisson cylinder process is direction non-atomic.
	\end{itemize}
\end{remark}

In Proposition \ref{prop:ExpectationVarianceVolume} (iii) above we have seen that for a general Poisson cylinder process we have that $c_v=v(V_n,W)>0$, provided that $\gamma>0$, $m_2>0$ and $V_n(W)>0$. However, this is not necessarily the case for other geometric functionals $\varphi$ as Examples~\ref{Ex:DirectionNonAtomic} and \ref{Ex:DirectionAtomic} {will show.} In contrast to Corollary \ref{cor:VarVolumeConstant} we now take $\varphi=V_{n-1}$ to be the intrinsic volume of order $n-1$. {In \cite[Theorem 3]{HeinrichSpiessCLTVundS}, formulas for $\lim_{r\rightarrow\infty}r^{-(n+m)}{\var(\sH^{n-1}(\partial Z\cap W_r^{\circ}))}$ for the $(n-1)$-dimensional Hausdorff measure $\sH^{n-1}$ in the direction non-atomic case and in the purely direction atomic case were stated without a proof. If the typical cylinder base has dimension $n-m$ almost surely, the formulas differ only by the factor $4$ from $v(V_{n-1},W)$ since, then, $V_{n-1}=\frac{1}{2}\sH^{n-1}$. Moreover, we see that in this case the different treatment of the boundary of $W$ does not play a role.}

\begin{corollary}\label{cor:VarSurface}
Fix $n\in\NN$ and $m\in\{1,\ldots,n-1\}$, and consider a Poisson cylinder process with intensity $\gamma\in(0,\infty)$  and base-direction distribution $\QQ$  with $\EE[V_j(\Xi)^2]<\infty$ for all $j\in\{1,\hdots,n-m\}$. Let $T(W,\theta)$ be as in \eqref{eq:DefTwtheta} and put $m_1:=\EE[V_{n-m}(\Xi)]$, $s_1:=\EE[V_{n-m-1}(\Xi)]$. Also, let $f(z,\varrho_i)$ be as in \eqref{eq:defFunctionF} and define 
\begin{equation}\label{eqn:Definition_g} 
	g(z,\varrho_i) := \EE[\sH^{n-m-1}(\partial\Xi\cap(\Xi-z))\,{\bf 1}\{\Theta=\varrho_i\}],\qquad z\in\RR^{n-m}, i\in I,
\end{equation}
where $\{\varrho_i:i\in I\}$ is the (at most countable) set of atoms of the direction distribution $\QQ_{n,m}$.
Then,
\begin{align*}
v(V_{n-1},W)  &= {\lim_{r\to\infty}r^{-(n+m)}{\var(V_{n-1}(Z\cap W_r))}} \\
& =  \gamma e^{-2\gamma m_1}\EE[(\gamma s_1V_{n-m}(\Xi)-V_{n-m-1}(\Xi))^2T(W,\Theta){\bf 1}\{\Theta\notin\{\varrho_i:i\in I\}\}]\\
&\quad+\gamma e^{-2\gamma m_1} \sum_{i\in I} T(W,\varrho_i) \\
& \quad \times \bigg( \gamma \int_{\RR^{n-m}} [e^{\gamma f(x,\varrho_i)} \big( s_1^2 - s_1 g(x,\varrho_i) + \frac{1}{4} g(x,\varrho_i) g(-x,\varrho_i) \big) - s_1^2]\, \sL^{n-m}(\dint x) \\
& \qquad  + \frac{1}{4} \EE \bigg[ \mathbf{1}\{ \Theta = \varrho_i \} \int_{\partial \Xi} \int_{\partial \Xi} e^{\gamma f(y-z,\varrho_i)} \, \sH^{n-m-1}(\dint y) \, \sH^{n-m-1}(\dint z) \bigg] \\
& \qquad   + \frac{3}{4} \EE \bigg[ \mathbf{1}\{ V_{n-m}(\Xi)=0, \Theta = \varrho_i \} \int_{\partial \Xi} \int_{\partial \Xi} e^{\gamma f(y-z,\varrho_i)} \, \sH^{n-m-1}(\dint y) \, \sH^{n-m-1}(\dint z) \bigg] \bigg).
\end{align*}
\end{corollary}

\begin{remark}\label{remark:v}\rm 
\begin{itemize}
	\item[(i)] As for $v(V_n,W)$ we note that in the case of a uniform Poisson cylinder process the constant $v(V_{n-1},W)$ can be rewritten as
	\begin{align*}
	v(V_{n-1},W) &= \gamma e^{-2\gamma m_1}\EE[(\gamma s_1V_{n-m}(\Xi)-V_{n-m-1}(\Xi))^2]{\kappa_m\over m+1}\int_{A(n,1)}\sL^1(L\cap W)^{m+1}\,\nu(\dint L).
	\end{align*}
	\item[(ii)] As explained in Remark \ref{rem:VarianceVOLUME} (iii) the expression for $v(V_{n-1},W)$ in the direction non-atomic case simplifies considerably if we take $W=\BB^n$. In this case
	$$
	v(V_{n-1},\BB^n) = \gamma m!\pi^{-m}\kappa_m^2\kappa_{n+m} e^{-2\gamma m_1}\EE[(\gamma s_1V_{n-m}(\Xi)-V_{n-m-1}(\Xi))^2].
	$$
\end{itemize}
\end{remark}

Given the explicit formula in Corollary \ref{cor:VarSurface} we can now construct two examples of Poisson cylinder {processes} for which the variance constant $v(V_{n-1},W)$ is {equal} to zero. We would like to highlight that this surprising property of Poisson cylinder processes is in sharp contrast to the corresponding result for the Boolean model, where it is known that the asymptotic variance constant is strictly positive {if the geometric functional applied to the typical grain is non-zero with positive probability,} see \cite[Theorem 22.9]{LP}.

\begin{example}[Direction non-atomic case] \label{Ex:DirectionNonAtomic} \rm
We let $n\in\NN$ and $m\in\{1,\ldots,n-1\}$ be arbitrary and choose the deterministic cylinder base $\Xi=[0,1]^{n-m}$. Then $V_{n-m}(\Xi)=1=m_1$, $V_{n-m-1}(\Xi)=n-m=s_1$ and hence
\begin{align*}
\EE[(\gamma s_1V_{n-m}(\Xi)-V_{n-m-1}(\Xi))^2T(W,\Theta)] = (n-m)^2 (\gamma-1)^2 \,\EE[T(W,\Theta)].
\end{align*}
We may choose now $\gamma=1>0$ to see that in this case the last expression is equal to zero. By Corollary \ref{cor:VarSurface}  this implies that $v(V_{n-1},W)=0$, independently of the direction distribution ${\QQ}_{n,m}$, as long as it is non-atomic. 
\end{example}

\begin{example}[Purely direction atomic case] \label{Ex:DirectionAtomic}\rm
We choose the dimension parameters $n\in\NN$ and $m\in\{1,\ldots,n-1\}$ in such a way that $n-m=1$ (that is, we choose $m=n-1$) and consider the deterministic cylinder base $\Xi=[-a,b]$ for some $a,b>0$ with length $\ell:=a+b$. Further, we suppose that the direction distribution is concentrated on some finite set $\{\varrho_i:i\in I\}\subset\SO_{n,m}$ with weights $p_i:=\QQ_{n,m}(\{\varrho_i\})$, $i\in I$, and that the index set $I$ has at least two elements, that is $|I|\geq 2$. In this situation the two functions $f(z)=f(z,\varrho_i)$ and $g(z)=g(z,\varrho_i)$, $i\in I$, $z\in\RR$, as defined in Corollary \ref{cor:VarSurface} are independent of the directions $\varrho_i$ and given by
\begin{align*}
f(z)=\begin{cases}
0 &: z\notin[-\ell,\ell]\\
z+\ell &: z\in[-\ell,0]\\
-z+\ell &: z\in(0,\ell]
\end{cases}
\qquad\text{and}\qquad
g(z)=\begin{cases}
0 &: z\notin[-\ell,\ell]\\
1 &: z\in[-\ell,\ell]\setminus\{0\}\\
2 &: z=0.
\end{cases}
\end{align*}
In particular $m_1=V_1(\Xi)=\ell$ and $s_1=V_0(\Xi)=1$. Plugging this into the formula for $v(V_{n-1},W)$ provided by Corollary \ref{cor:VarSurface}  we arrive at
\begin{align*}
v(V_{n-1},W) = \gamma e^{-2\ell\gamma}\sum_{i\in I}T(W,\varrho_i)(A_i+B)
\end{align*}
with $A_i:={p_i\over 4}\sum_{u,v\in\{-a,b\}}e^{\gamma f(u-v)}$ for $i\in I$ and $B:=\gamma\int_{\RR}[e^{\gamma f(z)}(1-{1\over 2}g(z))(1-{1\over 2}g(-z)) -1]\,\sL(\dint z)$. Using now the definitions of $f$ and $g$ we see that
\begin{align*}
A_i = {p_i\over 4}\Big(e^{\gamma f(0)}+e^{\gamma f(\ell)}+e^{\gamma f(-\ell)}+e^{\gamma f(0)}\Big) = {p_i\over 2}(1+e^{\ell\gamma}),\qquad i\in I,
\end{align*}
and
$$
B = \gamma\int_{-\ell}^0 \Big[{1\over 4}e^{\gamma(z+\ell)} - 1\Big]\,\sL(\dint z) + \gamma\int_{0}^{\ell} \Big[{1\over 4}e^{\gamma(-z+\ell)} - 1\Big]\,\sL(\dint z) 
= {1\over 2}(e^{\ell\gamma}-4\ell\gamma-1).
$$
Thus,
\begin{align*}
v(V_{n-1},W) = {\gamma\over 2} e^{{-}2\ell\gamma}\sum_{i\in I}T(W,\varrho_i)\big(e^{\ell\gamma}(1+p_i)+p_i-4\ell\gamma-1\big).
\end{align*}
Next, we notice that the expression $e^{\ell\gamma}(1+p)+p-4\ell\gamma-1$ is zero precisely for $p=p(\gamma)={1+4\ell\gamma-e^{\ell\gamma}\over 1+e^{\ell\gamma}}$. 
\begin{minipage}{0.3\columnwidth}
\begin{tikzpicture}[scale=2.5]
\draw[->] (0, 0) -- (1.3, 0) node[right] {$\gamma$};
\draw[->] (0, 0) -- (0, 0.7) node[left] {$p(\gamma)$};
\draw[dotted] (0.45,0.62) -- (0.45, 0) node[below] {$M$};
\draw[dashed] (0.2,0.44) -- (0.2, 0) node[below] {$\gamma_*$};
\draw[dashed] (0.2,0.44) -- (0, 0.44) node[left] {$1\over |I|$};
\draw[domain=0:1.17, smooth, variable=\x, blue] plot ({\x}, {(1+8*\x-exp(2*\x))/(1+exp(2*\x))});
\end{tikzpicture}
\end{minipage}
\begin{minipage}{0.7\columnwidth}
In what follows we denote by $W(x)$, $x\geq 0$, the inverse function of $z\mapsto ze^z$, $z>0$ (also known as Lambert W-function). One can now check that for any fixed $a,b>0$ the mapping $\gamma\mapsto p(\gamma)$ is continuous and strictly increasing on the interval $[0,M]$, where $M:={1\over 2\ell}(1+2W(1/\sqrt{e}))$, and that $p(0)=0$ and
$$
p(M)={3+4W(1/\sqrt{e})-e^{1/2+W(1/\sqrt{e})}\over 1+e^{1/2+W(1/\sqrt{e})}}\approx 0.619,
$$
independently of $a$ and $b$.
\end{minipage}
Thus, by the intermediate value theorem from calculus and since $|I|\geq 2$, we can find an intensity $\gamma_*=\gamma_*(I,a,b)\in[0,M]$ satisfying $p(\gamma^\ast)=1/|I|\leq 1/2<p(M)$.  As a consequence, for this intensity $\gamma_*$ we can choose $p_i=1/|I|$ such that $e^{\ell\gamma_*}(1+p_i)+p_i-4\ell\gamma_*-1=0$ for all $i\in I$. This in turn implies that $v(V_{n-1},W)=0$ in this case, independently of the precise choice of the directions $\varrho_i$, $i\in I$. On the other hand, since closed intervals are the only compact convex subsets of the real line, this example also shows that for $m=n-1$ we necessarily have that $v(V_{n-1},W)>0$ if $|I|=1$ (as long as $\sL^n(W)>0$).
\end{example}

The following result gives a sufficient condition {which ensures that the asymptotic variance constant is positive for intrinsic volumes of order greater or equal to $ m $} in the non-atomic case. We {emphasise} that this condition rules out the example just presented.

\begin{proposition}\label{prop:CovMatrixIntVolPositiveDefinite}
	{Fix $n\in\NN$ and $ m\in\{1,\ldots,n-1\}$, and consider a direction non-atomic Poisson cylinder process with rotation invariant base-direction distribution $ \QQ$. Suppose that the covariance matrix $C= \big(\Cov(V_i(\Xi),V_j(\Xi))\big)_{i,j={0}}^{{n-m}}$ of the intrinsic volumes of the typical cylinder base $\Xi$ exists and is positive definite. Then $v(V_k,W)>0$ for each $ k \in \{m,\ldots,n\} $.}
\end{proposition}

An example to which Proposition \ref{prop:CovMatrixIntVolPositiveDefinite} applies arises if $\Xi$ is a random dilatation of a fixed convex body, that is, if $\Xi=R\cdot K$ for a non-constant random variable $R\geq 0$ and a fixed convex body $K\in\cK(\RR^{n-m})$ {with $\sL^{n-m}(K)>0$}. In fact, assuming {that $R$ has a density and that} {$\EE[R^{2(n-m)}]<\infty$} we obtain for ${0}\leq i,j\leq {n-m}$ that
$$
\Cov(V_i(\Xi),V_j(\Xi)) = V_i(K)V_j(K)\,\Cov(R^i,R^j) = V_i(K)V_j(K)\,\big(\EE[R^{i+j}] - \EE[R^i]\,\EE[R^j]\big).
$$
{For $u=(u_0,\hdots,u_{n-m})\in\mathbb{R}^{n-m+1}$, we obtain that
$$
u^TCu=\EE\bigg[ \bigg(\sum_{\ell=0}^{n-m} u_\ell V_\ell(K) R^\ell\bigg)^2 \bigg] - \EE\bigg[  \sum_{\ell=0}^{n-m} u_\ell V_\ell(K) R^\ell \bigg]^2.
$$
By H\"older's inequality and the assumptions on the distribution of $R$ the last expression is non-zero.} This implies the positive definiteness of the matrix $C$ and hence positivity of the asymptotic variance constant for {the intrinsic volumes of order $m,\ldots,n$ and convex windows $W$.}

\begin{remark}\rm 
Similarly to Corollary \ref{cor:VarVolumeConstant} and Corollary \ref{cor:VarSurface} it is in principle possible to provide an explicit formula for $v(V_j,W)$ for all intrinsic volumes of order $j\in\{0,1,\ldots,n\}$ based on \cite[Lemma 5.1]{ConcentrationIneq} for uniform Poisson cylinder processes. However, the resulting expressions are rather involved and simplify nicely only for $j=n-1$ and $j=n$. This is the reason why we focussed on these two particular cases, which can already illustrate the phenomenon of a vanishing asymptotic variance constant.
\end{remark}

Finally, we provide a formula for the asymptotic covariance between the $n$-dimensional volume and the $(n-1)$-st intrinsic volume. {In case that the cylinder base has dimension $n-m$ almost surely and that we are in the non-atomic case or in the purely direction atomic case, the formula coincides up to the factor two with the formulas for the asymptotic covariances between the volume and the Hausdorff measure of $\partial Z$ in the interior of the observation window given in \cite[Theorem 3]{HeinrichSpiessCLTVundS} similarly as discussed before Corollary \ref{cor:VarSurface}.}

\begin{corollary}\label{cor:Covariance}
Fix $n\in\NN$ and $m\in\{1,\ldots,n-1\}$, and consider a Poisson cylinder process with intensity $\gamma\in(0,\infty)$ and base-direction distribution $\QQ$ with $\EE[V_j(\Xi)^2]<\infty$ for all $j\in\{1,\hdots,n-m\}$. Let $T(W,\theta)$ be as in \eqref{eq:DefTwtheta} and put $m_1:=\EE[V_{n-m}(\Xi)]$, $s_1:=\EE[V_{n-m-1}(\Xi)]$. Also, let $f(z,\varrho_i)$ be as in \eqref{eq:defFunctionF} and $g(z,\varrho_i)$ as in \eqref{eqn:Definition_g}, where $\{\varrho_i:i\in I\}$ is the (at most countable) set of atoms of the direction distribution $\QQ_{n,m}$.	
Then, 
\begin{align*}
& \lim_{r\to\infty} r^{-(n+m)} {\rm cov}(V_n(Z\cap W_r),V_{n-1}(Z\cap W_r)) \\
& = \gamma e^{-2\gamma m_1} \EE\Big[ V_{n-m}(\Xi) {\big(V_{n-m-1}(\Xi) - \gamma s_1 V_{n-m}(\Xi) \big)} T(W,\Theta) \mathbf{1}\{\Theta\notin\{\varrho_i:i\in I\}\} \Big] \\
& \quad + \gamma e^{-2\gamma m_1}  \sum_{i\in I}   T(W,\varrho_i)\int_{\RR^{n-m}} {\Big[s_1 - e^{\gamma f(x,\varrho_i)} \Big( s_1 - \frac{g(x,\varrho_i)}{2} \Big) \Big]} \, \sL^{n-m}(\dint x) .
\end{align*}
\end{corollary}

\section{Proofs I: Central limit theorem for the volume}\label{proof:CLTVolume}

\subsection{Preparations}

We prepare the proof of Theorem \ref{thm:Volume} with three results. The first one is a translative integral formula for cylinders, which is repeatedly applied in this paper. It {generalises} the well-known translative integral formula
\begin{equation}\label{eq:TranslativeClassical}
\int_{\RR^{n-m}}\sL^{n-m}((X+x)\cap W)\,\sL^{n-m}(\dint x) = \sL^{n-m}(X)\sL^{n-m}(W),\qquad X,W\in\cC(\RR^{n-m}),
\end{equation}
stated in \cite[Theorem 5.2.1]{SW}, which is included as special case $m=0$ (and with $n$ formally replaced by $n-m$). A version for cylinders with a convex base can be found in \cite[Theorem 2]{SchneiderWeil86}.

\begin{proposition}\label{prop:IntegralFormulaVolume}
Let $n\in\NN$ and $m\in\{0,1,\ldots,n-1\}$. Let $X\in\cC(\RR^{n-m})$ and $W\in\cC(\RR^n)$. Then, for any $\theta\in\SO_{n,m}$,
$$
\int_{\RR^{n-m}}\sL^n(Z(x,\theta,X)\cap W)\,\sL^{n-m}(\dint x) = \sL^{n-m}(X)\sL^n(W).
$$
\end{proposition}
\begin{proof}
Using Fubini's theorem we write
\begin{align*}
&\int_{\RR^{n-m}}\sL^n(Z(x,\theta,X)\cap W)\,\sL^{n-m}(\dint x)\\
&= \int_{\RR^{n-m}}\int_{\RR^n}{\bf 1}\{y\in W\}{\bf 1}\{y\in Z(x,\theta,X)\}\,\sL^n(\dint y)\,\sL^{n-m}(\dint x)\\
&= \int_{\RR^n}{\bf 1}\{y\in W\}\int_{\RR^{n-m}}{\bf 1}\{y\in Z(x,\theta,X)\}\,\sL^{n-m}(\dint x)\,\sL^n(\dint y).
\end{align*}
Next we notice that ${\bf 1}\{y\in Z(x,\theta,X)\}=1$ if and only if $y\in\theta((X+x)\times\EE^m)$, which in turn is equivalent to $\theta^Ty\in(X+x)\times\EE^m$. {This leads to}
\begin{align*}
&\int_{\RR^{n-m}}{\bf 1}\{y\in Z(x,\theta,X)\}\,\sL^{n-m}(\dint x) \\
&\qquad\qquad= \int_{\RR^{n-m}}{\bf 1}\{\theta^Ty\in (X+x)\times\EE^m\}\,\sL^{n-m}(\dint x) = \sL^{n-m}(X),
\end{align*}
independently of $\theta$ and $y$. Thus {we obtain}
\begin{align*}
&\int_{\RR^{n-m}}\sL^n(Z(x,\theta,X)\cap W)\,\sL^{n-m}(\dint x)\\
&\qquad\qquad = \sL^{n-m}(X)\int_{\RR^n}{\bf 1}\{y\in W\}\,\sL^n(\dint y) = \sL^{n-m}(X)\sL^n(W)
\end{align*}
and the proof is complete.
\end{proof}

We shall next compute the difference operators of the Poisson functional $F:=\sL^n(Z\cap W)$. 
A version of this result for the Boolean model (which is included as the special case $m=0$) can be found in \cite[Lemma 3.3]{HugLastSchulteBM} or \cite[Lemma 22.6]{LP}.

\begin{lemma}\label{lem:DiffOperatorsVolume}
Let $W\in\cC(\RR^n)$, $F:=\sL^n(Z\cap W)$ and fix $k\in\NN$. Then $\PP$-almost surely and for $(\sL^{n-m}\otimes\QQ)^k$-almost all $(x_1,\theta_1,X_1),\ldots,(x_k,\theta_k,X_k)$ one has that
\begin{align*}
\sfD_{(x_1,\theta_1,X_1),\ldots,(x_k,\theta_k,X_k)}^k F = (-1)^k\bigg[\sL^n\Big(Z\cap\bigcap_{i=1}^kZ(x_i,\theta_i,X_i)\cap W\Big)-\sL^n\Big(\bigcap_{i=1}^kZ(x_i,\theta_i,X_i)\cap W\Big)\bigg].
\end{align*}
\end{lemma}
\begin{proof}
 We recall that for $(x,\theta,X)\in\RR^{n-m}\times\MM_{n,m}$ the first-order difference operator $\sfD_{(x,\theta,X)}F$ is given by
$$
\sfD_{(x,\theta,X)}F = \sL^n((Z\cup Z(x,\theta,X))\cap W) - \sL^n(Z\cap W).
$$
This representation and the additivity of $\sL^n$ imply that
\begin{align*}
\sfD_{(x,\theta,X)}F = \sL^n(Z(x,\theta,X)\cap W) - \sL^n(Z\cap Z(x,\theta,W)\cap W)
\end{align*}
$\PP$-almost surely and for $(\sL^{n-m}\otimes\QQ)$-almost all $(x,\theta,X)$. Similarly, one has that {$\PP$-almost} surely and for $(\sL^{n-m}\otimes\QQ)^2$-almost all $(x_1,\theta_1,X_1),(x_2,\theta_2,X_2)\in\RR^{n-m}\times\MM_{n,m}$,
{\begin{align*}
&\sfD_{(x_1,\theta_1,X_1),(x_2,\theta_2,X_2)}^2F = \sfD_{(x_1,\theta_1,X_1)}\sfD_{(x_2,\theta_2,X_2)}F \\
&= (\sL^n(Z(x_2,\theta_2,X_2)\cap W) - \sL^n((Z\cup Z(x_1,\theta_1,X_1))\cap Z(x_2,\theta_2,X_2)\cap W) ) \\
& \qquad\qquad - (\sL^n(Z(x_2,\theta_2,X_2)\cap W)- \sL^n(Z\cap Z(x_2,\theta_2,X_2)\cap W)) \\
&= \sL^n(Z\cap Z(x_2,\theta_2,X_2)\cap W) - \sL^n((Z\cup Z(x_1,\theta_1,X_1))\cap Z(x_2,\theta_2,X_2)\cap W) \\
&= \sL^n(Z \cap Z(x_1,\theta_1,X_1) \cap Z(x_2,\theta_2,X_2)\cap W) - \sL^n(Z(x_1,\theta_1,X_1) \cap Z(x_2,\theta_2,X_2)\cap W).
\end{align*}}
Iterating this argument leads to the result {for general $k\in\NN$.}
\end{proof}

\begin{remark}\label{rem:DiffOperatorAdditiveFct}\rm 
An analysis of the proof of the previous lemma shows that the result remains true if the $n$-dimensional Lebesgue measure $\sL^n$ is replaced by an arbitrary additive functional {(see \eqref{eq:Additivity})}. 
In Section \ref{sec:ProofCLTAdditiveFunctionals} below we shall use Lemma \ref{lem:DiffOperatorsVolume} in this form.
\end{remark}

\begin{corollary}\label{cor:DiffOpVolumeBound}
One has that $\PP$-almost surely and for $(\sL^{n-m}\otimes\QQ)$-almost all $(x,\theta,X)$,
$$
|\sfD_{(x,\theta,X)}F| \leq \sL^n(Z(x,\theta,X)\cap W).
$$
Moreover, one has that {$\PP$-almost} surely and for $(\sL^{n-m}\otimes\QQ)^2$-almost all $(x_1,\theta_1,X_1),(x_2,\theta_2,X_2)$, 
$$
|\sfD^2_{(x_1,\theta_1,X_1),(x_2,\theta_2,X_2)}F| \leq \sL^n(Z(x_1,\theta_1,X_1)\cap Z(x_2,\theta_2,X_2)\cap W).
$$
\end{corollary}

\subsection{Proof of Theorem \ref{thm:Volume}: Wasserstein bound}\label{sec:WassersteinBoundVolume}

Recall that $F:=\sL^n(Z\cap W)$ and introduce the abbreviations $F_r:=\sL^n(Z\cap W_r)$ and $G_r:={F_r-\EE[F_r]\over\sqrt{\var(F_r)}}$. From Corollary \ref{cor:DiffOpVolumeBound} it follows that {$\PP$-almost surely}
\begin{align}
|\sfD_{(x,\theta,X)} G_r| &\leq {\sL^n(Z(x,\theta,X)\cap W_r)\over\sqrt{\var(F_r)}}, \label{eq:BoundD1}\\
|\sfD_{(x_1,\theta_1,X_1),(x_2,\theta_2,X_2)}G_r| &\leq {\sL^n(Z(x_1,\theta_1,X_1)\cap Z(x_2,\theta_2,X_2)\cap W_r)\over\sqrt{\var(F_r)}} \label{eq:BoundD2}
\end{align}
for $(\sL^{n-m}\otimes\QQ)$-almost all $(x,\theta,X)$ and $(\sL^{n-m}\otimes\QQ)^2$-almost all $(x_1,\theta_1,X_1),(x_2,\theta_2,X_2)$, respectively. Using these bounds, the fact that
\begin{align}\label{eq:CylinderVolBoundBase}
\sL^n(Z(x,\theta,X)\cap W_r) \leq\sL^{n-m}(X)\,\diam(W_r)^m = r^m\,\sL^{n-m}(X)\,\diam(W)^m
\end{align}
and Fubini's theorem we find {for $\alpha_{G_r,1}$ as defined in Section~\ref{subsec:SecondOrderPoincareUnivariate} that}
\begin{equation}\label{eq:BoundAlpha1}
\begin{split}
\alpha_{G_r,1}^2 &\leq {4\gamma^3\over(\var(F_r))^2}\int_{(\RR^{n-m})^3}\int_{\MM_{n,m}^3}\sL^n(Z(x_1,\theta_1,X_1)\cap W_r)\sL^n(Z(x_2,\theta_2,X_2)\cap W_r)\\
&\qquad\times \sL^n(Z(x_1,\theta_1,X_1)\cap Z(x_3,\theta_3,X_3)\cap W_r)\sL^n(Z(x_2,\theta_2,X_2)\cap Z(x_3,\theta_3,X_3)\cap W_r)\\
&\qquad\times\QQ^3(\dint((\theta_1,X_1),(\theta_2,X_2),(\theta_3,X_3)))\,(\sL^{n-m})^3(\dint(x_1,x_2,x_3))\\
&\leq {4\gamma^3\diam(W)^{2m}r^{2m}\over(\var(F_r))^2}\int_{\MM_{n,m}^3}\int_{(\RR^{n-m})^3}\sL^{n-m}(X_1)\sL^{n-m}(X_2)\\
&\qquad\times \sL^n(Z(x_1,\theta_1,X_1)\cap Z(x_3,\theta_3,X_3)\cap W_r)\sL^n(Z(x_2,\theta_2,X_2)\cap Z(x_3,\theta_3,X_3)\cap W_r)\\
&\qquad\times(\sL^{n-m})^3(\dint(x_1,x_2,x_3))\,\QQ^3(\dint((\theta_1,X_1),(\theta_2,X_2),(\theta_3,X_3))).
\end{split}
\end{equation}
Using the translative integral formula from Proposition \ref{prop:IntegralFormulaVolume} we can carry out the integration with respect to {$x_1$ and $x_2$, {i.e.}
\begin{align*}
&\int_{\RR^{n-m}}\sL^n(Z(x_i,\theta_i,X_i)\cap Z(x_3,\theta_3,X_3)\cap W_r)\,\sL^{n-m}(\dint x_i) \\
& \qquad = \sL^{n-m}(X_i)\sL^n(Z(x_3,\theta_3,X_3)\cap W_r)
\end{align*}
for $i\in\{1,2\}$.}
{Thus, we deduce the bound}
\begin{align*}
\alpha_{G_r,1}^2 &\leq {4\gamma^3\diam(W)^{2m}r^{2m}\over(\var(F_r))^2}\int_{\MM_{n,m}^3}\sL^{n-m}(X_1)^2\sL^{n-m}(X_2)^2\\
&\qquad\qquad\times\int_{\RR^{n-m}}\sL^n(Z(x_3,\theta_3,X_3)\cap W_r)^2\,\sL^{n-m}(\dint x_3)\,\QQ^3(\dint((\theta_1,X_1),(\theta_2,X_2),(\theta_3,X_3))).
\end{align*}
We apply now \eqref{eq:CylinderVolBoundBase} and once again the translative integral formula to arrive at
\begin{align*}
\alpha_{G_r,1}^2 &\leq {4\gamma^3\diam(W)^{3m}r^{3m}\over(\var(F_r))^2}\int_{\MM_{n,m}^3}\sL^{n-m}(X_1)^2\sL^{n-m}(X_2)^2\sL^{n-m}(X_3)\\
&\qquad\qquad\times\int_{\RR^{n-m}}\sL^n(Z(x_3,\theta_3,X_3)\cap W_r)\,\sL^{n-m}(\dint x_3)\,\QQ^3(\dint((\theta_1,X_1),(\theta_2,X_2),(\theta_3,X_3)))\\
& = {4\gamma^3\diam(W)^{3m}\sL^n(W)m_2^3r^{3m+n}\over(\var(F_r))^2}.
\end{align*}
Next, we choose $r_0\in(0,\infty)$ such that $\var(F_r)\geq \frac{{\underline{c}_v}}{2}r^{n+m}$ for all $ r \geq r_0 $, which is possible according to Proposition \ref{prop:ExpectationVarianceVolume} (iii). Then, we obtain
\begin{align}\label{eq:Alpha1Volume}
\alpha_{G_r,1} \leq 4{\underline{c}_v^{-1}} \gamma^{3/2}\diam(W)^{3m/2}\sL^n(W)^{1/2} {m_2^{3/2}} \,r^{-\frac{n-m}{2}}
\end{align}
for $r\geq r_0$. Combining the bounds for the difference operators in \eqref{eq:BoundD1} and \eqref{eq:BoundD2} with the definition of $\alpha_{G_r,2}$, we see that the terms on the right-hand side of \eqref{eq:BoundAlpha1} bound {$4\alpha_{G_r,2}^2$} as well so that
\begin{align}\label{eq:Alpha2Volume}
\alpha_{G_r,2} \leq  2{\underline{c}_v^{-1}}\gamma^{3/2}\diam(W)^{3m/2}\sL^n(W)^{1/2} {m_2^{3/2}} \,r^{-\frac{n-m}{2}}
\end{align}
for $ r \geq r_0 $.

{Instead of dealing directly with $\alpha_{G_r,3}$, we consider the more general term
$$
T_{p,q}:= \frac{\gamma}{(\var(F_r))^{pq/2}} \int_{\RR^{n-m}}\int_{\MM_{n,m}} \mathbb{E}\big[|\sfD_{(x,\theta,X)} {F_r}|^p\big]^q \,\QQ(\dint(\theta,X))\,\sL^{n-m}(\dint x)
$$
for $p,q>0$ with {$pq\geq 1$.} It follows from Corollary~\ref{cor:DiffOpVolumeBound} that
\begin{align*}
T_{p,q} \leq {\gamma\over(\var(F_r))^{pq/2}}\int_{\RR^{n-m}}\int_{\MM_{n,m}}\sL^n(Z(x,\theta,X)\cap W_r)^{pq} \,\QQ(\dint(\theta,X))\,\sL^{n-m}(\dint x).
\end{align*}
Using \eqref{eq:CylinderVolBoundBase} together with Fubini's theorem yields
\begin{align*}
T_{p,q} &\leq {\gamma\, r^{(pq-1)m}\,\diam(W)^{(pq-1)m}\over(\var(F_r))^{pq/2}}\int_{\MM_{n,m}}\sL^{n-m}(X)^{pq-1}\\
&\qquad\qquad\times\int_{\RR^{n-m}}\sL^n(Z(x,\theta,X)\cap W_r)\,\sL^{n-m}(\dint x)\,\QQ(\dint(\theta,X)).
\end{align*}
Next, we apply the translative integral formula from Proposition \ref{prop:IntegralFormulaVolume} to the inner integral to deduce that
\begin{align*}
T_{p,q} &\leq {\gamma\,r^{(pq-1)m+n}\,\diam(W)^{(pq-1)m}\,\sL^n(W)\over(\var(F_r))^{pq/2}}\,m_{pq}.
\end{align*}
Inserting the lower variance bound for $F_r$ we find that, for $ r\geq r_0 $,
\begin{equation}\label{eq:TpqVolume}
\begin{split}
T_{p,q} &\leq {\gamma\,r^{(pq-1)m+n}\,\diam(W)^{(pq-1)m}\,\sL^n(W)\over(\frac{1}{2} {\underline{c}_v}\, r^{n+m})^{pq/2}}\,m_{pq} \\
&  = 2^{pq/2} \gamma{\underline{c}_v^{-pq/2}}\diam(W)^{(pq-1)m}\sL^n(W)m_{pq}\, r^{(pq/2-1)m+(1-pq/2)n}.
\end{split}
\end{equation}}
{For $p=3$ and $q=1$ we obtain
\begin{equation}\label{eq:Alpha3Volume}
\alpha_{G_r,3} \leq 2^{3/2} \gamma {\underline{c}_v^{-3/2}}\diam(W)^{2m}\sL^n(W)m_{3}\, r^{-\frac{n-m}{2}}
\end{equation}
for $r\geq r_0$.} Putting together \eqref{eq:Alpha1Volume}, \eqref{eq:Alpha2Volume} and \eqref{eq:Alpha3Volume} we deduce from Proposition \ref{prop:CLTWasserstein} that
\begin{align*}
&d_W(G_r,N) \\
&\leq \alpha_{G_r,1}+\alpha_{G_r,2}+\alpha_{G_r,3}\\
& \leq \big({6}\gamma^{3/2} {\underline{c}_v^{-1}} \diam(W)^{3m/2}\sL^n(W)^{1/2}m_3^{3/2}+2^{3/2} \gamma {\underline{c}_v^{-{3/2}}} \diam(W)^{2m}\sL^n(W)m_3\big)\,r^{-{n-m\over 2}}
\end{align*}
for any $r\geq r_0$. This completes the proof of Theorem \ref{thm:Volume} (i).\qed

\subsection{Proof of Theorem \ref{thm:Volume}: Kolmogorov bound}\label{sec:KolmogorovBoundVolume}

To prove the second part of Theorem~\ref{thm:Volume} we need to bound the three addtional terms $\alpha_{G_r,4},\alpha_{G_r,5}$ and $ \alpha_{G_r,6}$ appearing in Proposition~\ref{prop:CLTKolmogorov}. For $\alpha_{G_r,4}$ we start by dealing with the fourth moment of $G_r$. Due to \eqref{eq:4thMomentBound} we know that 
\begin{equation}\label{eq:4thMomentFromLPS}
\begin{split}
\EE\left[G_r^4\right]\leq \max\Bigg\{&256\left(\gamma\int_{\RR^{n-m}\times \MM_{n,m}}\EE\left[(\sfD_{(x,\theta,X)}G_r)^4\right]^{\frac{1}{2}}(\sL^{n-m}\otimes\QQ)(\dint(x,\theta,X))\right)^2,\\
&\qquad 4\gamma\int_{\RR^{n-m}\times \MM_{n,m}}\EE\left[(\sfD_{(x,\theta,X)}G_r)^4\right](\sL^{n-m}\otimes\QQ)(\dint(x,\theta,X))+2\Bigg\}.
\end{split}
\end{equation}
{It follows from \eqref{eq:TpqVolume} with $p=4$ and $q=1/2$ or $q=1$ that
\begin{align*}
\EE\left[G_r^4\right]\leq \max\Bigg\{ 1024 \gamma^2 m_{2}^2 {\underline{c}_v^{-2}} \diam(W)^{2m}\sL^n(W)^2, 16 \gamma {\underline{c}_v^{-2}} \diam(W)^{3m}\sL^n(W)m_4\, r^{m-n}+2 \Bigg\}
\end{align*} 
for $r\geq r_0$, where $r_0$ is as in the proof of part (i). We thus conclude that, for $r\geq r_0$, $\frac{1}{2}\EE[G_r^4]^{\frac{1}{4}}$ is bounded by a constant $c\in(0,\infty)$ say, which only depends on the parameters mentioned in {Remark \ref{rm:ConstantVolume}}.}

{We can now proceed to bound $\alpha_{G_r,4}$. From \eqref{eq:TpqVolume} with $p=4$ and $q=3/4$ we obtain
\begin{equation}\label{eq:Alpha4Volume}
\alpha_{G_r,4} \leq c \gamma\,2^{3/2} {\underline{c}_v^{-3/2}}\diam(W)^{2m}\sL^n(W)m_{3}\, r^{-\frac{n-m}{2}} \leq c_4\,r^{-\frac{n-m}{2}}
\end{equation}
for $r\geq r_0$, where $c_4\in(0,\infty)$ is a constant only depending on the parameters mentioned in {Remark \ref{rm:ConstantVolume}}. For $\alpha_{G_r,5}$, \eqref{eq:TpqVolume} with $p=4$ and $q=1$ leads to
\begin{equation} \label{eq:Alpha5Volume}
\alpha_{G_r,5}^2 \leq 4 \gamma {\underline{c}_v^{-2}} \diam(W)^{3m}\sL^n(W)m_{4} r^{m-n} \leq c_5 r^{m-n}
\end{equation}
for $r\geq r_0$ with a constant $c_5$.}

Using {\eqref{eq:BoundD1}, \eqref{eq:BoundD2},} \eqref{eq:CylinderVolBoundBase}, Proposition \ref{prop:IntegralFormulaVolume} {and the lower variance bound for $F_r$ (see Proposition \ref{prop:ExpectationVarianceVolume} (iii))} we find for {$\alpha_{G_r,6}^2$} that
\begin{align}
& {\alpha_{G_r,6}^2 }\leq \frac{{9}\gamma^2}{\var(F_r)^2}\int_{\MM_{n,m}^2}\int_{(\RR^{n-m})^2}\sL^n(Z(x_1,\theta_1,X_1)\cap W_r)^2 \notag \\
&\phantom{\leq \frac{1}{\var(F_r)^2}\int_{\MM_{n,m}^2}\int_{(\RR^{n-m})^2}}\times\sL^n(Z(x_1,\theta_1,X_1)\cap Z(x_2,\theta_2,X_2)\cap W_r)^2 \notag \\
&\phantom{\leq \frac{{9}}{\var(F_r)^2}\int_{\MM_{n,m}^2}\int_{(\RR^{n-m})^2}}\times (\sL^{n-m})^2(\dint(x_1,x_2))\QQ^2(\dint((\theta_1,X_1),(\theta_2,X_2))) \notag \\
&\leq \frac{{9}\gamma^2\diam(W)^{2m}\,r^{2m} }{\var(F_r)^2}\int_{\MM_{n,m}^2}\int_{(\RR^{n-m})^2}\sL^n(Z(x_1,\theta_1,X_1)\cap W_r)\sL^{n-m}(X_1) \sL^{n-m}(X_2) \notag \\
&\phantom{\leq \frac{{9}}{\var(F_r)^2}\int_{\MM_{n,m}^2}\int}\times\sL^n(Z(x_1,\theta_1,X_1)\cap Z(x_2,\theta_2,X_2)\cap W_r) \notag \\
&\phantom{\leq \frac{1}{\var(F_r)^2}\int_{\MM_{n,m}^2}\int}\times(\sL^{n-m})^2(\dint(x_1,x_2))\QQ^2(\dint((\theta_1,X_1),(\theta_2,X_2))) \notag \\
& = \frac{{9}\gamma^2\diam(W)^{2m}r^{2m} }{\var(F_r)^2}\int_{\MM_{n,m}^2}\int_{\RR^{n-m}}\sL^n(Z(x_1,\theta_1,X_1)\cap W_r)^2\sL^{n-m}(X_1) \sL^{n-m}(X_2)^2 \notag \\
&\phantom{\leq \frac{1}{\var(F_r)^2}\int_{\MM_{n,m}^2}\int_{(\RR^{n-m})^2}}\times\sL^{n-m}(\dint x_1)\QQ^2(\dint((\theta_1,X_1),(\theta_2,X_2))) \notag \\
&\leq \frac{{9}\gamma^2\sL^n(W)r^{3m+n}\diam(W)^{3m} m_2m_3}{{(\frac{\underline{c}_v}{2} r^{n+m})^2}} \leq c_{6}\, r^{m-n} \label{eq:Alpha6Volume}
\end{align}
for $ r \geq r_0 $ and where $c_{6}\in(0,\infty)$ is a constant which only depends on the parameters mentioned in {Remark \ref{rm:ConstantVolume}}. 
Plugging now \eqref{eq:Alpha1Volume}, \eqref{eq:Alpha2Volume}, \eqref{eq:Alpha3Volume}, \eqref{eq:Alpha4Volume}, \eqref{eq:Alpha5Volume} and \eqref{eq:Alpha6Volume} into Proposition \ref{prop:CLTKolmogorov} shows that there exists a constant $C\in(0,\infty)$ only depending on the parameters mentioned in {Remark \ref{rm:ConstantVolume}} such that, for $ r \geq r_0 $,
$$
d_K(G_r,N)\leq Cr^{-{n-m\over 2}}.
$$
This completes the proof of Theorem~\ref{thm:Volume} (ii).\qed

\section{Proofs II: Central limit {theorems} for geometric functionals}\label{sec:ProofCLTAdditiveFunctionals}

\subsection{Preparations}

For $A\in\cK(\RR^n)$ we denote by $\eta([A])$ the number of cylinders of $\eta$ that intersect $A$. Let us also recall that the number of $k$-dimensional faces of the $n$-dimensional cube is given by $2^{n-k}{n\choose k}$, $k\in\{0,1,\ldots,n\}$. Hence, the total number of faces of an $n$-dimensional cube is $3^n$, a constant which will repeatedly appear below. The following two lemmas {generalise} the ideas of \cite[Lemma 3.2]{HugLastSchulteBM} and \cite[Proposition 22.4]{LP} from the case of the Boolean model ($m=0$) {to general Poisson cylinder processes} (with $m\geq 1$). Recall from \eqref{eq:LocalBoundedness} {that
$$
M(\varphi):=\sup\{|\varphi(K)|:K\in\cK(\RR^n),K\subseteq[0,1]^n\}
$$
for a geometric functional $\varphi: \cR(\RR^n)\to\RR$.}

\begin{lemma}\label{lem:BoundDifferenceOperator}
For a geometric functional $\varphi: \cR(\RR^n)\to\RR$, $W\in \cK(\RR^n)$, {$k\in\NN$} and {$(\sL^{n-m}\otimes\QQ)^{{k}}$}-almost all $(x_1,\theta_1,X_1),\hdots,(x_k,\theta_k,X_k)\in \RR^{n-m}\times\MM_{n,m}$,
\begin{align*}
& \big| \sfD^k_{(x_1,\theta_1,X_1),\hdots,(x_k,\theta_k,X_k)}\varphi(Z\cap W) \big| \\
& \qquad\leq 3^n\, M(\varphi) \sum_{z\in\mathbb{Z}^n} \mathbf{1}\Big\{(z+[0,1]^n)\cap \bigcap_{i=1}^k Z(x_i,\theta_i,X_i)\cap  W\neq \varnothing\Big\} 2^{\eta([z+[0,1]^n])}
\end{align*}
holds $\PP$-almost surely.
\end{lemma}

\begin{proof}
It follows from Lemma~\ref{lem:DiffOperatorsVolume} for additive functionals (recall Remark \ref{rem:DiffOperatorAdditiveFct}) that $\PP$-almost surely and for $(\sL^{n-m}\otimes\QQ)^k$-almost all $(x_1,\theta_1,X_1),\hdots,(x_k,\theta_k,X_k)\in \RR^{n-m}\times\MM_{n,m}$,
\begin{align*}
&\sfD^k_{(x_1,\theta_1,X_1),\hdots,(x_k,\theta_k,X_k)}\varphi(Z\cap W) \\
&\qquad = (-1)^k \Big[ \varphi\Big(Z\cap \bigcap_{i=1}^k Z(x_i,\theta_i,X_i)\cap W\Big) - \varphi\Big(\bigcap_{i=1}^k Z(x_i,\theta_i,X_i)\cap W\Big) \Big].
\end{align*}
For a set $K\in\cK(\RR^n)$ that is contained in a translate of $[0,1]^n$ let $Z_1,\hdots,Z_{\eta([K])}$ denote the cylinders of $\eta$ that hit $K$. {The additivity (see \eqref{eq:Additivity}), the translation invariance (see \eqref{eq:TranslationInvariance}) and the local boundedness (see \eqref{eq:LocalBoundedness})} of $\varphi$ yield
\begin{equation}\label{eq:BoundVarphi}
|\varphi(Z\cap K)| = \bigg| \sum_{\mathcal{I}\subseteq \{1,\hdots,\eta([K])\}} (-1)^{|\mathcal{I}|-1} \varphi\Big(\bigcap_{i\in \mathcal{I}} Z_i\cap K\Big)  \bigg| \leq (2^{\eta([K])}-1) M(\varphi).
\end{equation}
Let $A\in\cK(\RR^n)$ and define $\mathcal{Q}(A):=\{z+[0,1]^n: z\in\mathbb{Z}^n, (z+[0,1]^n)\cap A\neq \varnothing \}$. By additivity of $\varphi$ we obtain that
\begin{equation}\label{eq:PhiZintesectA}
\varphi(Z\cap A) = \sum_{\mathcal{Q}\subseteq\mathcal{Q}(A)} (-1)^{|\mathcal{Q}|-1} \mathbf{1}\Big\{\bigcap_{Q\in\mathcal{Q}}Q\neq \varnothing\Big\} \varphi\Big(Z\cap A \cap \bigcap_{Q\in\mathcal{Q}} Q\Big)
\end{equation}
so that \eqref{eq:BoundVarphi} leads to
$$
|\varphi(Z\cap A)|\leq M(\varphi) \sum_{\mathcal{Q}\subseteq\mathcal{Q}(A)} \mathbf{1}\Big\{\bigcap_{Q\in\mathcal{Q}}Q\neq \varnothing \Big\} (2^{\eta([\bigcap_{Q\in\mathcal{Q}} Q])}-1).
$$
Similarly, we have
\begin{align*}
|\varphi(A)| \leq \sum_{\mathcal{Q}\subseteq\mathcal{Q}(A)} \mathbf{1}\Big\{\bigcap_{Q\in\mathcal{Q}}Q\neq \varnothing\Big\} \Big|\varphi\Big(A\cap \bigcap_{Q\in\mathcal{Q}} Q\Big)\Big| \leq M(\varphi) \sum_{\mathcal{Q}\subseteq\mathcal{Q}(A)} \mathbf{1}\Big\{\bigcap_{Q\in\mathcal{Q}}Q\neq \varnothing\Big\}.
\end{align*}
Since each cube of $\mathcal{Q}(A)$ is contained in at most $3^n$ intersections (recall that $3^n$ is the total number of faces of an $n$-dimensional cube), we obtain that
\begin{align} \label{eq:BoundVarphiWithoutZ}
|\varphi(Z\cap A)| \leq 3^n\, M(\varphi) \sum_{Q\in\mathcal{Q}(A)} (2^{\eta([Q])}-1)  \qquad \text{and} \qquad |\varphi(A)| \leq 3^n\, M(\varphi) |\mathcal{Q}(A)|,
\end{align}
which proves the assertion.
\end{proof}

To proceed, we define 
$$
c_{k,\gamma,\QQ} := \EE[ 2^{k\eta([[0,1]^n])}]=e^{\EE[\eta([[0,1]^n])] (2^k-1)} \quad \text{for} \quad k\in\mathbb{N} \quad \text{and} \quad c_{\gamma,\QQ}:=c_{4,\gamma,\QQ}=e^{15\EE[\eta([[0,1]^n])]},
$$
{and note that
\begin{align*}
\EE[\eta([[0,1]^n])] &= \gamma\int_{\MM_{n,m}}\int_{\RR^{n-m}}{\bf 1}\{Z(x,\theta,X)\cap[0,1]^n\neq\varnothing\}\,\sL^{n-m}(\dint x)\QQ(\dint(\theta,X))\\
&=\gamma\int_{\MM_{n,m}}\int_{\RR^{n-m}}{\bf 1}\{(X+x)\times\EE^m\cap\theta^T[0,1]^n\neq\varnothing\}\,\sL^{n-m}(\dint x)\QQ(\dint(\theta,X))\\
&=\gamma\int_{\MM_{n,m}}\int_{\RR^{n-m}}{\bf 1}\{(X+x)\cap\Pi(\theta^T[0,1]^n)\neq\varnothing\}\,\sL^{n-m}(\dint x)\QQ(\dint(\theta,X))\\
&=\gamma\int_{\MM_{n,m}}\sL^{n-m}(\Pi(\theta^T[0,1]^n)+(-X))\,\QQ(\dint(\theta,X))\\
&=\gamma\EE[\sL^{n-m}(\Pi(\Theta^T[0,1]^n)+(-\Xi))],
\end{align*}
where we recall that $\Pi:\RR^n\to\RR^{n-m}$ stands for the orthogonal projection onto the first $n-m$ coordinates.}

\begin{lemma}\label{lem:BoundsExpectationDiffOp}
Let $\varphi$ be a geometric functional and $W\in\cK(\RR^n)$. Then, for $(x_1,\theta_1,X_1),(x_2,\theta_2,X_2),$ $(x_3,\theta_3,X_3)\in \RR^{n-m}\times\mathbb{M}_{n,m}$ and $k\in\mathbb{N}$,
\begin{align*}
(i)\quad& \EE[(\sfD_{(x_1,\theta_1,X_1)}\varphi(Z\cap W))^2 (\sfD_{(x_2,\theta_2,X_2)}\varphi(Z\cap W))^2 ] \\
& \quad\leq 81^{n} \, c_{\gamma,\QQ} \,M(\varphi)^4\,  \sL^n( Z(x_1,\theta_1,X_1^{\sqrt{n}}) \cap W^{\sqrt{n}} )^2\, \sL^n( Z(x_2,\theta_2,X_2^{\sqrt{n}}) \cap W^{\sqrt{n}} )^2,\\[2mm]
(ii)\quad & \EE[(\sfD^2_{(x_1,\theta_1,X_1),(x_3,\theta_3,X_3)}\varphi(Z\cap W))^2 (\sfD^2_{(x_2,\theta_2,X_2),(x_3,\theta_3,X_3)}\varphi(Z\cap W))^2 ] \\
& \quad\leq 81^n \, c_{\gamma,\QQ} \,M(\varphi)^4   \,\sL^n( Z(x_1,\theta_1,X_1^{\sqrt{n}}) \cap Z(x_3,\theta_3,X_3^{\sqrt{n}}) \cap W^{\sqrt{n}} )^2\\
& \quad \quad \times \sL^n( Z(x_2,\theta_2,X_2^{\sqrt{n}})\cap Z(x_3,\theta_3,X_3^{\sqrt{n}}) \cap W^{\sqrt{n}} )^2,\\[2mm]
(iii)\quad& \EE[(\sfD^2_{(x_1,\theta_1,X_1),(x_2,\theta_2,X_2)}\varphi(Z\cap W))^4]\\
& \quad\leq 81^n\, c_{\gamma,\QQ} \,  M(\varphi)^4\, \sL^n( Z(x_1,\theta_1,X_1^{\sqrt{n}}) \cap Z(x_2,\theta_2,X_2^{\sqrt{n}}) \cap W^{\sqrt{n}} )^4,\\[2mm]
(iv)\quad &
\EE[|\sfD_{(x_1,\theta_1,X_1)}\varphi(Z\cap W)|^k] \leq 3^{kn}\, c_{k,\gamma,\QQ} \, M(\varphi)^k\, \sL^n( Z(x_1,\theta_1,X_1^{\sqrt{n}}) \cap W^{\sqrt{n}} )^k.
\end{align*}
\end{lemma}

\begin{proof}
By Lemma \ref{lem:BoundDifferenceOperator} the expressions in the expectations can be bounded by products of the form
$$
\prod_{j=1}^{\ell}\bigg[3^n M(\varphi) \sum_{z\in\mathbb{Z}^n} \mathbf{1}\{(z+[0,1]^n)\cap W\cap Z_j \neq \varnothing\} 2^{\eta([z+[0,1]^n])}\bigg],
$$
where $\ell\in\NN$ and $Z_j$, $j\in\{1,\hdots,\ell\}$, are intersections of subsets of $Z(x_i,\theta_i,X_i)$, $i\in\{1,2,3\}$. For $A\subseteq \RR^n$ we have that
\begin{align}\label{eq:BoundQ(A)}
\sum_{z\in\mathbb{Z}^n} \mathbf{1}\{(z+[0,1]^n)\cap A\neq\varnothing\} \leq \sL^n(A^{\sqrt{n}})
\end{align}
by definition of $A^{\sqrt{n}}$. Hence, the $j$-th factor in the above product has {at most $\sL^n(W^{\sqrt{n}}\cap Z_j^{\sqrt{n}})$ summands}. Moreover, H\"older's inequality yields that, for $z_1,\hdots,z_\ell\in\mathbb{Z}^n$,
$$
\EE \prod_{j=1}^\ell 2^{\eta([z_j+[0,1]^n])} \leq \EE 2^{\ell \eta([0,1]^n)}.
$$
Combining these estimates proves the desired inequalities (i)--(iv).
\end{proof}

The following bound for the expectation of the $k$-th difference operator will be used in Section \ref{sec:VarianceAsymptotics}. 

\begin{lemma}\label{lem:BoundKernels}
Let $\varphi$ be a geometric functional, $W\in\cK(\RR^n)$ and $k\in\mathbb{N}$. Then
$$
\big| \EE[\sfD^k_{(x_1,\theta_1,X_1),\hdots,(x_k,\theta_k,X_k)}\varphi(Z\cap W)] \big|\leq 3^{n}\, c_{1,\gamma,\QQ} \, M(\varphi)\, \sL^n\big( \bigcap_{j=1}^k Z(x_j,\theta_j,X_j^{\sqrt{n}}) \cap W^{\sqrt{n}} \big)
$$
for $(x_1,\theta_1,X_1),\hdots,(x_k,\theta_k,X_k)\in \RR^{n-m}\times\mathbb{M}_{n,m}$.
\end{lemma}

\begin{proof}
We have that
$$
\big| \EE[\sfD^k_{(x_1,\theta_1,X_1),\hdots,(x_k,\theta_k,X_k)}\varphi(Z\cap W)] \big|\leq \EE[|\sfD^k_{(x_1,\theta_1,X_1),\hdots,(x_k,\theta_k,X_k)}\varphi(Z\cap W)|].
$$
Bounding the right-hand side as in the proof of Lemma \ref{lem:BoundsExpectationDiffOp} yields the desired inequality.
\end{proof}

{The next auxiliary result will be applied in Section \ref{sec:VarianceAsymptotics}. We denote by $A\triangle B$ the symmetric difference of two sets $A,B\subseteq \mathbb{R}^n$, i.e. $A\triangle B=(A\setminus B)\cup (B\setminus A)$.}

\begin{lemma}\label{lem:DiffVarphi}		
	Let $ \varphi  $ be a geometric functional and $ A, B \in \mathcal{K}(\RR^n) $. Then
	\begin{itemize}
		\item[(i)] $\displaystyle \EE[|\varphi(Z\cap A)-\varphi(Z\cap B)|]\leq 2 \cdot 3 ^n M(\varphi)\big(c_{1,\gamma,\QQ}-1\big) \sL^n((A\triangle B)^{\sqrt{n}})$,
		\item[(ii)] $\displaystyle\EE[ | \varphi( A)-\varphi( B)|]\leq 2 \cdot 3 ^n M(\varphi) \sL^n((A\triangle B)^{\sqrt{n}})${.}
	\end{itemize}
\end{lemma} 
\begin{proof}
	It follows from \eqref{eq:PhiZintesectA} that
	$$
	\varphi(Z\cap A) = \sum_{\mathcal{Q}\subseteq\mathcal{Q}(A)} (-1)^{|\mathcal{Q}|-1} \mathbf{1}\Big\{\bigcap_{Q\in\mathcal{Q}}Q\neq \varnothing\Big\} \varphi\Big(Z\cap A\cap \bigcap_{Q\in\mathcal{Q}} Q\Big)
	$$
	and 
	$$
	\varphi(Z\cap B) = \sum_{\mathcal{Q}\subseteq\mathcal{Q}(B)} (-1)^{|\mathcal{Q}|-1} \mathbf{1}\Big\{\bigcap_{Q\in\mathcal{Q}}Q\neq \varnothing\Big\} \varphi\Big(Z\cap B\cap \bigcap_{Q\in\mathcal{Q}} Q\Big).
	$$
These identities imply that
	\begin{align*}
	|\varphi(Z\cap A)-\varphi(Z\cap B)|& \leq \sum_{\mathcal{Q}\subseteq\mathcal{Q}(A)} \mathbf{1}\Big\{\bigcap_{Q\in\mathcal{Q}} Q \cap A \neq \bigcap_{Q\in\mathcal{Q}}Q \cap B\Big\} \Big|\varphi\Big(Z\cap A \cap \bigcap_{Q\in\mathcal{Q}} Q\Big)\Big| \\
	&\quad +\sum_{\mathcal{Q}\subseteq\mathcal{Q}(B)} \mathbf{1}\Big\{\bigcap_{Q\in\mathcal{Q}} Q\cap A \neq \bigcap_{Q\in\mathcal{Q}} Q\cap B\Big\} \Big|\varphi\Big(Z\cap B \cap \bigcap_{Q\in\mathcal{Q}} Q\Big)\Big|.
	\end{align*}
It follows from the equivalence of $\bigcap_{Q\in\mathcal{Q}} Q\cap A \neq \bigcap_{Q\in\mathcal{Q}} Q\cap B$ and $\bigcap_{Q\in\mathcal{Q}} Q \cap (A\triangle B)\neq \varnothing$ together with \eqref{eq:BoundVarphi} that
	\begin{align*}
& \mathbb{E}[|\varphi(Z\cap A)-\varphi(Z\cap B)|] \\
 & \leq  M(\varphi) (c_{1,\gamma,\QQ}-1) \bigg( \sum_{\mathcal{Q}\subseteq\mathcal{Q}(A)} \mathbf{1}\Big\{ \bigcap_{Q\in\mathcal{Q}} Q \cap (A\triangle B)\neq \varnothing \Big\} + \sum_{\mathcal{Q}\subseteq\mathcal{Q}(B)} \mathbf{1}\Big\{ \bigcap_{Q\in\mathcal{Q}} Q \cap (A\triangle B)\neq \varnothing \Big\} \bigg).
	\end{align*}
Combining the fact that each $Q$ belongs to at most $3^n$ intersections with
$$
 \sum_{Q\in\mathcal{Q}(M)} \mathbf{1}\Big\{Q\cap (A\triangle B) \neq \varnothing\Big\} \leq \sum_{z\in\mathbb{Z}^n} \mathbf{1}\Big\{(z+[0,1]^n)\cap (A\triangle B) \neq \varnothing\Big\} \leq \sL^{n}((A\triangle B)^{\sqrt{n}})
$$
for $M\in\{A,B\}$ completes the proof of (i). The proof of (ii) goes analogously.
\end{proof}

\subsection{Proof of Theorem~\ref{thm:AdditiveFunctionals}: Wasserstein bound {and qualitative result}}

Using Proposition~\ref{prop:CLTWasserstein} we can now prove a central limit theorem  for general additive functionals for the Wasserstein distance. In principle, the proof works in the same way as the proof of Theorem~\ref{thm:Volume}, but some careful analysis is needed to handle the general additive functionals. Analogously to Section~\ref{sec:WassersteinBoundVolume} we put
$$
F_{\varphi, r} {:=} \varphi(Z\cap W_r) \qquad\text{ and }\qquad G_{\varphi,r}:=\frac{F_{\varphi, r}-\EE\left[F_{\varphi, r}\right]}{\sqrt{\var(F_{\varphi, r})}}.
$$
Starting with the bound for $\alpha_{G_{\varphi, r},1}$, Fubini's theorem and Lemma~\ref{lem:BoundsExpectationDiffOp}  yield
\begin{equation}\label{eqn:Bound_alpha1_general}
\begin{split}
\alpha_{G_{\varphi, r},1}^2\leq &\frac{4\cdot 81^n\,\gamma^3\, c_{\gamma,\QQ} M(\varphi)^4}{\var(F_{\varphi, r})^2}\int_{\MM_{n,m}^3}\int_{(\RR^{n-m})^3} \sL^n( Z(x_1,\theta_1,X_1^{\sqrt{n}}) \cap W_r^{\sqrt{n}})\\
&\quad\times\sL^n( Z(x_2,\theta_2,X_2^{\sqrt{n}}) \cap W_r^{\sqrt{n}})\sL^n( Z(x_1,\theta_1,X_1^{\sqrt{n}})\cap Z(x_3,\theta_3,X_3^{\sqrt{n}}) \cap W_r^{\sqrt{n}})\\
&\quad\times \sL^n( Z(x_2,\theta_2,X_2^{\sqrt{n}})\cap Z(x_3,\theta_3,X_3^{\sqrt{n}}) \cap W_r^{\sqrt{n}})\\
&\quad\times (\sL^{n-m})^3(\dint(x_1,x_2,x_3)) \, {\QQ^3}(\dint((\theta_1,X_1),(\theta_2,X_2),(\theta_3,X_3))).
\end{split}
\end{equation}
{Using exactly the same arguments as in the proof of Theorem \ref{thm:Volume} for the Wasserstein distance, we derive that}
\begin{align*}
\alpha_{G_{\varphi, r},1}^2 \leq&\frac{4 \cdot 81^n\,\gamma^3\, c_{\gamma,\QQ} M(\varphi)^4\, {\diam(W_r^{\sqrt{n}})^{3m}}}{\var(F_{\varphi, r})^2} \sL^n(W_r^{\sqrt{n}})\,(\widetilde{m}_2)^3,
\end{align*}
where for {$a>0$} we define ${\widetilde{m}_a := \EE[\sL^{n-m}(\Xi^{\sqrt{n}})^a].}$
Note that we use the convention $W_r^{\sqrt{n}}=(W_r)^{\sqrt{n}}$. For $r\geq 1$ we have
$$
\sL^n(W_r^{\sqrt{n}}) = \sL^n(W_r+\mathbb{B}^n_{\sqrt{n}}) \leq \sL^n(W_r+ r \mathbb{B}^n_{\sqrt{n}}) \leq r^n \sL^n(W+ \mathbb{B}^n_{\sqrt{n}}) = r^n \sL^n(W^{\sqrt{n}})
$$
{and similarly $\diam(W_r^{\sqrt{n}}) \leq r \diam(W^{\sqrt{n}})$.} Thus, for such $r$ we obtain
\begin{align}
\alpha_{G_{\varphi, r},1}\leq 2 \cdot 9^n\,\gamma^\frac{3}{2} \, c_{\gamma,\QQ}^\frac{1}{2} \, M(\varphi)^2\, \diam(W^{\sqrt{n}})^{\frac{3}{2}m}\,\widetilde{m}_2^\frac{3}{2}\,{\sL^n(W^{\sqrt{n}})^{\frac{1}{2}}}\,\frac{r^{\frac{3}{2}m+\frac{1}{2}n}}{\var(F_{\varphi, r})}. \label{Bound:Alpha1}
\end{align}
Note that by Steiner's formula and the power mean inequality $(\sum_{j=1}^\ell a_i)^k\leq \ell^{k-1}\sum_{j=1}^\ell a_i^k$, which is valid for $k,\ell\geq 1$ and $a_1,\ldots,a_\ell>0$, we have that
\begin{align}\label{eq:BdTildeMByIntVol}
\widetilde{m}_k \leq (n-m+1)^{k-1} \sum_{j=0}^{n-m}n^{k(n-m-j)\over 2}\kappa_{n-m-j}^k\EE[V_j(\Xi)^k],
\end{align}
which implies that $\widetilde{m}_2<\infty$ by the assumptions on the moments of the intrinsic volumes of the typical cylinder base.

Because of Lemma \ref{lem:BoundsExpectationDiffOp}, the right-hand side of \eqref{eqn:Bound_alpha1_general} is also a bound for {$4\alpha_{G_{\varphi, r},2}^2$,} whence
 \begin{align}
 \alpha_{G_{\varphi, r},2}\leq 9^n\, \gamma^{\frac{3}{2}} \, c_{\gamma,\QQ}^{\frac{1}{2}} M(\varphi)^2\, \diam(W^{\sqrt{n}})^{\frac{3}{2}m}\,\widetilde{m}_2^{\frac{3}{2}}\,{\sL^n(W^{\sqrt{n}})^{\frac{1}{2}}}\,\frac{r^{\frac{3}{2}m+\frac{1}{2}n}}{\var(F_{\varphi, r})}.\label{Bound:Alpha2}
 \end{align}
{Combining Lemma~\ref{lem:BoundsExpectationDiffOp} (iv) with the same arguments that were used to show \eqref{eq:TpqVolume}, we obtain
\begin{equation}\label{eq:Tpq}
\begin{split}
T_{pq} := & \gamma\int_{\RR^{n-m}\times\MM_{n,m}}\EE[|\sfD_{(x,\theta,X)}G_{\varphi, r}|^p]^q \,(\sL^{n-m}\otimes \QQ)(\dint(x,\theta,X)) \\
\leq & \frac{3^{pqn} \gamma c_{p,\gamma,\QQ}^q M(\varphi)^{pq}}{ \var(F_{\varphi, r})^{{pq/2}}} \diam(W^{\sqrt{n}})^{(pq-1)m}\sL^n(W^{\sqrt{n}}) \widetilde{m}_{pq}\, r^{(pq-1)m+n}
\end{split}
\end{equation}
{for $p\in\mathbb{N}$ and $q>0$ with $pq\geq 1$.}
This implies that}
 \begin{align}
 \alpha_{G_{\varphi, r},3}\leq 27^n\, \gamma \, c_{3,\gamma,\QQ} M(\varphi)^3\, \diam(W^{\sqrt{n}})^{2m}\, \widetilde{m}_3 \,\sL^n(W^{\sqrt{n}})\, \frac{r^{2m+n}}{\var(F_{\varphi, r})^{\frac{3}{2}}}. \label{Bound:Alpha3}
 \end{align}
Again, \eqref{eq:BdTildeMByIntVol} shows that $\widetilde{m}_3<\infty$ by our assumptions on the moments of the intrinsic volumes of the typical cylinder base.

Now, we choose $r_0\in(0,\infty)$ such that $r^{-(n+m)}\var(F_{\varphi,r})\geq {1\over 2}v(\varphi,W)$ for all $r\geq r_0$. Then, plugging \eqref{Bound:Alpha1}, \eqref{Bound:Alpha2} and \eqref{Bound:Alpha3} into the bound in Proposition~\ref{prop:CLTWasserstein} and inserting the lower bound on $\var(F_{\varphi, r}) $ for $ r \geq r_0 $, we find that
\begin{align*}
d_W(G_{\varphi, r},N) 
&\leq \alpha_{G_{\varphi, r},1}+\alpha_{G_{\varphi, r},2}+\alpha_{G_{\varphi, r},3}\\ &\leq \Big(6v(\varphi,W)^{-1} 9^n\, \gamma^\frac{3}{2} \, c_{\gamma,\QQ}^\frac{1}{2} M(\varphi)^2\, \diam(W^{\sqrt{n}})^{\frac{3}{2}m}\,\widetilde{m}_2^\frac{3}{2}\,{\sL^n(W^{\sqrt{n}})^{\frac{1}{2}}}\\&\qquad+{2^\frac{3}{2}}v(\varphi,W)^{-\frac{3}{2}}27^n \,\gamma c_{3,\gamma,\QQ} M(\varphi)^3\, \diam(W^{\sqrt{n}})^{2m}\, \widetilde{m}_3 \,\sL^n(W^{\sqrt{n}})\Big)\,r^{-{n-m\over 2}}
\end{align*}
for all $r\geq\max\{1,r_0\}$. This completes the proof of Theorem \ref{thm:AdditiveFunctionals} {(ii).}

{Finally, we prove convergence in distribution under the second moment assumption. Since the previous upper bounds on $\alpha_{G_{\varphi,r},1}$ and $\alpha_{G_{\varphi,r},2}$ depend only on second moments, the assertion follows from the second bound in Proposition~\ref{prop:CLTWasserstein} if we can control the term $\alpha_{G_{\varphi,r},3}'$. For this we fix $u>0$. By the Cauchy-Schwarz inequality, we obtain
\begin{align*}
\alpha_{G_{\varphi,r},3}' & \leq \gamma\int_{\RR^{n-m}\times\MM_{n,m}} \mathbf{1}\{ \sL^{n-m}(X^{\sqrt{n}})\leq u \} \EE[|\sfD_{(x,\theta,X)}G_{\varphi, r}|^3] \,(\sL^{n-m}\otimes \QQ)(\dint(x,\theta,X)) \\
& \quad + {2}\gamma\int_{\RR^{n-m}\times\MM_{n,m}} \mathbf{1}\{ \sL^{n-m}(X^{\sqrt{n}})>u \} \EE[|\sfD_{(x,\theta,X)}G_{\varphi, r}|^3]^{2/3} \,(\sL^{n-m}\otimes \QQ)(\dint(x,\theta,X)) \\
& =: S_{1,u} + S_{2,u}.
\end{align*}
By bounding $S_{1,u}$ and $S_{2,u}$ as $T_{pq}$ above, we obtain that $S_{1,u}\to 0$, as $r\to\infty$, and that
$$
\limsup_{r\to\infty} S_{2,u} \leq \frac{{2}\gamma 3^{2n} c_{3,\gamma,\QQ}^{2/3} M(\varphi)^{2}}{ v(\varphi,W)} \diam(W^{\sqrt{n}})^{m}\sL^n(W^{\sqrt{n}}) \mathbb{E}[\mathbf{1}\{ \sL^{n-m}(\Xi^{\sqrt{n}})>u \} \sL^{n-m}(\Xi^{\sqrt{n}})^2 ].
$$
Letting now $u\to\infty$ completes the proof of Theorem \ref{thm:AdditiveFunctionals} (i) since the expectation on the right-hand side converges to zero by the second moment assumptions on the intrinsic volumes of $\Xi$.}
\qed

\subsection{Proof of Theorem \ref{thm:AdditiveFunctionals}: Kolmogorov bound}

In this section we carry {out} the proof of Theorem \ref{thm:AdditiveFunctionals} {(iii).} As for the volume, this requires to bound the three additional terms $\alpha_{G_{\varphi, r},4} $, $ \alpha_{G_{\varphi, r},5} $ and $ \alpha_{G_{\varphi, r},6} $ in Proposition~\ref{prop:CLTKolmogorov} to obtain a bound in the Kolmogorov distance.
In order to bound the fourth moment of $G_{\varphi, r}$ we use inequality \eqref{eq:4thMomentBound}, i.e.
\begin{align*}
\EE[G_{\varphi, r}^4] \leq \max\bigg\{&256\bigg(\gamma\int_{\RR^{n-m}\times\MM_{n,m}}(\EE[(\sfD_{(x,\theta,X)}G_{\varphi, r})^4])^{1/2}\,(\sL^{n-m}\otimes \QQ)(\dint(x,\theta,X))\bigg)^2,\\
&\qquad\qquad 4\gamma\int_{\RR^{n-m}\times\MM_{n,m}}\EE[(\sfD_{(x,\theta,X)}G_{\varphi, r})^4]\,(\sL^{n-m}\otimes \QQ)(\dint(x,\theta,X)+2\bigg\}.
\end{align*}
{It follows from \eqref{eq:Tpq} with $p=4$ and $q=1/2$ or $q=1$ that
\begin{align*}
\EE[G_{\varphi, r}^4] \leq  \max\bigg\{ & \frac{ 256 \cdot 81^{n} \gamma^2 c_{\gamma,\QQ} M(\varphi)^{4}}{ \var(F_{\varphi, r})^2} \diam(W^{\sqrt{n}})^{2m}\sL^n(W^{\sqrt{n}})^2 \widetilde{m}_{2}^2\, r^{2m+2n},\\
& \frac{ 4 \cdot 81^{n} \gamma c_{\gamma,\QQ} M(\varphi)^{4}}{ \var(F_{\varphi, r})^{{2}}} \diam(W^{\sqrt{n}})^{3m}\sL^n(W^{\sqrt{n}}) \widetilde{m}_{4}\, r^{3m+n} + 2 \bigg\}.
\end{align*}}
{The first expression in the maximum} is bounded because of our assumption that  $v(\varphi,W)>0$ and our assumptions on the expected intrinsic volume of the typical cylinder base, recall \eqref{eq:BdTildeMByIntVol}. 

{The first summand of the second expression} tends to $ 0 $, as  $r\rightarrow \infty$, since $m\in\{0,1,\ldots,n-1\}$ and $ \var(F_{\varphi, r})\geq \frac{v(\varphi,W)}{2}r^{n+m} $ for $ r\geq r_0 $, where $r_0$ is the same constant as in the proof of part (i). Hence, taking these two bounds together yields that, {for $r\geq r_0$, $ \EE\left[G_{\varphi, r}^4 \right]$ can be bounded by a constant, $\widetilde{c}\in(0,\infty)$ say, only depending on the parameters mentioned in Remark \ref{rm:ConstantsGeneral}.} {Together with \eqref{eq:Tpq} for $p=4$ and $q=3/4$ we find that, for $r\geq r_0$,}
\begin{align*}
\alpha_{G_{\varphi, r},4}\leq {\sqrt{2}} \gamma \, \tilde{c}^{\frac{1}{4}}\,27^n\, v(\varphi,W)^{-\frac{3}{2}}\,M(\varphi)^3 \, c_{\gamma,\QQ}^\frac{3}{4} \diam(W^{\sqrt{n}})^{2m}\,\widetilde{m}_3 \,\sL^n(W^{\sqrt{n}}) \ r^{-\frac{n-m}{2}}.
\end{align*}
{From \eqref{eq:Tpq} with $p=4$ and $q=1$ we deduce that}
\begin{align*}
\alpha_{G_{\varphi, r},5}\leq \gamma^{\frac{1}{2}} \,9^n\, 2v(\varphi,W)^{-1}\,M(\varphi)^2 \, c_{\gamma,\QQ}^\frac{1}{2} \diam(W^{\sqrt{n}})^{\frac{3}{2}m}\,\widetilde{m}_4^\frac{1}{2}\, \sL^n(W^{\sqrt{n}})^\frac{1}{2} \ r^{-\frac{n-m}{2}}
\end{align*}
for $r\geq r_0$. 
{From Lemma~\ref{lem:BoundsExpectationDiffOp} (iii) and (iv) we get}
\begin{align*}
{ \alpha_{G_{\varphi, r},6}^2 } &\leq \frac{ { 9 } \cdot 81^n\, M(\varphi)^4\, c_{\gamma,\QQ} \gamma^2}{\var(F_{\varphi, r})^2}\int_{ \MM_{n,m}^2}\int_{(\RR^{n-m})^2} \sL^n( Z(x_1,\theta_1, X_1^{\sqrt{n}}) \cap W_r^{\sqrt{n}})^2\\
&\hspace{5cm} \times\sL^n( Z(x_1,\theta_1, X_1^{\sqrt{n}})\cap Z(x_2,\theta_2, X_2^{\sqrt{n}}) \cap W_r^{\sqrt{n}}  )^2\\
&\hspace{5cm} \times(\sL^{n-m})^2(\dint(x_1,x_2))\QQ^2(\dint((\theta_1, X_1),(\theta_2, X_2)).
\end{align*}
{Bounding the right-hand side as in \eqref{eq:Alpha6Volume}, we obtain}
\begin{align*}
\alpha_{G_{\varphi, r},6}^2 \leq {9 \cdot} 81^n\, M(\varphi)^4\, c_{\gamma,\QQ}\, \gamma^2 \diam(W^{\sqrt{n}})^{3m}\,\sL(W^{\sqrt{n}})\,\widetilde{m}_2\, \widetilde{m}_3\,\frac{r^{3m+n}}{\var(F_{\varphi, r})^2}.
\end{align*}
Putting together all these bounds, using \eqref{eq:BdTildeMByIntVol} and the lower variance bound shows that
\begin{align*}
d_K(G_{\varphi, r},N) 
&\leq \alpha_{G_{\varphi, r},1}+\alpha_{G_{\varphi, r},2}+\alpha_{G_{\varphi, r},3}+\alpha_{G_{\varphi, r},4}+\alpha_{G_{\varphi, r},5}+ { \alpha_{G_{\varphi, r},6} } 
\leq C\,r^{-{n-m\over 2}}
\end{align*}
for $r\geq\max\{1,r_0\}$, where $C\in(0,\infty)$ is a constant only depending on the parameters mentioned in Remark \ref{rm:ConstantsGeneral}.\hfill$  \Box $

\subsection{{Proof of Theorem \ref{thm:Multivariate}}}\label{sec:MultiProof} 

Part (i) follows via the Cramer-Wold device immediately from Theorem \ref{thm:AdditiveFunctionals} (i).

According to Proposition \ref{prop:CLTmultiD3} for part (ii) we need to bound the three terms $\alpha_{\bF_r,1}$, $\alpha_{\bF_r,2}$ and $\alpha_{\bF_r,3}$ in order to establish a bound for the $d_3$-distance between $\bF_r$ and the Gaussian {random} vector $\bN_{\bF_r}$. Starting with $ \alpha_{\bF_r,1}^2$ as given in \eqref{Alpha1} we notice that by Lemma~\ref{lem:BoundsExpectationDiffOp} we have
\begin{align*}
&(\EE[(\sfD_{(x_1,\theta_1,X_1)}F_r^{(i)})^2(\sfD_{(x_2,\theta_2,X_2)}F_r^{(i)})^2])^{1/2}\\
&\hspace{4cm}\times(\EE[(\sfD_{(x_1,\theta_1,X_1),(x_3,\theta_3,X_3)}^2F_r^{(j)})^2(\sfD_{(x_2,\theta_2,X_2),(x_3,\theta_3,X_3)}^2F_r^{(j)})^2])^{1/2}\\
&\leq \frac{81^n\, M(\varphi_i)^2\,M(\varphi_j)^2 \, c_{\gamma,\QQ} }{r^{2(n+m)}}\sL^n(Z(x_1,\theta_1,X_1^{\sqrt{n}}) \cap W_r^{\sqrt{n}})\sL^n( Z(x_2,\theta_2,X_2^{\sqrt{n}}) \cap W_r^{\sqrt{n}}) \\
&\quad \times\sL^n( Z(x_1,\theta_1,X_1^{\sqrt{n}})\cap Z(x_3,\theta_3,X_3^{\sqrt{n}}) \cap W_r^{\sqrt{n}})\sL^n( Z(x_2,\theta_2,X_2^{\sqrt{n}})\cap Z(x_3,\theta_3,X_3^{\sqrt{n}}) \cap W_r^{\sqrt{n}}),
\end{align*}
so that we can proceed in exactly the same way as for $\alpha_{G_r,1}^2$ in the proof of Theorem \ref{thm:AdditiveFunctionals} to bound the threefold integral in \eqref{Alpha1}. We find that
\begin{align*}
 \alpha_{\bF_r,1}^2\leq &\gamma^3 \,81^n\, c_{\gamma,\QQ} \diam(W^{\sqrt{n}})^{3m}\,\sL(W^{\sqrt{n}})\,\widetilde{m}_2^3\, {\bigg( \sum_{i=1}^{m} M(\varphi_i)^2 \bigg)^2 } \ r^{-(n-m)}. 
\end{align*}
For the integrand in $\alpha_{\bF_r,2}$,  Lemma~\ref{lem:BoundsExpectationDiffOp} (ii) yields
\begin{align*}
&(\EE[(\sfD_{(x_1,\theta_1,X_1),(x_3,\theta_3,X_3)}^2F_r^{(i)})^2(\sfD_{(x_2,\theta_2,X_2),(x_3,\theta_3,X_3)}^2F_r^{(i)})^2])^{1/2}\\
& \hspace{4cm} \times(\EE[(\sfD_{(x_1,\theta_1,X_1),(x_3,\theta_3,X_3)}^2F_r^{(j)})^2(\sfD_{(x_2,\theta_2,X_2),(x_3,\theta_3,X_3)}^2F_r^{(j)})^2])^{1/2}\\
&\leq 81^n\,M(\varphi_i)^2\,M(\varphi_j)^2\, c_{\gamma,\QQ} \, \sL^n(W_r^{\sqrt{n}}\cap Z(x_1, \theta_1, X_1^{\sqrt{n}})\cap  Z(x_3, \theta_3, X_3^{\sqrt{n}}))^2 \\
&\hspace{4cm}\times\sL^n(W_r^{\sqrt{n}}\cap Z(x_2, \theta_2, X_2^{\sqrt{n}})\cap(x_3, \theta_3, X_3^{\sqrt{n}}))^2.
\end{align*}
With similar considerations as above we obtain
\begin{align*}
\alpha_{\bF_r,2}^2 \leq \gamma^3\,81^n\, c_{\gamma,\QQ} \diam(W^{\sqrt{n}})^{3m}\,\sL^n(W^{\sqrt{n}})\,\widetilde{m}_2^3\,\bigg(\sum_{i=1}^{m}M(\varphi_i)^2\bigg)^2\ r^{-(n-m)}.
\end{align*}
In the final term $\alpha_{\bF_r,3}$ each of the summands is of the same form as $\alpha_{G_r,3}$ in the proof of Theorem \ref{thm:AdditiveFunctionals}, so that
\begin{align*}
\alpha_{\bF_r,3} \leq \gamma \,27^n \, c_{3,\gamma,\QQ} \diam(W^{\sqrt{n}})^{2m}\, \widetilde{m}_3 \sL^n(W^{\sqrt{n}})  \bigg(\sum_{i=1}^{m} M(\varphi_i)^3\bigg)\ r^{-{n-m\over 2}}\,.
\end{align*}
Putting together the bounds for $\alpha_{\bF_r,1}$, $\alpha_{\bF_r,2}$ and $\alpha_{\bF_r,3}$ yields that
\begin{align*}
{d_3(\bF_r,\bN_{\bF_r})\leq d\,\alpha_{\bF_r,1}+{d\over 2}\,\alpha_{\bF_r,2}+{d^2\over 4}\,\alpha_{\bF_r,3}\leq C\, r^{-\frac{n-m}{2}}}
\end{align*}
for some constant $ C \in (0, \infty) $ only depending on the parameters mentioned in the statement of the theorem. This completes the proof. \qed

\section{Proofs III: Variance asymptotics for geometric functionals}\label{sec:VarianceAsymptotics}

\subsection{Proof of Theorem \ref{thm:asymptoticvariance}}
To prove Theorem \ref{thm:asymptoticvariance} we build on the general Fock space representation of Poisson functionals (see \cite[Theorem 18.6]{LP}), which applied to {$F_{\varphi,r}:=\varphi(Z\cap W_r)$} yields that
\begin{equation}\label{Fockspacerep.}
\begin{split}
\var(F_{\varphi,r}) &=\sum_{k=1}^{\infty} \frac{\gamma^k}{k!} \int_{(\RR^{n-m}\times \MM_{n,m})^k} \EE\left[\sfD^k_{(x_1,\theta_1, X_1), \ldots, (x_k,\theta_k, X_k)}F_{\varphi, r}\right]^2\\
&\qquad\qquad\times (\sL^{n-m}\otimes \QQ)^k(\dint ((x_1,\theta_1, X_1),\ldots, (x_k,\theta_k, X_k))) .
\end{split}
\end{equation}
In a first step we prove that for a direction non-atomic Poisson cylinder process, as a function of $r$, all terms with $k\geq 2$ are of order strictly less than $r^{n+m}$, and hence do not contribute to the asymptotic behaviour of $r^{-(n+m)}\var(F_{\varphi,r})$, as $r\to\infty$.
\begin{lemma}\label{lem:HigherOrderTerms} 
	For $k\geq 2$ define
	\begin{align*}
	I_k &:= \int_{(\RR^{n-m}\times \MM_{n,m})^k} \EE\left[\sfD^k_{(x_1,\theta_1, X_1), \ldots, (x_k,\theta_k, X_k)}F_{\varphi, r}\right]^2 \\
	&\qquad\qquad\times(\sL^{n-m}\otimes \QQ)^k(\dint ((x_1,\theta_1, X_1),\ldots, (x_k,\theta_k, X_k))).
	\end{align*}
For a direction non-atomic Poisson cylinder process with intensity $\gamma\in(0,\infty)$ and base-direction distribution $\QQ$ such that $\EE[V_j(\Xi)^2]<\infty$ for all $j\in\{1,\hdots,n-m\}$ we have
	$$ \lim_{r \rightarrow \infty} \frac{1}{r^{n+m}} \sum_{k=2}^{\infty}  \frac{{\gamma^k}}{k!} I_k =0.$$ 
\end{lemma}
\begin{proof}
	For every $ k \geq 2 $, Lemma~\ref{lem:BoundKernels} yields
\begin{align*}
&\gamma^k \int_{(\RR^{n-m}\times \MM_{n,m})^k} \EE\left[\sfD^k_{(x_1,\theta_1, X_1), \ldots, (x_k,\theta_k, X_k)}F_{\varphi, r}\right]^2 (\sL^{n-m}\otimes \QQ)^k(\dint((x_1,\theta_1, X_1),\ldots, (x_k,\theta_k, X_k)))\\
&\leq 9^n M(\varphi)^2 c_{1,\gamma,\QQ}^2 \gamma^k\int_{(\RR^{n-m}\times \MM_{n,m})^k} \sL^n\Big( \bigcap_{j=1}^k Z(x_j, \theta_j, X_j^{\sqrt{n}}) \cap W_r^{\sqrt{n}} \Big)^2 \\
&\qquad \qquad \qquad \qquad \qquad \qquad  \qquad \qquad\times(\sL^{n-m}\otimes \QQ)^k(\dint((x_1,\theta_1, X_1),\ldots, (x_k,\theta_k, X_k)))\\
&\leq 9^n M(\varphi)^2 c_{1,\gamma,\QQ}^2 \gamma^k\int_{(\RR^{n-m}\times \MM_{n,m})^k} \max_{y_1,y_2\in\mathbb{R}^{n-m}}\sL^n\Big(Z(y_1, \theta_1, X_1^{\sqrt{n}}) \cap Z(y_2, \theta_2, X_2^{\sqrt{n}}) \cap W_r^{\sqrt{n}}\Big) \\
&\qquad \qquad \qquad \qquad \qquad \qquad  \qquad \qquad\times \sL^n\Big(\bigcap_{j=1}^k Z(x_j, \theta_j, X_j^{\sqrt{n}}) \cap W_r^{\sqrt{n}} \Big) \\
&\qquad \qquad \qquad \qquad \qquad \qquad  \qquad \qquad\times(\sL^{n-m}\otimes \QQ)^k(\dint((x_1,\theta_1, X_1),\ldots, (x_k,\theta_k, X_k))).
\end{align*}
Using the translative integral formula from Proposition~\ref{prop:IntegralFormulaVolume} $ k$ times we obtain
\begin{align*}
{\gamma^k} I_k & \leq 9^n M(\varphi)^2 c_{1,\gamma,\QQ}^2 \gamma^k \EE[ \sL^{n-m}(\Xi^{\sqrt{n}})]^{k-2} \sL^n(W_r^{\sqrt{n}})\\
& \quad \times \mathbb{E}\bigg[  \sL^{n-m}(\Xi_1^{\sqrt{n}}) \sL^{n-m}(\Xi_2^{\sqrt{n}}) \max_{y_1,y_2\in\mathbb{R}^{n-m}}\sL^n\Big(Z(y_1, \Theta_1, \Xi_1^{\sqrt{n}}) \cap Z(y_2, \Theta_2, \Xi_2^{\sqrt{n}}) \cap W_r^{\sqrt{n}}\Big) \bigg]
\end{align*}
with the typical cylinder base $\Xi$ and independent $(\Xi_1,\Theta_1)$ and $(\Xi_2,\Theta_2)$ distributed according to $\QQ$. This implies that
\begin{equation}\label{eq:SumChaoses}
\begin{split}
\frac{1}{r^{n+m}}\sum_{k=2}^\infty \frac{\gamma^k}{k!} I_k & \leq  9^n M(\varphi)^2 c_{1,\gamma,\QQ}^2 \sum_{k=2}^\infty \frac{\gamma^k \EE[ \sL^{n-m}(\Xi^{\sqrt{n}})]^{k-2}}{k!} \sL^n(W^{\sqrt{n}})\\
& \quad \times \frac{1}{r^m}\mathbb{E}\bigg[  \sL^{n-m}(\Xi_1^{\sqrt{n}}) \sL^{n-m}(\Xi_2^{\sqrt{n}}) \\
& \quad \quad \quad \quad \quad \times \max_{y_1,y_2\in\mathbb{R}^{n-m}}\sL^n\Big( Z(y_1, \Theta_1, \Xi_1^{\sqrt{n}}) \cap Z(y_2, \Theta_2, \Xi_2^{\sqrt{n}}) \cap W_r^{\sqrt{n}} \Big) \bigg].
\end{split}
\end{equation}
Note that the series on the right-hand side is convergent. By \eqref{eq:CylinderVolBoundBase} we have $\mathbb{P}$-almost surely that
\begin{align*}
& \frac{1}{r^m} \sL^{n-m}(\Xi_1^{\sqrt{n}}) \sL^{n-m}(\Xi_2^{\sqrt{n}}) \max_{y_1,y_2\in\mathbb{R}^{n-m}}\sL^n\Big(Z(y_1, \Theta_1, \Xi_1^{\sqrt{n}}) \cap Z(y_2, \Theta_2, \Xi_2^{\sqrt{n}}) \cap W_r^{\sqrt{n}} \Big) \\
& \leq \frac{1}{r^m} \sL^{n-m}(\Xi_1^{\sqrt{n}})^2 \sL^{n-m}(\Xi_2^{\sqrt{n}}) \diam(W_r^{\sqrt{n}}))^{m}
\leq \sL^{n-m}(\Xi_1^{\sqrt{n}})^2 \sL^{n-m}(\Xi_2^{\sqrt{n}}) \diam(W^{\sqrt{n}})^{m}.
\end{align*}
Combining the assumption $\EE\left[ V_j(\Xi)^2 \right]<\infty$ for $j\in\{1,\hdots,n-m\}$ with Steiner's formula leads to $\mathbb{E}[\sL^{n-m}(\Xi^{\sqrt{n}})], \mathbb{E}[\sL^{n-m}(\Xi^{\sqrt{n}})^2]<\infty$ so that $\mathbb{E}[\sL^{n-m}(\Xi_1^{\sqrt{n}})^2 \sL^{n-m}(\Xi_2^{\sqrt{n}})]
<\infty$. Thus, we can compute the limit for $r\to\infty$ of the right-hand side of \eqref{eq:SumChaoses} with the dominated convergence theorem. Together with the observation that
\begin{equation}\label{eq:MaxPointwise}
\lim_{r \rightarrow \infty} \max_{z_1,z_2 \in \RR^{n-m}}\frac{\sL^n( Z(z_1, \Theta_1, \Xi_1)\cap  Z(z_2, \Theta_2, \Xi_2) \cap W_r^{\sqrt{n}} )}{r^m} =0 
\end{equation}
$\mathbb{P}$-almost surely, which is due to our assumption that the cylinder process is direction non-atomic and, thus, $\Theta_1\neq\Theta_2$ $\mathbb{P}$-almost surely, we deduce that the right hand side of \eqref{eq:SumChaoses} converges to zero as $r\to\infty$. This completes the proof.
\end{proof}

After we have seen that for $ k \geq 2 $ the terms in \eqref{Fockspacerep.} are negligible after multiplication with $r^{-(n+m)}$, we shall now prove that the limit of the first summand has the desired form. More precisely, we show the following result for which we recall that $ H(x, \theta)=\theta(\EE^m+x) $ stands for the $ m $-dimensional subspace of $ \RR^n $ that arises by first shifting $\EE^m={\rm span}\{e_{n-m+1}, \ldots, e_n\} $ by $ x \in \RR^{n-m}$ and then applying $ \theta \in \SO_{n,m} $. 

\begin{lemma}\label{lem:FirstSummand}
	In the setting of Theorem~\ref{thm:asymptoticvariance} the term
	\begin{align*}
	I_1 {:=} \int_{\RR^{n-m}\times\MM_{n,m}}\EE[\sfD_{(x,\theta, X)}F_{\varphi,r}]^2\,(\sL^{n-m}\otimes\QQ)(\dint(x,\theta,X))
	\end{align*}
	satisfies
	\begin{align*}
	\lim_{r \rightarrow \infty} \frac{I_1}{r^{n+m}} 
	& = \int_{\mathbb{M}_{n,m}} \bigg( \EE [\varphi( Z\cap \theta ([0,1)^m+X))] - \varphi(\theta ([0,1)^m+X)) \bigg)^2 {T(W,\theta)} \, \mathbb{Q}(\dint(\theta,X)).
	\end{align*}
\end{lemma}

\begin{proof}
{Throughout this proof we can assume without loss of generality that the origin in $\mathbb{R}^{n-m}$ is the centre of the circumball of $\Xi$.} It follows from Lemma \ref{lem:BoundKernels}, \eqref{eq:CylinderVolBoundBase} and Proposition \ref{prop:IntegralFormulaVolume} that, for $\mathbb{Q}$-almost all $(\theta,X)\in\mathbb{M}_{n,m}$,
\begin{align*}
& \int_{\mathbb{R}^{n-m}}\EE\left[\sfD_{(x,\theta, X)}F_{\varphi, r}\right]^2 \sL^{n-m}(\dint x)\\
& \leq 9^n c_{1,\gamma,\QQ}^2 M(\varphi)^2 \int_{\mathbb{R}^{n-m}} \sL^{n}( Z(x,\theta,X^{\sqrt{n}}) \cap W_r^{\sqrt{n}} )^2 \, \sL^{n-m}(\dint x) \\
& \leq 9^n c_{1,\gamma,\QQ}^2 M(\varphi)^2  \diam(W^{\sqrt{n}})^m  r^m \sL^{n-m}(X^{\sqrt{n}}) \int_{\mathbb{R}^{n-m}} \sL^{n}( Z(x,\theta,X^{\sqrt{n}}) \cap W_r^{\sqrt{n}} ) \, \sL^{n-m}(\dint x) \\
& \leq 9^n c_{1,\gamma,\QQ}^2 M(\varphi)^2  \diam(W^{\sqrt{n}})^m  r^m \sL^{n-m}(X^{\sqrt{n}})^2 \sL^{n}(W_r^{\sqrt{n}}) \\
& \leq 9^n c_{1,\gamma,\QQ}^2 M(\varphi)^2  \diam(W^{\sqrt{n}})^m \sL^{n}(W^{\sqrt{n}}) r^{n+m} \sL^{n-m}(X^{\sqrt{n}})^2.
\end{align*}
Combining the assumption $\EE\left[ V_j(\Xi)^2 \right]<\infty$ for $j\in\{1,\hdots,n-m\}$ with Steiner's formula yields that $\EE\left[ \sL^{n-m}(\Xi^{\sqrt{n}})^2 \right]<\infty$. This allows us to apply the dominated convergence theorem, whence it is sufficient to prove that
\begin{align*}
& \lim_{r\to\infty} \frac{1}{r^{n+m}} \int_{\mathbb{R}^{n-m}}\EE\left[\sfD_{(x,\theta, X)}F_{\varphi, r}\right]^2 \sL^{n-m}(\dint x) \\
& = \bigg( \EE [\varphi( Z\cap \theta ([0,1)^m+X))] - \varphi(\theta ([0,1)^m+X)) \bigg)^2 {T(W,\theta)}
\end{align*}
for $\mathbb{Q}$-almost all $(\theta,X)\in\mathbb{M}_{n,m}$. Using the substitution $x=ry$ {and the definition of $T(W,\theta)$ in \eqref{eq:DefTwtheta},} we see that this is equivalent to
\begin{equation}\label{eq:Equivalent2}
\begin{split}
& \lim_{r\to\infty} \frac{1}{r^{2m}} \int_{\mathbb{R}^{n-m}}\EE\left[\sfD_{(ry,\theta, X)}F_{\varphi, r}\right]^2 \sL^{n-m}(\dint y) \\
& = \bigg( \EE [\varphi( Z\cap \theta ([0,1)^m+X))] - \varphi(\theta ([0,1)^m+X)) \bigg)^2 
\int_{\mathbb{R}^{n-m}} \sL^{m}(H(y,\theta) \cap W)^2  \, \sL^{n-m}(\dint y) 
\end{split}
	\end{equation}
for $\mathbb{Q}$-almost all $(\theta,X)\in\mathbb{M}_{n,m}$. It follows from Lemma \ref{lem:BoundKernels}, \eqref{eq:CylinderVolBoundBase} and $W_r^{\sqrt{n}}\subseteq (W^{\sqrt{n}})_r$ for $r\geq 1$ that
\begin{align*}
& \frac{1}{r^{2m}} \EE\left[\sfD_{(ry,\theta, X)}F_{\varphi, r}\right]^2 \\
& \leq \frac{1}{r^{2m}} 9^n c_{1,\gamma,\QQ}^2 M(\varphi)^2 \sL^{n}( Z(ry,\theta,X^{\sqrt{n}}) \cap W_r^{\sqrt{n}} )^2  \mathbf{1}\{ Z(ry,\theta,X^{\sqrt{n}}) \cap W_r^{\sqrt{n}} \neq \varnothing\} \\
& \leq \frac{1}{r^{2m}} 9^n c_{1,\gamma,\QQ}^2 M(\varphi)^2 \diam(W_r^{\sqrt{n}})^{2m} \sL^{n-m}(X^{\sqrt{n}})^2 \mathbf{1}\{ Z(y,\theta,X^{\sqrt{n}}/r) \cap W^{\sqrt{n}} \neq \varnothing\}  \\
& \leq 9^n c_{1,\gamma,\QQ}^2 M(\varphi)^2 \diam(W^{\sqrt{n}})^{2m} \sL^{n-m}(X^{\sqrt{n}})^2 \mathbf{1}\{ Z(y,\theta,X^{\sqrt{n}}) \cap W^{\sqrt{n}} \neq \varnothing \}
\end{align*}
for $r\geq 1$. Note that, for $\mathbb{Q}$-almost all $(\theta,X)\in\mathbb{M}_{n,m}$,
$$
\int_{\mathbb{R}^{n-m}} \mathbf{1}\{W^{\sqrt{n}} \cap Z(y,\theta,X^{\sqrt{n}}) \neq \varnothing \} \, \sL^{n-m}(\dint y) <\infty
$$ 
so that we can apply the dominated convergence theorem in \eqref{eq:Equivalent2}. Thus, the assertion of the lemma is proven if we can show that
\begin{equation}\label{eq:PointwiseLimit}
\begin{split}
& \lim_{r\to\infty} \frac{1}{r^{m}} \EE\left[\sfD_{(ry,\theta, X)}F_{\varphi, r}\right] \\
& = \bigg( \varphi(\theta ([0,1)^m+X)) - \EE [ \varphi( Z\cap \theta ([0,1)^m+X))] \bigg)
\sL^{m}(H(y,\theta) \cap W)  
\end{split}
\end{equation}
for $\sL^{n-m}\otimes\mathbb{Q}$-almost all $(y,\theta,X)\in\mathbb{R}^{n-m}\times\mathbb{M}_{n,m}$. In the following, we establish this identity for fixed $(y,\theta,X)\in\mathbb{R}^{n-m}\times\mathbb{M}_{n,m}$. First assume that $H(y,\theta) \cap W=\varnothing$. This implies that $Z(ry,\theta,X)\cap W_r=\varnothing$ for $r$ sufficiently large. Since, by Lemma \ref{lem:DiffOperatorsVolume} and Remark \ref{rem:DiffOperatorAdditiveFct},
\begin{equation}\label{eq:Edry}
\EE\left[\sfD_{(ry,\theta, X)}F_{\varphi, r} \right] = \EE\left[\varphi(Z(ry,\theta,X)\cap W_r) - \varphi(Z\cap Z(ry,\theta,X)\cap W_r)\right]{,}
\end{equation}
this implies $\EE\left[\sfD_{(ry,\theta, X)}F_{\varphi, r}\right]=0$ for $r$ sufficiently large and, thus, proves \eqref{eq:PointwiseLimit} for $H(y,\theta) \cap W=\varnothing$.

Next we assume that $H(y,\theta) \cap W \neq \varnothing$. It follows from \eqref{eq:Edry} and Lemma \ref{lem:DiffVarphi} that
\begin{align*}
& \big| \EE\left[\sfD_{(ry,\theta, X)}F_{\varphi, r} \right] - \big(\varphi((H(ry,\theta) \cap W_r)+\theta X) - \EE\left[\varphi(Z\cap ((H(ry,\theta)\cap W_r)+\theta X))\right] \big) \big| \\
& \leq 2 \cdot 3 ^n c_{1,\gamma,\QQ} M(\varphi) \sL^n((( (H(ry,\theta)\cap W_r)+\theta X) \triangle (Z(ry,\theta,X)\cap W_r))^{\sqrt{n}}).
\end{align*}
Recalling that $ R(X )$ denotes the circumradius of $X$ and introducing the inner $s$-parallel set $A^{-s}:=\{x\in A:d(x,\partial A)\geq s\}$ of a set $A\subset\RR^n$ for $s\geq 0$, we have
\begin{align*}
((H(ry,\theta) \cap W_r)+\theta X)\triangle (Z(ry, \theta, X)\cap W_r)\subseteq \left(H(ry, \theta)\cap \big(W_r^{R(X)}\backslash W_r^{-R(X)}\big)\right)+ \theta X.
\end{align*}
This implies that
\begin{align*}
&\sL^n\Big(\big(((H(ry,\theta) \cap W_r)+\theta X)\triangle (Z(ry, \theta, X)\cap W_r)\big)^{\sqrt{n}}\Big)\\
&\qquad\qquad\leq \sL^n\Big(\big((H(ry, \theta)\cap (W_r^{R(X)}\backslash W_r^{-R(X)}))+ \theta X\big)^{\sqrt{n}}\Big) \\
&\qquad\qquad\leq \sL^n \big( (H(ry, \theta)\cap \big(W_r^{R(X)}\backslash W_r^{-R(X)}\big)^{\sqrt{n}})+ \theta X^{\sqrt{n}}\big),
\end{align*}
where $X^{\sqrt{n}}$ is the $\sqrt{n}$-parallel set of $X$ in $\mathbb{R}^{n-m}$. Moreover, we have, for $\sL^{n-m}$-almost all $y\in\mathbb{R}^{n-m}$,
\begin{align*}
&	\lim_{r\to\infty} \frac{1}{r^{m}} \sL^n \big( (H(ry, \theta)\cap (W_r^{R(X)}\backslash W_r^{-R(X)})^{\sqrt{n}}) + \theta X^{\sqrt{n}}\big) \\ 
& = \lim_{r\to\infty} \frac{1}{r^{m}} \sL^{m} \big(H(ry, \theta)\cap \big(W_r^{R(X)}\backslash W_r^{-R(X)}\big)^{\sqrt{n}}\big)  \sL^{n-m}(X^{\sqrt{n}})\\
& = \lim_{r\to\infty} \sL^{m} \big(H(y, \theta)\cap \big(W^{R(X)/r}\backslash W^{-R(X)/r}\big)^{\sqrt{n}/r}\big)  \sL^{n-m}(X^{\sqrt{n}}) = 0.
\end{align*}
This yields
\begin{equation}\label{eq:LimitIntegrand1}
\begin{split}
&\lim_{r\to\infty}	\frac{1}{r^m} \big| \EE\left[\sfD_{(ry,\theta, X)}F_{\varphi, r} \right] \\
&\qquad\qquad- \big(\varphi((H(ry,\theta)\cap W_r)+\theta X) - \EE\left[\varphi(Z\cap ((H(ry,\theta)\cap W_r)+\theta X))\right]\big)\big|=0
\end{split}
\end{equation}
for $\sL^{n-m}$-almost all $y\in \mathbb{R}^{n-m}$.
We define the function $\psi_{\theta,X}: \mathcal{R}^m\to \mathbb{R}$  by
$$
\psi_{\theta,X}(A) := \varphi(\theta (A+X)) - \EE[ \varphi( Z\cap \theta (A+X))].
$$
Since $Z$ is stationary and $\varphi$ is a geometric functional, $\psi_{\theta,X}$ is translation invariant, additive and conditionally bounded {(see \eqref{eq:BoundVarphiWithoutZ}).} It thus follows from \cite[Lemma 9.2.2]{SW} that
$$
\lim_{s\to\infty} \frac{\psi_{\theta,X}(sU)}{s^m \sL^m(U)}=\psi_{\theta,X}([0,1)^m)
$$
for all $U\in\mathcal{K}^m$ with $\sL^{m}(U)>0$. This implies that
\begin{equation}\label{eq:LimitIntegrand2}
\begin{split}
& \lim_{r\to\infty} \frac{1}{r^{m}} \big( \varphi( ( H(ry,\theta) \cap W_r)+\theta X) - \EE \varphi( Z \cap( (H(ry,\theta) \cap W_r)+\theta X) )\big) \\
& = \lim_{r\to\infty} \frac{1}{r^{m}} \psi_{\theta,X}(\theta^T(H(ry,\theta) \cap W_r)) \\
&= \lim_{r\to\infty} \frac{1}{r^{m}} \psi_{\theta,X}(r\theta^T(H(y,\theta) \cap W))\\
& = {\sL}^{m}(H(y,\theta)\cap W) \,\psi_{\theta,X}([0,1)^m) \\
& = \,{\sL}^{m}(H(y,\theta)\cap W) \big( \varphi(\theta ([0,1)^m+X)) - \EE [\varphi( Z\cap \theta ([0,1)^m+X)) ] \big).
\end{split}
\end{equation}
Combining \eqref{eq:LimitIntegrand1} and \eqref{eq:LimitIntegrand2} proves \eqref{eq:PointwiseLimit} and, thus, completes the proof.
\end{proof}

\begin{proof}[Proof of Theorem \ref{thm:asymptoticvariance}]
In the direction non-atomic case the claim follows directly by combining \eqref{Fockspacerep.} with Lemma~\ref{lem:HigherOrderTerms} and Lemma~\ref{lem:FirstSummand}. To prove the formula in the general case we start by noting that from \eqref{Fockspacerep.} it follows that
\begin{align*}
\var(F_{\varphi,r}) & = \gamma \int_{\RR^{n-m}\times \MM_{n,m}} \EE\left[\sfD_{(x_1,\theta_1, X_1)}F_{\varphi, r}\right]^2 \, (\sL^{n-m}\otimes \QQ)(\dint (x_1,\theta_1, X_1)) \\
& \quad + \sum_{k=2}^{\infty} \frac{\gamma^k}{k!} \int_{(\RR^{n-m}\times \MM_{n,m})^k} \mathbf{1}\{\theta_1=\hdots=\theta_k\} \EE\left[\sfD^k_{(x_1,\theta_1, X_1), \ldots, (x_k,\theta_k, X_k)}F_{\varphi, r}\right]^2\\
& \quad \qquad\qquad\times (\sL^{n-m}\otimes \QQ)^k(\dint ((x_1,\theta_1, X_1),\ldots, (x_k,\theta_k, X_k))) \\
& \quad +\sum_{k=2}^{\infty} \frac{\gamma^k}{k!} \int_{(\RR^{n-m}\times \MM_{n,m})^k} (1-\mathbf{1}\{\theta_1=\hdots=\theta_k\}) \EE\left[\sfD^k_{(x_1,\theta_1, X_1), \ldots, (x_k,\theta_k, X_k)}F_{\varphi, r}\right]^2\\
& \quad \qquad\qquad\times (\sL^{n-m}\otimes \QQ)^k(\dint ((x_1,\theta_1, X_1),\ldots, (x_k,\theta_k, X_k))) \\
& =: S_1+S_2+S_3.
\end{align*}
Following the proof of Lemma \ref{lem:HigherOrderTerms}, we can show that
\begin{equation}\label{eq:LimitS3}
\lim_{r\to\infty} \frac{S_3}{r^{n+m}}=0.
\end{equation}
Indeed, in the proof of Lemma \ref{lem:HigherOrderTerms} the fact that the direction distribution has no atoms is only used in \eqref{eq:MaxPointwise}. This identity still holds in case of atoms if we have $\Theta_1\neq \Theta_2$. Thus, it is sufficient to rename the integration variables in $S_3$ such that $\theta_1\neq\theta_2$. This leads to an additional factor $k$ that is absorbed in the convergent series in \eqref{eq:SumChaoses}.

In order to deal with $S_2$ we introduce the measure
\begin{align*}
\widetilde{\QQ}(\,\cdot\,) & =\sum_{k=2}^\infty \frac{\gamma^k}{k!} \sum_{i\in I} \int_{\mathbb{M}_{n,m}}\int_{(\mathbb{R}^{n-m}\times \mathbb{M}_{n,m})^{k-1}} \mathbf{1}\{\theta_1=\hdots=\theta_k=\varrho_i\} \\
&\qquad\times \mathbf{1}\Big\{ \Big(\varrho_i, X_1\cap\bigcap_{j=2}^k (x_i+X_i)\Big) \in \,\cdot\,\Big\} \,(\sL^{n-m}\otimes \QQ)^{k-1}(\dint ((x_2,\theta_2, X_2),\ldots\\
&\hspace{5cm}\ldots
, (x_k,\theta_k, X_k))) \, \QQ(\dint(\theta_1,X_1))
\end{align*}
on $\mathbb{M}_{n,m}$. Note that $\widetilde{\QQ}$ is a finite measure since by \eqref{eq:BoundQ(A)} and a $(k-1)$-fold application of the translative integral formula \eqref{eq:TranslativeClassical} it holds that
\begin{align*}
&\int_{(\mathbb{R}^{n-m})^{k-1}} \mathbf{1}\Big\{ X_1\cap\bigcap_{j=2}^k (x_i+X_i) \neq \varnothing \Big\} \,(\sL^{n-m})^{k-1}(\dint(x_2,\ldots,x_k))\\
&\leq \int_{\RR^{n-m}}\sum_{z\in\ZZ^{n-m}}{\bf 1}\Big\{(z+[0,1]^{n-m})\cap X_1\cap\bigcap_{j=2}^k(x_j+X_j) \neq \varnothing\Big\}\, {(\sL^{n-m})^{k-1}}(\dint(x_2,\ldots,x_k))\\
&\leq \int_{\RR^{n-m}}\sL^{n-m}\Big(\Big(X_1\cap\bigcap_{j=2}^k(x_j+X_j)\Big)^{\sqrt{n-m}}\Big)\,{(\sL^{n-m})^{k-1}}(\dint(x_2,\ldots,x_k))\\
&\leq \int_{\RR^{n-m}}\sL^{n-m}\Big(X_1^{\sqrt{n-m}}\cap\bigcap_{j=2}^k \big(x_j+X_j^{\sqrt{n-m}}\big)\Big)\,{(\sL^{n-m})^{k-1}}(\dint(x_2,\ldots,x_k))\\
&=\prod_{j=1}^k\sL^{n-m}(X_j^{\sqrt{n-m}}),
\end{align*}
and hence 
\begin{align*}
\widetilde{\QQ}(\mathbb{M}_{n,m}) & \leq \sum_{k=2}^\infty \frac{\gamma^k}{k!} \int_{\MM_{n,m}^{k}}\int_{(\mathbb{R}^{n-m})^{k-1}} \mathbf{1}\Big\{ X_1\cap\bigcap_{j=2}^k (x_i+X_i) \neq \varnothing \Big\} \\
& \qquad \qquad \qquad \times (\sL^{n-m})^{k-1}(\dint(x_2,\ldots,x_k))\QQ^k(\dint((\theta_1,X_1),\ldots,(\theta_k,X_k))) \\
& \leq \sum_{k=2}^\infty \frac{\gamma^k}{k!} \EE[\sL^{n-m}(\Xi^{\sqrt{n-m}})]^k.
\end{align*}
This series is finite since by Steiner's formula $\EE[\sL^{n-m}(\Xi^{\sqrt{n-m}})]$ can be expressed as a linear combination of the expected intrinsic volumes $\EE[V_j(\Xi)]$, $j\in\{1,\ldots,n-m\}$, of the typical cylinder base, which in turn are finite by assumption. Moreover, we have that 
\begin{align}
\nonumber & \int_{\mathbb{M}_{n,m}} \sL^{n-m}(X^{\sqrt{n}})^2 \, \widetilde{\QQ}(\dint(\theta,X)) \\
\nonumber & =\sum_{k=2}^\infty \frac{\gamma^k}{k!} \sum_{i\in I} \int_{\mathbb{M}_{n,m}}\int_{(\mathbb{R}^{n-m}\times \mathbb{M}_{n,m})^{k-1}} \mathbf{1}\{\theta_1=\hdots=\theta_k=\varrho_i\}  \sL^{n-m}\Big( \Big(X_1\cap\bigcap_{j=2}^k (x_i+X_i)\Big)^{\sqrt{n}}\Big)^2 \\
\nonumber & \qquad \qquad \qquad \times (\sL^{n-m}\otimes \QQ)^{k-1}(\dint ((x_2,\theta_2, X_2),\ldots, (x_k,\theta_k, X_k))) \, \QQ(\dint(\theta_1,X_1)) \allowdisplaybreaks\\
\nonumber & \leq \sum_{k=2}^\infty \frac{\gamma^k}{k!} \int_{\mathbb{M}_{n,m}}\int_{(\mathbb{R}^{n-m}\times \mathbb{M}_{n,m})^{k-1}} \sL^{n-m}( X_1^{\sqrt{n}}) \sL^{n-m}\Big( X_1^{\sqrt{n}}\cap\bigcap_{j=2}^k (x_i+X_i^{\sqrt{n}})\Big) \\
\nonumber & \qquad \qquad \qquad \times (\sL^{n-m}\otimes \QQ)^{k-1}(\dint ((x_2,\theta_2, X_2),\ldots, (x_k,\theta_k, X_k))) \, \QQ(\dint(\theta_1,X_1)) \allowdisplaybreaks\\
\nonumber & = \sum_{k=2}^\infty \frac{\gamma^k}{k!} \int_{\mathbb{M}_{n,m}^k} \sL^{n-m}( X_1^{\sqrt{n}})^2 \prod_{j=2}^k \sL^{n-m}( X_j^{\sqrt{n}}) \, \QQ^k(\dint((\theta_1,X_1),\hdots,(\theta_k,X_k)))\\
& = \sum_{k=2}^\infty \frac{\gamma^k}{k!} \EE[\sL^{n-m}(\Xi)^2] \EE[\sL^{n-m}(\Xi)]^{k-1}< \infty,\label{eq:SecondMomentTildeQ}
\end{align}
where the second equality is a consequence of the translative integral formula \eqref{eq:TranslativeClassical} and since $\EE[\sL^{n-m}(\Xi)^2]$ and $ \EE[\sL^{n-m}(\Xi)]$ are finite by assumption. Together with {Lemma \ref{lem:DiffOperatorsVolume} and} Remark \ref{rem:DiffOperatorAdditiveFct} we obtain
\begin{align*}
S_2 & = \sum_{k=2}^{\infty} \frac{\gamma^k}{k!} \int_{(\RR^{n-m}\times \MM_{n,m})^k} \mathbf{1}\{\theta_1=\hdots=\theta_k\} \\
& \qquad  \times \bigg(\EE\Big[\varphi\Big(Z\cap \bigcap_{j=1}^k Z(x_j,\theta_j,X_j) \cap W_r\Big)\Big] - \varphi\Big(\bigcap_{j=1}^k Z(x_j,\theta_j,X_j) \cap W_r\Big) \bigg)^2\\
& \qquad \times (\sL^{n-m}\otimes \QQ)^k(\dint ((x_1,\theta_1, X_1),\ldots, (x_k,\theta_k, X_k)))\allowdisplaybreaks \\
& = \sum_{k=2}^{\infty} \frac{\gamma^k}{k!} \int_{(\RR^{n-m}\times \MM_{n,m})^k} \sum_{i\in I} \mathbf{1}\{\theta_1=\hdots=\theta_k=\varrho_i\} \\
&   \qquad \times \bigg(\EE\Big[\varphi\Big(Z\cap \varrho_i\Big(\bigcap_{j=1}^k (x_j+X_j)\times\mathbb{E}^m\Big) \cap W_r\Big)\Big] - \varphi\Big(\varrho_i\Big(\bigcap_{j=1}^k (x_j+X_j)\times\mathbb{E}^m\Big) \cap W_r\Big) \bigg)^2\\
&  \qquad\times (\sL^{n-m}\otimes \QQ)^k(\dint ((x_1,\theta_1, X_1),\ldots, (x_k,\theta_k, X_k))) \allowdisplaybreaks\\
& = \sum_{k=2}^{\infty} \frac{\gamma^k}{k!} \int_{\mathbb{R}^{n-m}} \int_{\mathbb{M}_{n,m}} \int_{(\RR^{n-m}\times \MM_{n,m})^{k-1}} \sum_{i\in I} \mathbf{1}\{\theta_1=\hdots=\theta_k=\varrho_i\} \\
&   \qquad \times \bigg(\EE\Big[\varphi\Big(Z\cap \Big(\varrho_i x_1 + \varrho_i\Big(\Big(X_1\cap\bigcap_{j=1}^k (x_j+X_j)\Big)\times\mathbb{E}^m\Big)\Big) \cap W_r\Big)\Big] \\
& \qquad \qquad \qquad \qquad - \varphi\Big(\Big(\varrho_i x_1 + \varrho_i\Big(\Big(X_1\cap \bigcap_{j=2}^k (x_j+X_j)\Big)\times\mathbb{E}^m\Big)\Big) \cap W_r\Big) \bigg)^2\\
& \quad \qquad\qquad\times (\sL^{n-m}\otimes \QQ)^{k-1}(\dint ((x_2,\theta_2, X_2),\ldots, (x_k,\theta_k, X_k))) \, \QQ(\dint(\theta_1,X_1)) \, \sL^{n-m}(\dint x_1) \allowdisplaybreaks\\
& = \int_{\mathbb{R}^{n-m}} \int_{\mathbb{M}_{n,m}}
\bigg(\EE\Big[\varphi\Big(Z\cap \theta ((x_1 + X)\times \mathbb{E}^m) \cap W_r\Big)\Big]  - \varphi\Big(\theta ((x_1 + X) \times \mathbb{E}^m ) \cap W_r\Big) \bigg)^2
\\
&\hspace{6cm}\times\widetilde{\QQ}(\dint(\theta,X)) \, \sL^{n-m}(\dint x_1) \allowdisplaybreaks\\
& = \int_{\mathbb{R}^{n-m}} \int_{\mathbb{M}_{n,m}}
\EE\left[\sfD_{(x_1,\theta, X)}F_{\varphi, r}\right]^2
\, \widetilde{\QQ}(\dint(\theta,X)) \, \sL^{n-m}(\dint x_1). 
\end{align*}
Since $\widetilde{\QQ}$ is a finite measure and \eqref{eq:SecondMomentTildeQ} holds, we can treat the integral on the right-hand side exactly as $I_1$ in the proof of Lemma \ref{lem:FirstSummand}. This leads to
\begin{equation}\label{eq:LimitS2}
\begin{split}
\lim_{r\to\infty} \frac{S_2}{r^{n+m}} & = \int_{\mathbb{M}_{n,m}} \big( \EE [\varphi( Z\cap \theta ([0,1)^m+X))] - \varphi(\theta ([0,1)^m+X)) \big)^2 \\
&\qquad\qquad\qquad\qquad\times\int_{\mathbb{R}^{n-m}} \sL^{m}(H(y,\theta) \cap W)^2  \, \sL^{n-m}(\dint y) \,  \widetilde{\QQ}(\dint(\theta,X)).
\end{split}
\end{equation}
Finally, from Lemma \ref{lem:FirstSummand} we get 
\begin{equation}\label{eq:LimitS1}
\begin{split}
\lim_{r\to\infty} \frac{S_1}{r^{n+m}} & = \int_{\mathbb{M}_{n,m}} \big( \EE [\varphi( Z\cap \theta ([0,1)^m+X))] - \varphi(\theta ([0,1)^m+X)) \big)^2 \\
&\qquad\qquad\qquad\qquad\times\int_{\mathbb{R}^{n-m}} \sL^{m}(H(y,\theta) \cap W)^2  \, \sL^{n-m}(\dint y) \, \QQ(\dint(\theta,X)).
\end{split}
\end{equation}
Combining \eqref{eq:LimitS3}, \eqref{eq:LimitS2} and \eqref{eq:LimitS1} with the definition of $\widetilde{\QQ}$ proves Theorem \ref{thm:asymptoticvariance}.
\end{proof}

\subsection{Some integral formulas for $V_n$ and $V_{n-1}$}

The following lemma, whose proof relies on applying results for Boolean models from \cite{HugLastSchulteBM}, is a key ingredient for the proofs of Corollary \ref{cor:VarVolumeConstant} and Corollary \ref{cor:VarSurface}.

\begin{lemma}\label{lem:HLS}
For $\theta\in\SO_{n,m}$ and $i,j\in\{n-m-1,n-m\}$ define
\begin{align*}
\tau_{i,j}(\theta) := \sum_{k=1}^\infty \frac{\gamma^k}{k!} \int_{\MM_{n,m}^{k}} \int_{(\RR^{n-m})^{k-1}} & \mathbf{1}\{ \theta_1=\hdots=\theta_k= \theta \} V_i(X_1\cap\bigcap_{\ell=2}^k (x_\ell+X_\ell) ) \\
& \times V_j(X_1\cap \bigcap_{\ell=2}^k (x_\ell+X_\ell)) \\
& \times (\sL^{n-m})^{k-1}(\dint(x_2,\hdots,x_k)) \, \QQ_{n,m}^{k}(\dint ((\theta_1,X_1)\hdots,(\theta_k,X_k)))
\end{align*}
and let $f$ and $g$ be as in \eqref{eq:defFunctionF} and \eqref{eqn:Definition_g}. Then,
\begin{equation}\label{eqn:Rho_d}
\tau_{n-m,n-m}(\theta) =  \int_{\RR^{n-m}} \big( e^{\gamma f(x,\theta) } - 1 \big) \, \sL^{n-m}(\dint x),
\end{equation}
\begin{equation}\label{eqn:Rho_dd-1}
\tau_{n-m-1,n-m}(\theta) = \frac{\gamma}{2} \int_{\RR^{n-m}} e^{\gamma f(x,\theta)} g(x,\theta) \, {\sL^{n-m}(\dint x),}
\end{equation}
and
\begin{equation}\label{eqn:Rho_d-1d-1}
\begin{split}
& \tau_{n-m-1,n-m-1}(\theta) \\
& = \frac{\gamma^2}{4} \int_{\RR^{n-m}} e^{\gamma f(x,\theta)} g(x,\theta) g(-x,\theta) \, \sL^{n-m}(\dint x) \\
& \quad + \frac{\gamma}{4} \EE \bigg[ \mathbf{1}\{ \Theta = \theta \} \int_{\partial \Xi} \int_{\partial \Xi} e^{\gamma f(y-z,\theta)} \, \sH^{n-m-1}(\dint y) \, \sH^{n-m-1}(\dint z) \bigg] \\
& \quad + \frac{3\gamma}{4} \EE \bigg[ \mathbf{1}\{ V_{n-m}(\Xi)=0, \Theta = \theta \} \int_{\partial \Xi} \int_{\partial \Xi} e^{\gamma f(y-z,\theta)} \, \sH^{n-m-1}(\dint y) \, \sH^{n-m-1}(\dint z) \bigg].
\end{split}
\end{equation}
\end{lemma}

\begin{proof}
For $\theta\in\SO_{n,m}$ we define the measures $\mathbb{Q}_{\theta}(\,\cdot\,):=\mathbb{Q}(\{\theta\}\times\,\cdot\,)$ and
$$
\Lambda_{\theta}(\,\cdot\,):= \gamma \int_{\RR^{n-m}} \int_{\mathcal{K}(\RR^{n-m})} \mathbf{1}\{x+K\in \,\cdot\, \}  \, \mathbb{Q}_\theta(\dint K) \, \sL^{n-m}(\dint x)
$$
on $\mathcal{K}(\RR^{n-m})$. This allows us to rewrite $\tau_{i,j}(\theta)$ as
\begin{align*}
\tau_{i,j}(\theta) = \sum_{k=1}^\infty \frac{\gamma}{k!} \int_{\mathcal{K}(\RR^{n-m})} \int_{\mathcal{K}(\RR^{n-m})^{k-1}} & V_i(K_1\cap \hdots \cap K_k) \\
& \times V_j(K_1\cap \hdots \cap K_k) \, \Lambda_\theta^{k-1}(\dint (K_2,\hdots,K_k)) \, \QQ_\theta(\dint K_1).
\end{align*}
Throughout the proof we can assume that $\mathbb{Q}_\theta(\mathcal{K}(\RR^{n-m}))>0$ since the statements are obviously true otherwise. We define $\overline{\QQ}_\theta(\,\cdot\,):=\QQ_\theta(\,\cdot\,)/\QQ_\theta(\mathcal{K}(\RR^{n-m}))$ which turns ${\QQ}_\theta$ into a probability measure. Moreover, let $\gamma_\theta := \gamma \mathbb{Q}_\theta(\mathcal{K}(\RR^{n-m}))$ and note that
$$
\Lambda_{\theta}(\,\cdot\,)= \gamma_\theta \int_{\RR^{n-m}} \int_{\mathcal{K}(\RR^{n-m})} \mathbf{1}\{x+K\in \,\cdot\, \}  \, \overline{\mathbb{Q}}_\theta(\dint K) \, \sL^{n-m}(\dint x).
$$
For $i,j\in\{n-m-1,n-m\}$ this yields
\begin{align*}
\tau_{i,j}(\theta)= \sum_{k=1}^\infty \frac{\gamma_\theta}{k!} \int_{\mathcal{K}(\RR^{n-m})} \int_{\mathcal{K}(\RR^{n-m})^{k-1}} & V_{i}(K_1\cap \hdots \cap K_k) \\
& \times V_{j}(K_1\cap \hdots \cap K_k) \, \Lambda_\theta^{k-1}(\dint (K_2,\hdots,K_k)) \, \overline{\QQ}_\theta(\dint K_1).
\end{align*}
The idea of this proof is to apply \cite[Theorem 5.2]{HugLastSchulteBM} to the Boolean model in $\RR^{n-m}$ with intensity $\gamma_\theta$ and the distribution $\overline{\QQ}_\theta$ for the typical grain. In this case the function $C_{n-m}:\RR^{n-m}\to\RR$  (in \cite{HugLastSchulteBM} denoted by $C_d$ as $d$ is the dimension) is given by
$$
C_{n-m}(x) = \int_{\mathcal{K}(\RR^{n-m})} V_{n-m}(K \cap (K+x)) \, \overline{\QQ}_\theta(\dint K) ,\qquad x\in\RR^{n-m}
$$
{and} satisfies
\begin{equation}\label{eqn:Cd}
\gamma_\theta C_{n-m}(x) 
= \gamma \EE[ V_{n-m}(\Xi \cap (\Xi+x)) \mathbf{1}\{ \Theta= \theta\} ] = \gamma f(x,\theta).
\end{equation}

From \cite[Theorem 5.2 and (5.10)]{HugLastSchulteBM} and \eqref{eqn:Cd} it follows that
\begin{align*}
\tau_{n-m,n-m}(\theta) 
& = \int_{\RR^{n-m}} \big( e^{\gamma_\theta C_{n-m}(x) } - 1 \big) \, \sL^{n-m}(\dint x) = \int_{\RR^{n-m}} \big( e^{\gamma f(x,\theta) } - 1 \big) \, \sL^{n-m}(\dint x),
\end{align*}
which proves \eqref{eqn:Rho_d}.

Using again \cite[Theorem 5.2]{HugLastSchulteBM} and \eqref{eqn:Cd}, we obtain
$$
\tau_{n-m-1,n-m}(\theta) = \frac{\gamma_\theta}{2} \int_{\mathcal{K}(\RR^{n-m})} \int_{\RR^{n-m}} \int_{\partial K} e^{\gamma f(y-z,\theta)} \mathbf{1}\{z\in K\} \, \sH^{n-m-1}(\dint y) \, {\sL^{n-m}(\dint z)} \, \overline{\QQ}_\theta(\dint K).
$$
The substitution $x=z-y$, the symmetry of $f(\,\cdot\,,\theta)$ and the definition of $g$ lead to
\begin{align*}
\tau_{n-m-1,n-m}(\theta) & = \frac{\gamma_\theta}{2} \int_{\mathcal{K}(\RR^{n-m})} \int_{\RR^{n-m}} \int_{\partial K} e^{\gamma f(-x,\theta)} \mathbf{1}\{x\in K-y\} \, \sH^{n-m-1}(\dint y) \, {\sL^{n-m}(\dint x)} \, \overline{\QQ}_\theta(\dint K) \\
& = \frac{\gamma}{2} \int_{\RR^{n-m}} e^{\gamma f(x,\theta)} \int_{\mathcal{K}(\RR^{n-m})} \int_{\partial K} \mathbf{1}\{y\in K-x\} \, \sH^{n-m-1}(\dint y) \, \QQ_\theta(\dint K) \, {\sL^{n-m}(\dint x)} \\
& = \frac{\gamma}{2} \int_{\RR^{n-m}} e^{\gamma f(x,\theta)} g(x,\theta) \, {\sL^{n-m}(\dint x),}
\end{align*}
which is \eqref{eqn:Rho_dd-1}.

For our computation of $\tau_{n-m-1,n-m-1}(\theta)$ we first assume that
\begin{equation}\label{eqn:FullDimensional}
\overline{\QQ}_\theta(\{K\in \mathcal{K}(\RR^{n-m}): V_{n-m}(K)>0 \})=1.
\end{equation}
Under this assumption, it follows from \cite[Theorem 5.2]{HugLastSchulteBM} that
\begin{align*}
\tau_{n-m-1,n-m-1}(\theta) & = \frac{\gamma_\theta}{4} \int_{\mathcal{K}(\RR^{n-m})} \int_{\mathcal{K}(\RR^{n-m})} \int_{\partial K_2} \int_{\partial K_1} e^{\gamma f(y-z,\theta)} \mathbf{1}\{y\in K_2^\circ, z\in K_1^\circ\} \\
& \hspace{3.5cm} \times \sH^{n-m-1}(\dint y) \, \sH^{n-m-1}(\dint z) \, \Lambda_\theta(\dint K_1) \, \overline{\QQ}_\theta(\dint K_2) \\
& \quad + \frac{\gamma_\theta}{4} \int_{\mathcal{K}(\RR^{n-m})} \int_{\partial K} \int_{\partial K} e^{\gamma f(y-z,\theta)} \, \sH^{n-m-1}(\dint y) \, \sH^{n-m-1}(\dint z) \, \overline{\mathbb{Q}}_\theta(\dint K) \\
& =:T_1 + T_2.
\end{align*}
We can rewrite $T_1$ as
\begin{align*}
T_1 & = \frac{\gamma_\theta^2}{4} \int_{\mathcal{K}(\RR^{n-m})^2} \int_{\RR^{n-m}} \int_{\partial K_2} \int_{x+\partial K_1} e^{\gamma f(y-z,\theta)} \mathbf{1}\{y\in K_2^\circ, z-x\in K_1^\circ\} \\
& \hspace{5.5cm} \times \sH^{n-m-1}(\dint y) \, \sH^{n-m-1}(\dint z) \, \sL^{n-m}(\dint x) \, \overline{\QQ}_\theta^2(\dint (K_1, K_2)) \\
 & = \frac{\gamma^2}{4} \int_{\mathcal{K}(\RR^{n-m})^2} \int_{\RR^{n-m}} \int_{\partial K_2} \int_{\partial K_1} e^{\gamma f(x+y-z,\theta)} \mathbf{1}\{x+y\in K_2^\circ, z-x\in K_1^\circ\} \\
& \hspace{5.5cm} \times \sH^{n-m-1}(\dint y) \, \sH^{n-m-1}(\dint z) \, \sL^{n-m}(\dint x) \, \QQ_\theta^2(\dint (K_1, K_2)) \\
 & = \frac{\gamma^2}{4} \int_{\mathcal{K}(\RR^{n-m})^2} \int_{\partial K_2} \int_{\partial K_1} \int_{\RR^{n-m}} e^{\gamma f(x,\theta)} \mathbf{1}\{x+z\in K_2^\circ, y-x\in K_1^\circ\} \\
& \hspace{5.5cm} \times \sL^{n-m}(\dint x) \, \sH^{n-m-1}(\dint y) \, \sH^{n-m-1}(\dint z) \, \QQ_\theta^2(\dint (K_1, K_2)) \\
 & = \frac{\gamma^2}{4} \int_{\RR^{n-m}} e^{\gamma f(x,\theta)} g(x,\theta) g(-x,\theta) \, \sL^{n-m}(\dint x),
\end{align*}
and for $T_2$ we get
$$
T_2 = \frac{\gamma}{4} \EE \bigg[ \mathbf{1}\{ \Theta = \theta \} \int_{\partial \Xi} \int_{\partial \Xi} e^{\gamma f(y-z,\theta)} \, \sH^{n-m-1}(\dint y) \, \sH^{n-m-1}(\dint z) \bigg].
$$
This completes the proof of \eqref{eqn:Rho_d-1d-1} under assumption \eqref{eqn:FullDimensional}. Next, we show how to remove \eqref{eqn:FullDimensional}. For this, we write
\begin{align*}
& \tau_{n-m-1,n-m-1}(\theta) \\
& = \sum_{k=1}^\infty \frac{\gamma_\theta}{k!} \int_{\mathcal{K}(\RR^{n-m})} \int_{\mathcal{K}(\RR^{n-m})^{k-1}}  \mathbf{1}\{ V_{n-m}(K_\ell)>0, \ell\in\{1,\hdots,k\} \} V_{n-m-1}(K_1\cap \hdots \cap K_k)^2 \\
& \qquad \qquad \qquad \qquad \qquad \qquad \qquad \times \Lambda_\theta^{k-1}(\dint (K_2,\hdots,K_k)) \, \overline{\QQ}_\theta(\dint K_1) \\
& \quad + \sum_{k=1}^\infty \frac{\gamma_\theta}{k!} \int_{\mathcal{K}(\RR^{n-m})} \int_{\mathcal{K}(\RR^{n-m})^{k-1}}  \mathbf{1}\{ \exists \ell\in \{1,\hdots,k\} : V_{n-m}(K_\ell)=0 \} V_{n-m-1}(K_1\cap \hdots \cap K_k)^2 \\
& \quad \qquad \qquad \qquad \qquad \qquad \qquad \qquad \times \Lambda_\theta^{k-1}(\dint (K_2,\hdots,K_k)) \, \overline{\QQ}_\theta(\dint K_1) \\
& =: S_1 + S_2.
\end{align*}
By the same arguments as used to compute $\tau_{n-m-1,n-m-1}(\theta)$ under the assumption \eqref{eqn:FullDimensional}, we obtain
\begin{align*}
S_1 & = \frac{\gamma^2}{4} \int_{\RR^{n-m}} e^{\gamma f(x,\theta)} g(x,\theta) g(-x,\theta) \, \sL^{n-m}(\dint x) \\
& \quad + \frac{\gamma}{4} \EE \bigg[ \mathbf{1}\{ V_{n-m}(\Xi)>0, \Theta = \theta \} \int_{\partial \Xi} \int_{\partial \Xi} e^{\gamma f(y-z,\theta)} \, \sH^{n-m-1}(\dint y) \, \sH^{n-m-1}(\dint z) \bigg],
\end{align*}
where we used that
$$
\EE[ V_{n-m}(\Xi\cap (\Xi+x)) \mathbf{1}\{V_{n-m}(\Xi)>0, \Theta=\theta\} ] = \EE[ V_{n-m}(\Xi\cap (\Xi+x)) \mathbf{1}\{\Theta=\theta\} ]
$$
and
$$
\EE[ \sH^{n-m-1}(\partial \Xi\cap (\Xi-x)) \mathbf{1}\{V_{n-m}(\Xi)>0, \Theta=\theta\} ] = \EE[ \sH^{n-m-1}(\partial \Xi\cap (\Xi-x)) \mathbf{1}\{\Theta=\theta\} ] = g(x,\theta)
$$
for $\sL^{n-m}$-almost all $x\in\RR^{n-m}$. {For} $\Lambda_\theta^k$-almost all $K_1,\hdots,K_k\in\mathcal{K}(\RR^{n-m})$, $V_{n-m-1}(K_1\cap \hdots \cap K_k)=0$ whenever more than one of the sets $K_1,\hdots,K_k$ has volume zero. Indeed, this is a direct consequence of the translative integral formula \eqref{eq:TranslativeClassical}. By symmetry it follows that
\begin{align*}
S_2 & = \sum_{k=1}^\infty \frac{\gamma}{(k-1)!} \int_{\mathcal{K}(\RR^{n-m})} \mathbf{1}\{ V_{n-m}(K_1)=0 \} \int_{\mathcal{K}(\RR^{n-m})^{k-1}}   V_{n-m-1}(K_1\cap \hdots \cap K_k)^2 \\
& \qquad \qquad \qquad \qquad \qquad \qquad \qquad \qquad \qquad\qquad  \times \, \Lambda_\theta^{k-1}(\dint (K_2,\hdots,K_k)) \, \QQ_\theta(\dint K_1) \\
& = \gamma \int_{\mathcal{K}(\RR^{n-m})} \mathbf{1}\{ V_{n-m}(K)=0 \} \int_{K^2} \sum_{k=0}^\infty \frac{1}{k!} \int_{\mathcal{K}(\RR^{n-m})^{k}} \mathbf{1}\{ y,z\in K_1 \cap \hdots \cap K_k \} \\
& \hspace{6cm} \times \Lambda_\theta^{k}(\dint (K_1,\hdots,K_k)) \, (\sH^{n-m-1})^2(\dint(y,z))   \, \QQ_\theta(\dint K) \\
& = \gamma \int_{\mathcal{K}(\RR^{n-m})} \mathbf{1}\{ V_{n-m}(K)=0 \}  \int_{K^2} \sum_{k=0}^\infty \frac{1}{k!} \bigg( \int_{\mathcal{K}(\RR^{n-m})} \mathbf{1}\{y,z\in \widetilde{K}\} \, \Lambda_\theta(\dint \widetilde{K}) \bigg)^k \\
& \hspace{6cm} \times (\sH^{n-m-1})^2(\dint(y,z))   \, \QQ_\theta(\dint K).
\end{align*}
The observation that
\begin{align*}
\int_{\mathcal{K}(\RR^{n-m})} \mathbf{1}\{y_1,y_2\in\widetilde{K}\} \, \Lambda_\theta(\dint\widetilde{K}) & = \gamma \EE\Big[\mathbf{1}\{\Theta = \theta\} \int_{\RR^{n-m}} \mathbf{1}\{y \in \Xi+x\} \mathbf{1}\{z \in \Xi+x\} \, \sL^{n-m}(\dint x)\Big] \\
& = \gamma \EE[ \mathbf{1}\{\Theta = \theta\} V_{n-m}(\Xi\cap (\Xi+y-z)) ] = \gamma f(y-z,\theta)
\end{align*}
leads to
\begin{align*}
S_2 & = \gamma\int_{\mathcal{K}(\RR^{n-m})} \mathbf{1}\{ V_{n-m}(K)=0 \} \int_{K^2} e^{\gamma f(y-z,\theta)} \, (\sH^{n-m-1})^2(\dint(y_1,y_2))   \, \QQ_\theta(\dint K) \\
& = \gamma \EE \Big[\mathbf{1}\{ V_{n-m}(\Xi)=0, \Theta=\theta \} \int_{\Xi} \int_{\Xi} e^{\gamma f(y-z,\theta)} \, \sH^{n-m-1}(\dint y) \, \sH^{n-m-1}(\dint z)\Big],
\end{align*}
which completes the proof.
\end{proof}

\subsection{Proof of Corollary \ref{cor:VarVolumeConstant}}

We start by evaluating $\EE[ V_n( Z\cap \theta ([0,1)^m+K))]$ for $\theta\in\SO_{n,m}$ and $K\in\cK(\RR^{n-m})$ by means of Proposition \ref{prop:ExpectationVarianceVolume} (i). This yields that
$$
\EE[ V_n( Z\cap \theta ([0,1)^m+K))] = (1-e^{-\gamma m_1}) V_n(\theta ([0,1)^m+K))
$$
so that
\begin{equation}\label{eq:19-07-a}
\begin{split}
\EE[ V_n( Z\cap \theta ([0,1)^m+K))]-V_n( \theta ([0,1)^m+K)) &= {-}e^{-\gamma m_1}V_n( \theta ([0,1)^m+K)) \\
&
={-}e^{-\gamma m_1} V_{n-m}(K).
\end{split}
\end{equation}
Thus, it follows from Theorem \ref{thm:asymptoticvariance} that
\begin{align*}
v(V_n, W) &=\gamma\int_{\mathbb{M}_{n,m}} e^{-2\gamma m_1} V_{n-m}(X)^2  T(W,\theta) \mathbf{1}\{\theta\notin\{\varrho_i: i\in I\}\} \,  \mathbb{Q}(\dint(\theta,X)) \allowdisplaybreaks \\
& \quad + \sum_{i\in I} \sum_{k=1}^\infty \frac{\gamma^k}{k!} \int_{\mathbb{M}_{n,m}^k}  \mathbf{1}\{\theta_1=\hdots=\theta_k=\varrho_i\} T(W,\varrho_i) e^{-2\gamma m_1} \\
& \qquad \qquad \qquad \qquad \times \int_{(\mathbb{R}^{n-m})^{k-1}}  V_{n-m}\Big( X_1\cap\bigcap_{j=2}^k (x_j+X_j)\Big)^2 \, (\sL^{n-m})^{k-1}(\dint(x_2,\hdots,x_k)) \\
& \qquad \qquad \qquad \qquad \times \mathbb{Q}^k(\dint((\theta_1,X_1),\hdots, (\theta_k,X_k) )) \\
&=: V_1 + V_2.
\end{align*}
Here, we moved the case $\theta\in\{\varrho_i:i\in I\}$ from the first term to the series for $k=1$. Clearly, we have
$$
V_1 = \gamma e^{-2\gamma m_1} \EE\Big[ V_{n-m}(\Xi)^2 T(W,\Theta) \mathbf{1}\{\Theta\notin\{\varrho_i: i\in I\}\} \Big].
$$
From Lemma \ref{lem:HLS}, we get
$$
V_2 = e^{-2\gamma m_1} \sum_{i\in I} \tau_{n-m,n-m}(\varrho_i) T(W,\varrho_i) = e^{-2\gamma m_1} \sum_{i\in I} \int_{\RR^{n-m}} \big( e^{\gamma f(x,\varrho_i)} -1 \big) \, \sL^{n-m}(\dint x) T(W,\varrho_i),
$$
{which proves} the formula for $v(V_n,W)$.\qed

\subsection{Proof of Corollary \ref{cor:VarSurface}}

For the proof of Corollary \ref{cor:VarSurface} we need the following lemma, which partially {generalises} \cite[Lemma 5.1]{ConcentrationIneq} from the isotropic to the anisotropic case. It can be regarded as the analogue for Poisson cylinder processes of the corresponding mean value formula from \cite[p.\ 386]{SW} for the classical Boolean model.

\begin{lemma}\label{lem:MeanValue}
Consider the union set $Z$ of a stationary Poisson cylinder process with intensity $\gamma\in(0,\infty)$ and {$\PP$-almost} surely convex typical cylinder base $\Xi$. Let $W\in\cK(\RR^n)$, and put $m_1:=\EE[V_{n-m}(\Xi)]$ and $s_1:=\EE[V_{n-m-1}(\Xi)]$. Then,
$$
\EE[V_{n-1}(Z\cap W)] = (1-e^{-\gamma m_1}) V_{n-1}(W) + \gamma s_1e^{-\gamma m_1} V_n(W).
$$
\end{lemma}
\begin{proof}
As in the proof of \cite[Lemma 5.1]{ConcentrationIneq} the inclusion-exclusion formula in combination with the multivariate Mecke identity for {Poisson processes} yields that
\begin{align*}
\EE[V_{n-1}(Z\cap W)] &= \sum_{\ell=1}^\infty{(-1)^{\ell-1}\over\ell!}\gamma^\ell\int_{\MM_{n,m}^\ell}\int_{(\RR^{n-m})^\ell}V_{n-1}\big(Z(x_1,\theta_1,X_1)\cap\ldots\cap Z(x_\ell,\theta_\ell,X_\ell)\cap W\big)\\
&\qquad\qquad\times(\sL^{n-m})^\ell(\dint(x_1,\ldots,x_\ell))\QQ^\ell(\dint((\theta_1,X_1),\ldots,(\theta_\ell,X_\ell))).
\end{align*}
To evaluate the inner multiple translative integral we use \cite[Theorem 2]{SchneiderWeil86} (with $d=n$, $q=m$, $j=n-1$ and $\beta=\beta'=\RR^n$ there), which says that, for given $X\in\cK(\RR^{n-m})$, $Y\in\cK(\RR^n)$ and $\theta\in\SO_{n,m}$ the identity
$$
\int_{\RR^{n-m}}V_{n-1}(Z(x,\theta,X)\cap Y)\,\sL^{n-m}(\dint x) = V_{n-1}(Y)V_{n-m}(X)+V_n(Y)V_{n-m-1}(X)
$$
holds independently of $\theta$. Applying this formula $\ell$ times we conclude that, for $\QQ^\ell$-almost all $(\theta_1,X_1),\ldots,(\theta_\ell,X_\ell)\in\MM_{n,m}$,
\begin{align*}
&\int_{(\RR^{n-m})^\ell}V_{n-1}\big(Z(x_1,\theta_1,X_1)\cap\ldots\cap Z(x_\ell,\theta_\ell,X_\ell)\cap W\big)\,{(\sL^{n-m})^\ell}(\dint(x_1,\ldots,x_\ell))\\
&=V_{n-1}(W)V_{n-m}(X_1)\cdots V_{n-m}(X_\ell)\\
&\qquad+\sum_{i=1}^\ell V_n(W)V_{n-m}(X_1)\cdots V_{n-m}(X_{i-1})V_{n-m-1}(X_i)V_{n-m}(X_{i+1})\cdots V_{n-m}(X_\ell),
\end{align*}
again independently of $\theta_1,\ldots,\theta_\ell$. Plugging this into the above representation for $\EE[V_{n-1}(Z\cap W)]$ and using the definition of $m_1$ and $s_1$ we obtain
\begin{align*}
\EE[V_{n-1}(Z\cap W)] &= V_{n-1}(W)\sum_{\ell=1}^\infty {(-1)^{\ell-1}\over\ell!}\gamma^\ell m_1^\ell + V_n(W)s_1\sum_{\ell=1}^\infty {(-1)^{\ell-1}\over\ell!}\ell\gamma^\ell m_1^{\ell-1}\\
&=V_{n-1}(W)(1-e^{-\gamma m_1}) + \gamma V_n(W) s_1e^{-\gamma m_1}.
\end{align*}
This completes the argument.
\end{proof}

Moreover, we will make use of the following geometric lemma.

\begin{lemma}\label{lem:Vn-1Sum}
For $K\in\mathcal{K}(\mathbb{R}^{n-m})$,
$
V_{n-1}([0,1)^m+K) = V_{n-m-1}(K).
$
\end{lemma}

\begin{proof}
By \cite[Lemma 14.2.1]{SW}, we have
\begin{align*}
V_{n-1}([0,1]^m+K) = \sum_{k=0}^{n-1}V_k([0,1]^m)V_{n-k-1}(K).
\end{align*}
Since $V_k([0,1]^m)={m\choose k}$ for $k\in\{0,1,\ldots,m\}$ and $V_k([0,1]^m)=0$ otherwise, and since $V_{n-k-1}(K)=0$ for $k<m-1$ we conclude that
\begin{align*}
V_{n-1}([0,1]^m+K) = \sum_{k=m-1}^{m}{m\choose k}V_{n-k-1}(K) = mV_{n-m}(K) + V_{n-m-1}(K).
\end{align*}
Next, we notice that $\partial^+[0,1)^m$ is the union of precisely $m$ of the $2m$ facets of $[0,1]^m$, $F_1,\ldots,F_m$ say, which in turn are $(m-1)$-dimensional closed unit cubes. Thus, using the additivity of $V_{n-1}$ and the fact that $V_{n-1}$ vanishes for sets of dimension $n-2$ or less, we can argue as above to conclude that
\begin{align*}
V_{n-1}(\partial^+[0,1)^m+K) &= \sum_{i=1}^mV_{n-1}(F_i+K) = \sum_{i=1}^m\sum_{k=0}^{n-1}V_k(F_i)V_{n-k-1}(K) \\ 
&= \sum_{i=1}^m V_{m-1}(F_i)V_{n-m}(K) = mV_{n-m}(K).
\end{align*}
Hence, we obtain
\begin{align*}
V_{n-1}([0,1)^m+K) &= V_{n-1}([0,1]^m+K) - V_{n-1}(\partial^+[0,1)^m+K)  \notag \\
&= mV_{n-m}(K) + V_{n-m-1}(K) - mV_{n-m}(K)\notag \\
& = V_{n-m-1}(K),
\end{align*}
which completes the proof.
\end{proof}

With these preparations we can prove Corollary \ref{cor:VarSurface}. 

\begin{proof}[Proof of Corollary \ref{cor:VarSurface}]
For $K\in\cK(\RR^{n-m})$ and $\theta\in\SO_{n,m}$ it follows from Lemma \ref{lem:MeanValue} and Lemma \ref{lem:Vn-1Sum} that
\begin{align}\label{eqn:ExpectationSurface}
\notag&\EE[ V_{n-1}( Z\cap \theta ([0,1)^m+K))] - V_{n-1}(\theta ([0,1)^m+K)) \\
\notag& = \gamma s_1e^{-\gamma m_1} V_{n}(\theta([0,1)^m+K))+(1-e^{-\gamma m_1})V_{n-1}(\theta([0,1)^m+K))  - V_{n-1}(\theta ([0,1)^m+K)) \\
\notag& = \gamma s_1e^{-\gamma m_1} V_{n}([0,1)^m+K) -e^{-\gamma m_1} V_{n-1}([0,1)^m+K) \\
& = \gamma s_1e^{-\gamma m_1} V_{n-m}(K) -e^{-\gamma m_1} V_{n-m-1}(K).
\end{align}
Next we apply Theorem \ref{thm:asymptoticvariance}, where we can include again the situation $\theta\in\{\varrho_i:i\in I\}$ in the first term as $k=1$ in the second term. Together with \eqref{eqn:ExpectationSurface} we obtain that
\begin{align*}
v(V_{n-1},W) = & \gamma \int_{\MM_{n,m}} \bigg( \gamma s_1e^{-\gamma m_1} V_{n-m}(X) -e^{-\gamma m_1} V_{n-m-1}(X) \bigg)^2 T(W,\theta) \\
& \qquad \qquad \times \mathbf{1}\{\theta\notin \{\varrho_i: i\in I\}\} \, \QQ(\dint(\theta,X)) \\
& +\sum_{i\in I} \sum_{k=1}^\infty \frac{\gamma^k}{k!} \int_{\MM_{n,m}^k} \mathbf{1}\{\theta_1=\hdots=\theta_k=\varrho_i\} T(W,\varrho_i) e^{-2\gamma m_1}  \\
& \qquad \quad \times \int_{(\mathbb{R}^{n-m})^{k-1}} \bigg( \gamma s_1 V_{n-m}(X_1\cap\bigcap_{j=2}^k (x_j+X_j)) \\
& \qquad \qquad \qquad \qquad - V_{n-m-1}(X_1\cap\bigcap_{j=2}^k (x_j+X_j)) \bigg)^2 \, (\sL^{n-m})^{k-1}(\dint(x_2,\hdots,x_k)) \\
& \qquad \qquad \qquad \times \mathbb{Q}^k(\dint((\theta_1,X_1),\hdots, (\theta_k,X_k) )) \\
=:& V_1+V_2.
\end{align*}
Here, $V_1$ can be rewritten as
$$
V_1 = \gamma e^{-2\gamma m_1} \EE[ (\gamma s_1 V_{n-m}(\Xi) - V_{n-m-1}(\Xi))^2 T(W,\Theta) \mathbf{1}\{\Theta\notin\{\varrho_i:i\in I\}\} ],
$$
while multiplying out the quadratic term in $V_2$ leads to
\begin{align*}
V_2 &= e^{-2\gamma m_1} \sum_{i\in I}  T(W,\varrho_i) 
\bigg( \gamma^2 s_1^2 \tau_{n-m,n-m}(\varrho_i)  -2 \gamma s_1 \tau_{n-m,n-m-1}(\varrho_i) + \tau_{n-m-1,n-m-1}(\varrho_i) \bigg)
\end{align*}
with $\tau_{i,j}(\,\cdot\,)$, $i,j\in\{n-m-1,n-m\}$, as defined in Lemma \ref{lem:HLS}. This result also yields that
\begin{align*}
V_2 & = e^{-2\gamma m_1} \sum_{i\in I} T(W,\varrho_i)\\
& \quad \times \bigg( \gamma^2 s_1^2 \int_{\RR^{n-m}} \big( e^{\gamma f(x,\varrho_i) } - 1 \big) \, \sL^{n-m}(\dint x) - \gamma^2 s_1 \int_{\RR^{n-m}} e^{\gamma f(x,\varrho_i)} g(x,\varrho_i) \, {\sL^{n-m}(\dint x)}\\
& \qquad \qquad + \frac{\gamma^2}{4} \int_{\RR^{n-m}} e^{\gamma f(x,\varrho_i)} g(x,\varrho_i) g(-x,\varrho_i) \, \sL^{n-m}(\dint x) \\
& \qquad \qquad + \frac{\gamma}{4} \EE \bigg[ \mathbf{1}\{ \Theta = \varrho_i \} \int_{\partial \Xi} \int_{\partial \Xi} e^{\gamma f(y-z,\varrho_i)} \, \sH^{n-m-1}(\dint y) \, \sH^{n-m-1}(\dint z) \bigg] \\
& \qquad \qquad + \frac{3\gamma}{4} \EE \bigg[ \mathbf{1}\{ V_{n-m}(\Xi)=0, \Theta = \varrho_i \} \int_{\partial \Xi} \int_{\partial \Xi} e^{\gamma f(y-z,\varrho_i)} \, \sH^{n-m-1}(\dint y) \, \sH^{n-m-1}(\dint z) \bigg] \bigg) \allowdisplaybreaks\\
& =  \gamma e^{-2\gamma m_1} \sum_{i\in I} T(W,\varrho_i) \\
& \qquad \times \bigg( \gamma \int_{\RR^{n-m}} \big( e^{\gamma f(x,\varrho_i)} \big( s_1^2 - s_1 g(x,\varrho_i) + \frac{1}{4} g(x,\varrho_i) g(-x,\varrho_i) \big) 
- s_1^2 \big) \, \sL^{n-m}(\dint x) \\
& \qquad \qquad + \frac{1}{4} \EE \bigg[ \mathbf{1}\{ \Theta = \varrho_i \} \int_{\partial \Xi} \int_{\partial \Xi} e^{\gamma f(y-z,\varrho_i)} \, \sH^{n-m-1}(\dint y) \, \sH^{n-m-1}(\dint z) \bigg] \\
& \qquad \qquad + \frac{3}{4} \EE \bigg[ \mathbf{1}\{ V_{n-m}(\Xi)=0, \Theta = \varrho_i \} \int_{\partial \Xi} \int_{\partial \Xi} e^{\gamma f(y-z,\varrho_i)} \, \sH^{n-m-1}(\dint y) \, \sH^{n-m-1}(\dint z) \bigg] \bigg),
\end{align*}
which completes the proof.
\end{proof}

\subsection{Proof of Corollary \ref{cor:Covariance}}

By combining Theorem \ref{thm:asymptoticvariance} (as described in Remark \ref{rem:Covariance}) with \eqref{eq:19-07-a} and \eqref{eqn:ExpectationSurface} we obtain 
\begin{align*}
& \lim_{r\to\infty} r^{-(n+m)} {\rm cov}(V_n(Z\cap W_r),V_{n-1}(Z\cap W_r)) \\
&={-}\gamma\int_{\mathbb{M}_{n,m}} e^{-2\gamma m_1} V_{n-m}(X) \big(\gamma s_1 V_{n-m}(X) - V_{n-m-1}(X)\big)  \\
& \qquad \qquad \qquad \qquad \times T(W,\theta) \mathbf{1}\{\theta\notin\{\varrho_i:i\in I\}\} \,  \mathbb{Q}(\dint(\theta,X)) \allowdisplaybreaks \\
& \quad {-} \sum_{i\in I} \sum_{k=1}^\infty \frac{\gamma^k}{k!} \int_{\mathbb{M}_{n,m}^k}  \mathbf{1}\{\theta_1=\hdots=\theta_k=\varrho_i\} T(W,\varrho_i) \\
& \qquad \qquad \qquad \qquad \times \int_{(\mathbb{R}^{n-m})^{k-1}} e^{-2\gamma m_1} V_{n-m}(X_1\cap\bigcap_{j=2}^k (x_j+X_j)) \\
& \qquad \qquad \qquad \qquad \qquad \times \big(\gamma s_1 V_{n-m}(X_1\cap\bigcap_{j=2}^k (x_j+X_j)) - V_{n-m-1}(X_1\cap\bigcap_{j=2}^k (x_j+X_j))\big) \\
& \qquad \qquad \qquad \qquad \qquad \times (\sL^{n-m})^{k-1}(\dint(x_2,\hdots,x_k)) \, \mathbb{Q}^k(\dint((\theta_1,X_1),\hdots, (\theta_k,X_k) )) \\
& =: V_1 + V_2. 
\end{align*}
We have
$$
V_1 = {-} \gamma e^{-2\gamma m_1} \EE\Big[ V_{n-m}(\Xi) \big(\gamma s_1 V_{n-m}(\Xi) - V_{n-m-1}(\Xi)\big) T(W,\Theta) \mathbf{1}\{\Theta\notin\{\varrho_i:i\in I\}\} \Big] 
$$
and, by Lemma \ref{lem:HLS},
\begin{align*}
V_2 & = {-} e^{-2\gamma m_1}  \sum_{i\in I} ( \gamma s_1 \tau_{n-m,n-m}(\varrho_i) - \tau_{n-m,n-m-1}(\varrho_i) ) T(W,\varrho_i) \\
& = {-} \gamma e^{-2\gamma m_1}  \sum_{i\in I}  \int_{\RR^{n-m}} \Big(e^{\gamma f(x,\varrho_i)} \Big( s_1 - \frac{g(x,\varrho_i)}{2} \Big)  - s_1 \Big) \, \sL^{n-m}(\dint x) \,  T(W,\varrho_i),
\end{align*}
which completes the proof.\qed

\subsection{Proof of Proposition \ref{prop:CovMatrixIntVolPositiveDefinite}}

{Fix $ k\in \{m, \ldots, n\} $ and consider $ \varphi=V_k $.
	In view of the identity for $v(\varphi,W)$ in Theorem \ref{thm:asymptoticvariance} we start by applying the mean value formula \cite[Proposition 5.1]{ConcentrationIneq}. This first allows us to express $\EE[V_k(Z\cap\theta([0,1)^m+X))]$ for each $(\theta,X)\in\MM_{n,m}$ there as linear combination of the intrinsic volumes of $[0,1)^m+X$ of order $k$ to $n$ (we remark that this requires our assumption that $k\in\{m,\ldots,n\}$). It also shows that it is sufficient to prove that for $ a_k=1 $ and all $a_{k+1},\ldots,a_n\in\RR$,
	\begin{equation}\label{eq:22-11-1}
	\PP\Big(\sum\limits_{i=k}^na_iV_i([0,1)^m+\Xi)\neq 0\Big)>0.
	\end{equation}
	Recall that $\partial^+[0,1]^m:=[0,1]^m\setminus[0,1)^m$ and observe that by \cite[Lemma 14.2.1]{SW},
	\begin{align*}
	V_i([0,1)^m+\Xi) &= V_i([0,1]^m+\Xi) - V_i(\partial^+[0,1]^m+\Xi)\\
	&= \sum_{\ell=0}^iV_\ell([0,1]^m)V_{i-\ell}(\Xi) - \sum_{\ell=0}^iV_\ell(\partial^+[0,1]^m)V_{i-\ell}(\Xi)\\
	&= \sum_{\ell=0}^i c_{\ell,m}\,V_{i-\ell}(\Xi)
	\end{align*}
	with coefficients $c_{\ell,m}:=V_\ell([0,1]^m)-V_\ell(\partial^+[0,1]^m)\in\RR$ depending only on $\ell$ and $m$ and which satisfy $c_{\ell,m}=0$ whenever $\ell>m$. Hence, putting $ a_0=a_1=\ldots =a_{k-1}:=0 $, we have
	\begin{align*}
	\sum\limits_{i=0}^na_iV_i([0,1)^m+\Xi) &= \sum\limits_{i=0}^na_i\Big(\sum_{\ell=0}^i c_{\ell,m}\,V_{i-\ell}(\Xi)\Big) = \sum\limits_{i=0}^na_i\Big(\sum_{\ell=0}^i c_{i-\ell,m}\,V_{\ell}(\Xi)\Big) \\
	&= \sum_{\ell=0}^nV_\ell(\Xi) \Big(\sum\limits_{i=\ell}^na_i c_{i-\ell,m}\Big),
	\end{align*}
	where we applied an index shift in the second step and in the last step we changed the order of summation. Shifting the index once again, using that $V_\ell(\Xi)=0$ if $\ell>n-m$ and that $c_{\ell,m}=0$ for $\ell>m$, we see that the last expression is equal to
	\begin{align*}
	\sum_{\ell=0}^{n-m}V_\ell(\Xi) \Big(\sum\limits_{i=0}^{n-\ell}a_{i+\ell} c_{i,m}\Big) = 	\sum_{\ell=0}^{n-m}V_\ell(\Xi) \Big(\sum\limits_{i=0}^{m}a_{i+\ell} c_{i,m}\Big).
	\end{align*}
	Thus,
	$$
	\sum\limits_{i=0}^na_iV_i([0,1)^m+\Xi) = \sum_{\ell=0}^{n-m}b_\ell V_\ell(\Xi)
	$$
	for some $b_0,\ldots,b_{n-m}\in\RR$.
	Note that since $  k \in \{m,\ldots ,n\} $ we have $ b_{k-m}=a_k c_{m,m}=1 $ since $ a_{i}=0 $ for $ i\leq k-1 $, implying that ${\bf b}:=(b_0, b_1, \ldots, b_{n-m}) \neq (0,\ldots, 0) $. Because of the assumed positive definiteness of the covariance matrix $ C $ we have ${\bf b} ^T C {\bf b} >0 $ and hence $ \var(\sum_{\ell=0}^{n-m}b_\ell V_\ell(\Xi)) >0$. From this it follows that
		$$
	\PP\Big(\sum_{\ell=0}^{{n-m}} b_\ell V_\ell(\Xi)\neq 0\Big)>0,
	$$
	which in turn completes the argument.}\qed

\subsection*{Acknowledgement}

We are grateful to Claudia Redenbach (Kaiserslautern) for simulating Poisson cylinder processes for us that led to the pictures shown in Figure \ref{fig:cylinders}. CB was supported by the German Academic Exchange Service (DAAD) via grant 57468851, and CB and CT have been supported by the DFG priority program SPP 2265 \textit{Random Geometric Systems}.

\addcontentsline{toc}{section}{References}


\end{document}